\documentclass[a4paper,11pt]{report}
\usepackage[centertags]{amsmath}
\usepackage{amsfonts, fmultico, amssymb,amsthm,newlfont,amscd,makeidx,titlesec,sectsty}

\usepackage[pagebackref=false]{hyperref}

\usepackage{fancyhdr}

\def\baselinestretch{1.2}

\newlength{\defbaselineskip}
\titlespacing*{\chapter}{0pt}{50pt}{30pt}
\titleformat{\chapter}[display]{\normalfont\huge\bfseries}{\chaptertitlename\ \thechapter}{20pt}{\Huge}

\theoremstyle{plain}
\newtheorem{thm}{Theorem}[section]
\newtheorem{cor}[thm]{Corollary}
\newtheorem{lem}[thm]{Lemma}
\newtheorem{prop}[thm]{Proposition}

\newtheorem{exam}[thm]{Example}
\newtheorem{ques}[thm]{Question}
\newtheorem{prob}[thm]{Problem}
\newtheorem{defn}[thm]{Definition}
\newtheorem{rem}[thm]{Remark}

\newtheorem{exno}[thm]{Example and Notation}
\newtheorem{nore}[thm]{Notation and Remark}

\numberwithin{equation}{section}

\newcommand{\rad}{\mbox{rad}\,}
\newcommand{\coker}{\mbox{Coker}\,}
\newcommand{\Hom}{\mbox{Hom}\,}
\newcommand{\Ext}{\mbox{Ext}\,}

\newcommand{\Spec}{\mbox{Spec}\,}
\newcommand{\Max}{\mbox{Max}\,}
\newcommand{\Ker}{\mbox{Ker}\,}
\newcommand{\Ass}{\mbox{Ass}\,}
\newcommand{\Assh}{\mbox{Assh}\,}
\newcommand{\Att}{\mbox{Att}\,}
\newcommand{\Supp}{\mbox{Supp}\,}

\newcommand{\depth}{\mbox{depth}\,}
\newcommand{\grade}{\mbox{grade}\,}
\renewcommand{\dim}{\mbox{dim}\,}
\renewcommand{\Im}{\mbox{Im}\,}
\newcommand{\im}{\mbox{Im}\,}
\newcommand{\cd}{\mbox{cd}\,}

\newcommand{\Min}{\mbox{Min}\,}

\newcommand{\h}{\mbox{ht}}

\renewcommand{\H}{\mbox{H}}
\newcommand{\mH}{\mathcal{H}}
\newcommand{\V}{\mbox{V}}
\newcommand{\Z}{\mathbb{Z}}

\newcommand{\N}{\mathbb{N}}
\newcommand{\lo}{\longrightarrow}

\newcommand{\fa}{\mathfrak{a}}
\newcommand{\fb}{\mathfrak{b}}

\newcommand{\fm}{\mathfrak{m}}
\newcommand{\fp}{\mathfrak{p}}
\newcommand{\fq}{\mathfrak{q}}
\newcommand{\fP}{\mathfrak{P}}
\newcommand{\fQ}{\mathfrak{Q}}
\newcommand{\fn}{\mathfrak{n}}

\newcommand{\fr}{\mathfrak{r}}
\newcommand{\nCM}{\mbox{non--CM}\,}
\newcommand{\LH}{Lichtenbaum-Hartshorne Theorem}
\renewcommand{\a}{\mathrm{a}}
\newcommand{\CM}{\mbox{CM}}
\newcommand{\lc}{\mbox{H}}
\newcommand{\C}{\mathcal{C}}

\makeindex

\begin{document}

\clearpage\pagenumbering{roman}
\def\baselinestretch{1.3}

\thispagestyle{empty}

\noindent
{\Huge \bf Cousin complexes \\  and  applications}

\vspace{1cm}

\noindent
{\large \bf Raheleh Jafari}

\noindent
\href{mailto:jafarirahele@gmail.com}{jafarirahele@gmail.com}

\vspace{3.5cm}
{\large
\noindent
Ph.D.  thesis


\noindent
Defence: June 19, 2011


\noindent
Thesis advisor: Mohammad-T. Dibaei, Tarbiat Moallem University

\noindent
Evaluating committee:   Abdoljavad Taherizadeh,

\vspace{-0.2cm}
\quad\hspace{3cm} \ Tarbiat Moallem University

\quad\hspace{3cm} \  Masoud Tousi,

\vspace{-0.2cm}
\quad\hspace{3cm} \ Shahid Beheshti University

\quad\hspace{3cm} \ Siamak Yassemi,

\vspace{-0.2cm}
\quad\hspace{3cm} \  University of Tehran

\quad\hspace{3cm} \ Hossein Zakeri,

\vspace{-0.2cm}
\quad\hspace{3cm} \ Tarbiat Moallem University

\vspace{5.5cm}

\begin{center}

Faculty of Mathematical Sciences and Computer

Tarbiat Moallem University

Tehran, Iran

2011

\end{center}
}

\newpage
\thispagestyle{empty}
\quad

\def\baselinestretch{1.2}
\chapter*{Acknowledgement}

\quad
I wish to thank, first and foremost,  professor \textbf{Mohammad~T.~Dibaei} who offered invaluable assistance, constant support and  tactful guidance as my supervisor. His insights, passion and  patience  let me to overcome difficulties and win chills during my study and research life.

I am grateful to  professors \textbf{H.~Zakeri}, \textbf{M.~Tousi}, \textbf{S.~Yassemi} and \textbf{A.~Taherizadeh}, for serving on my thesis committee,  their helpful comments and suggestions, and  for beneficial courses, seminars and discussions I had with them during my
 undergraduate years.

I wish to express my sincere thanks to professor \textbf{S.~Zarzuela~Armengou} for accepting me as an official
visiting guest of the university of Barcelona and his hospitality.  I
was delighted to interact him by giving me the chance  to involve in several academic activities.

As a great opportunity,  I had a short but fruitful discussion with  professor \textbf{R.~Y.~Sharp} about my thesis, during his visit to the university of Barcelona.~I~owe my gratitude to him for his comments and  conversations.


Last but not the least, I owe my special gratitude and thanks to \textbf{my parents},  encouraging me towards higher educations by supporting  spiritually and financially throughout my life and their dedication that provided the foundation for this work.

\vspace{1cm}

{\hfill Raheleh Jafari}

{\hfill June  2011}

\newpage
\thispagestyle{empty}
\quad

\chapter*{Declaration}
\def\baselinestretch{0.4}
 Chapters 2, 3 and 4 of this thesis are devoted to the original
 results.

\newpage
\thispagestyle{empty}
\quad


\chapter*{Abstract}

\quad In this thesis, the class of modules whose Cousin complexes
have finitely generated cohomologies are studied as a  subclass of
modules which have uniform local cohomological annihilators and
it is shown that  these two classes coincide  over local rings
with Cohen-Macaulay formal fibres. This point of view enables us
to  obtain some properties of
 modules with finite Cousin complexes and find some characterizations of them.

 In this connection
we discuss attached prime ideals of certain local cohomology
modules in terms of cohomologies of Cousin complexes.
 In continuation, we
study the top local cohomology modules with specified set of
attached primes.

Our approach to study Cousin complexes  leads  us to characterization of
generalized Cohen-Macaulay modules in terms of uniform annihilators
of  local cohomology. We use these results to study the
Cohen-Macaulay loci of modules and find two classes of rings  over
which the Cohen-Macaulay  locus of any finitely generated module is a Zariski--open subset of the spectrum of the ring.

\vspace{3cm}\noindent{\bf Key words and phrases.}  Cousin
complexes, uniform local cohomological annihilator, Cohen-Macaulay
locus, local cohomology, attached primes.

\vspace{0.5cm}
 \noindent{\bf 2010 Mathematics Subject Classification.
}{13D02; 13D45; 13C14; 13D07.}

\newpage
\thispagestyle{empty}
\quad


\tableofcontents

\newpage
\thispagestyle{empty}
\quad

\clearpage \pagenumbering{arabic}
\csname@twosidetrue\endcsname


\chapter*{Introduction}
\addcontentsline{toc}{chapter}{Introduction}

\markboth{}{Introduction \hspace{12.3cm}}

Many concepts in commutative algebra are inspired by algebraic
geometric objects. Of particular interest and effective tool in this
thesis, is the Cousin complex of a module which is algebraic
analogue of the Cousin complex introduced in 1963/64  by
A. Grothendieck and R. Hartshorne \cite[Chapter IV]{H1}. They used this
notion to prove a duality theorem for cohomology of quasi--coherent
sheaves, with respect to a proper morphism of locally noetherian
preschemes.

\index{Cousin complex, $\C_R(M)$}
 In 1969, R. Y. Sharp presented the commutative algebraic
analogue of the Cousin complex   (see section 1.3) and approved it
as a powerful tool  by characterizing Cohen-Macaulay and Gorenstein
rings in terms of Cousin complexes \cite{S1}. This concept is
developed in \cite{S9} by Sharp and is discussed by S. Goto and
K. Watanabe in the $\Bbb{Z}$--graded context \cite{GW}. In \cite{S1},
Sharp shows that a commutative noetherian ring $R$ is Cohen-Macaulay
if and only if the Cousin complex $\C_R(R)$ of $R$ is exact, which
is improved to modules by himself in  \cite{S4},
 while $R$ is Gorenstein if and only if $\C_R(R)$ provides the
minimal injective resolution of $R$. He  also introduced Gorenstein
modules and characterized them by using Cousin complexes in
\cite{S4}. \index{Cohen-Macaulay}

 From the Cousin complex definition
is apparent that it terms are very much look like non--finitely
generated, and despite of it, $R$ is Cohen-Macaulay if and only if
$\C_R(R)$ is exact, i.e. its cohomologies are zero and so finitely
generated. Now, one may ask what rings or modules admit finitely
generated Cousin complex cohomologies and what properties these
rings or modules have.

In 2001, M. T. Dibaei and M.~Tousi, while studying the structure of dualizing
complexes, \index{dualizing complex} found a class of modules whose
Cousin complexes have finitely generated cohomologies.  The theory
of dualizing complexes comes also from algebraic geometry which was
discussed firstly by Grothendieck and Hartshorne in 1963/64 and used
to prove their duality theorem \cite[Chapter V]{H1}. Afterwards
Sharp and a number of authors studied its commutative algebraic
analogue   and found it as a useful tool.

 For the rest of this section, $R$ is a commutative noetherian
 ring and $M$ is a finitely generated $R$--module.

\index{dualizing complex|textbf} A dualizing complex for  a ring $R$
is a bounded injective complex $I^\bullet$, where all cohomology
 modules $\H^i(I^\bullet)$ are finitely generated $R$--modules and
the natural map
$M\longrightarrow\Hom_R(\Hom_R(M,I^\bullet),I^\bullet)$ is quasi
isomorphism for any finitely generated $R$--module $M$. A dualizing
complex $I^\bullet$ is said to be fundamental whenever
$\oplus_{i\in\Bbb{Z}}I^i\cong\oplus_{\fp\in\Spec R}E(R/\fp)$, where
 $E(R/\fp)$ is the injective envelope of $R/\fp$ as $R$--module,
 i.e. each prime ideal of $R$ occurs in
exactly one term of $I^\bullet$ and exactly once \cite[1.1]{S7}. It
is known that a ring $R$ possesses a  dualizing complex if and only
if it possesses a fundamental dualizing complex (see
\cite[3.6]{Hall} and \cite[1.2]{S7}), which is unique up to
isomorphism of complexes and shifting (see \cite[4.5]{S8} and
\cite[4.2]{Hall}).

Now, a natural and interesting treatment is to determine this unique
complex. In 1998, Dibaei and Tousi described that if a local ring
$R$ which satisfies the condition $(S_2)$, possesses a dualizing
complex, then the fundamental dualizing complex for $R$ is
isomorphic to the Cousin complex of the canonical module of $R$ with
respect to the height filtration \index{height!-filtration}(which is
equal to the dimension filtration in this case)\cite[2.4]{DT1}. As
an application  they proved that if a local ring $R$ satisfies the
condition $(S_2)$ and has a canonical module $K$, then finiteness of
cohomologies of the Cousin complex of $K$ with respect to a certain
filtration is necessary and sufficient condition for $R$ to possess
a dualizing complex \cite[3.4]{DT1}. In 2001, they generalized their
structural property of dualizing complex of \cite{DT1} and showed
that the Cousin cohomologies of $M$ over a local ring $R$, are
finitely generated if  $R$  has a dualizing complex and $M$ is
equidimensional which satisfies the condition $(S_2)$, in
\cite{DT2}. In continuation of \cite{DT1} and \cite{DT2}, Dibaei
studied some properties of Cousin complexes through the dualizing
complexes in 2005, and proved the following result. \index{Cousin
complex, $\C_R(M)$} \index{Cousin complex, $\C_R(M)$!finite-}\index{formal
fibre}\index{equidimensional}\index{Cohen-Macaulay}

\vspace{0.3cm}\noindent{\bf Theorem 1.}\,\,{\cite[Theorem 2.1]{D}} 
 Assume
that all formal fibres of $R$ are Cohen-Macaulay and $M$ satisfies
$(S_2)$. If $\widehat{M}$ is equidimensional, then $\C_R(M)$ has
finitely generated cohomology modules.

These ideas have been pursued in algebraic geometry by J.~Lipman, S.~Nayak
and P.~Sastry in \cite{LNS}. Taking motivation from \cite{DT1} and
\cite{DT2}, Kawasaki studied Cousin complex of a module over a
noetherian ring and improved
 results, independently from \cite{D}, in \cite{K}. More precisely, he
proved the following results.

\index{universally catenary}

\index{Cousin complex, $\C_R(M)$!finite-}\index{formal
fibre}\index{equidimensional}\index{Cohen-Macaulay}

\vspace{0.3cm}\noindent{\bf Theorem 2.}\,\,
{\cite[Theorem 1.1]{K}}
 \, Assume that $M$ is   equidimensional  and

{(i)} $R$ is universally catenary,

{(ii)} all the formal fibers of all the localizations of $R$ are
Cohen-Macaulay,

{(iii)} the Cohen-Macaulay locus of each finitely generated
$R$--algebra is open, \index{Cohen-Macaulay!-locus, $\CM(M)$|textbf}

\noindent  Then all the cohomology modules of the Cousin complex of
$M$ are finitely generated and only finitely many of them are
non--zero.

\vspace{0.3cm} The  assumptions of the above result are also
necessary in a sense.

\vspace{0.3cm}\noindent{\bf Theorem 3.}\,\,
{\cite[Theorem 1.4]{K}}
\, Let $R$  be a  catenary ring. Then the following statements are
equivalent.
\begin{itemize}
\item[(i)] $R$ satisfies the conditions (i), (ii) and {(iii)} of Theorem 2.
\item[{(ii)}] for any finitely generated equidimensional $R$--module $M$, all
the cohomology modules of the Cousin complex of $M$ are finitely
generated and only finitely many of them are non--zero.\end{itemize}

\vspace{0.3cm} In special case when $R$ is local,  Kawasaki obtains
a more simple but interesting version of his result, Theorem 2, as
the following.

 \vspace{0.3cm}\noindent{\bf Theorem 4.}\,\,{\cite[Theorem 5.5]{K}}
\,Assume that $R$  is  a universally catenary  local ring and $M$ is
an equidimensional $R$--module. If all formal fibres of $R$ are
Cohen-Macaulay, then all the cohomology modules of $\C_R(M)$ are
finitely generated.

\vspace{0.3cm} Note that if a local ring $R$ is universally
catenary, then $\widehat{M}$ is equidimensional for each finitely
generated equidimensional $R$--module $M$.  In the proof of the
above theorem,
 the assumption that $R$  is universally catenary, is  used to show that
$\widehat{M}$ is equidimensional. So one may consider this theorem
as a generalization of Theorem~1.

After reviewing some well known results and basic concepts which we
need throughout the thesis in Chapter 1, we start our study on
Cousin complexes by discussing some useful techniques and essential
properties of cohomology modules of Cousin complexes in the first
section of Chapter 2.  As a consequence we remove the condition
$(S_2)$ from Theorem 1 and recover Theorem 4, in Corollary \ref{216}
and Proposition \ref{233}.

In all above results about finiteness of the Cousin complex of an
$R$--module $M$, there are some crucial common conditions on $R$ and
$M$:
\begin{enumerate}\item[(a)] $M$ is equidimensional;
\item[(b)] $R$ is universally catenary; \item[(c)] all formal fibres of $R$ are Cohen-Macaulay.\end{enumerate}

\index{Cousin complex, $\C_R(M)$!finite-}\index{formal
fibre}\index{equidimensional}\index{Cohen-Macaulay}\index{universally
catenary} When $R$ is a local ring, these conditions are sufficient
for finiteness of cohomology modules of $\C_R(M)$ by Theorem 4, and
conditions (b) and (c) are necessary for finiteness of cohomology
modules of Cousin complexes of all equidimensional $R$--modules by
Theorem 3. It is now natural to ask that which of these conditions
are satisfied if $\C_R(M)$ has finitely generated cohomology modules
for an $R$--module $M$.

In 2006, C.~Zhou studied the properties of noetherian rings containing
uniform local cohomological annihilators and showed that all such
rings are universally catenary and locally equidimensional\cite{Z}.
Recall that an element $x\in R$ is called a  uniform local
cohomological annihilator of $M$, if $x\in R\setminus
\cup_{\fp\in\Min M}\fp$ and for each maximal ideal $\fm$ of $R$,
$x\H_\fm^i(M)= 0$ for all $i< \dim M_\fm$.\index{uniform local
cohomological annihilator|textbf}

 We continue Chapter 2, improving  some results of Zhou for modules which have uniform local
 cohomological annihilators and find some characterizations of these modules in Section 2.2.
We show that if a finitely generated  $R$--module $M$ has a uniform
local cohomological annihilator, then $M$ is locally equidimensional
and $R/0:_RM$ is universally catenary in Proposition \ref{equ} and
Corollary \ref{cat}. We also obtain that the property that $M$ has a
uniform local cohomological annihilator is independent of the module
structure  and depends only on the support of $M$ in Corollary
\ref{c215}. Finally we investigate our main result in this section
by proving that if $\C_R(M)$ has finitely generated cohomology
modules, then $M$ has a uniform local cohomological annihilator in
Theorem \ref{229} and so $M$ is equidimensional and $R/0:_RM$ is
universally catenary.

Our approach  in studying Cousin complexes is also useful for
discussing about uniform local cohomological  annihilators and helps
us to recover some results in this context by a different and may be
simple method, for instance see Corollary \ref{Zx} and Proposition
\ref{Zs}.

\index{uniform local cohomological annihilator}

This point of view, also enables us to characterize modules with
finite Cousin cohomologies over a local ring $R$ with Cohen-Macaulay
 formal fibres. The last section of
Chapter 2 is devoted to some applications of our approach.  In
Theorem \ref{234}, we show that over these rings, $\C_R(M)$ has
finitely generated cohomologies if and only if $M$ has a uniform
local cohomological annihilator, if and only if $\widehat{M}$ is
equidimensional $\widehat{R}$--module. Our results about the
annihilators of cohomologies of Cousin complexes in Section 2.1,
lead us to present the height of an ideal of $R$ in terms of Cousin
complex in Theorem \ref{height}.

Another important subject which is strongly related to the uniform
annihilators of local cohomology, is the notion $\a(M)$ which is
defined for a finitely generated $R$--module $M$ over a local ring
$(R,\fm)$ as $\a(M)=\bigcap_{i<\dim M}(0:_R\H^i_\fm(M)).$ Note that
by definition, an $R$--module $M$ has a uniform local cohomological
annihilator if and only if $\a(M)\nsubseteq \cup_{\fp\in\Min M}\fp$.
On the other hand if $\fp\in\Min M$, then $\a(M)\subseteq \fp$ if
and only if $\fp\in\Att\H^i_\fm(M)$ for some $i<\dim M$ (see Lemma
\ref{3.6}). These facts motivate us to study the relations between
local cohomology modules and Cousin complexes.

It is well known that for a finitely generated $R$--module $M$ with finite dimension $d=\dim M$,
$\Att \H^{d}_\fm(M)=\Assh M$ (see Theorem \ref{G}), we start
Chapter 3 by discussing  $\Att \H^t_\fm(M)$ for certain $t$, in
particular
 $\Att \H^{d-1}_\fm(M)$, in terms of cohomologies of $\C_R(M)$.  As a consequence we find a
non--vanishing criterion of $\H^{d-1}_\fm(M)$ when $\C_R(M)$
has finitely generated cohomologies in Corollary
\ref{2.6}.\index{local cohomology!-functor}

The main object of Section 3.2 is the following question which is
raised by Dibaei and S.~Yassemi in \cite{DY3}. They investigate the set
$\Att \H^{d}_\fa(M)$ for a finitely generated $R$--module $M$
and an ideal $\fa$ of $R$ and show  that $\Att \H^{d}_\fa(M)\subseteq \Assh M$.  Now it is natural to ask,

 \vspace{0.3cm}\noindent{\bf Question
5.}\,\,{\cite[Question 2.9]{DY3}} For any subset $T$ of ${\Assh} M$,
is there an ideal $\fa$ of $R$ such that ${\Att} {\lc}^{d}_\fa(M)= T$? \vspace{0.3cm}

Theorem \ref{328} presents a positive answer to this question in the
case where $(R,\fm)$ is a complete local ring. In
\cite[Theorem 1.6]{DY2}, it is proved that if $(R,\fm)$ is a complete local ring,
then for any pair of ideals $\fa$ and $\fb$ of $R$, if $\Att
\lc^d_\fa(M)={\Att} {\lc}^d_\fb(M)$, then ${\lc}^d_\fa(M) \cong
{\lc}^d_\fb(M)$. As a consequence we show that the number of
non--isomorphic top local cohomology modules of $M$ with respect to
all ideals of $R$ is equal to $2^{|\Assh M|}$ in Corollary
\ref{329}. \index{attached prime, $\Att$}

In last section 3.3, we  use results of  sections 3.1 and 3.2 and
those of  Chapter 2  for studying the class of generalized
Cohen-Macaulay modules.  In Corollary \ref{3.9}, we find a new
characterization of generalized Cohen-Macaulay rings in terms of
uniform annihilators of local cohomologies. Our results in this
section are useful in the last chapter of thesis to study the
Cohen-Macaulay loci of modules.

The Cohen-Macaulay locus of $M$ is denoted by
\begin{center} $\CM(M):=\{\fp\in\Spec R: M_\fp$ is Cohen-Macaulay
as $R_\fp$--module$\}$.\end{center}

The topological property of Cohen-Macaulay loci of modules and
determining when it is a zariski--open subset of $\Spec R$ have been
studied by many authors. Grothendieck in \cite{G} states that
$\CM(M)$ is a Zariski--open subset of $\Spec R$ whenever $R$ is an
excellent ring and in \cite{H1}, Hartshorne shows that $\CM(R)$ is
open when $R$ possesses a dualizing complex.  In \cite{RS}, C.~Rotthaus
and L.~M.~$\c{S}$ega study the Cohen-Macaulay loci of graded modules over
a noetherian homogeneous graded ring
$R=\bigoplus_{i\in\mathbb{N}}R_i$ considered as $R_0$--modules.

 \index{Cohen-Macaulay!-locus, $\CM(M)$}
\index{non--Cohen-Macaulay locus, non-$\CM(M)$|textbf}

 Our aim in the first section of
Chapter 4, is to determine when $\CM(M)$ is a Zariski--open subset
of $\Spec R$. We find two classes of rings, over  which,  $\CM(M)$
is open for all  $R$--modules $M$. The first is the class of rings
whose formal fibres are Cohen--Macaulay (see Remark \ref{3.7}) and
the second is the class of catenary local rings $R$  with finite
non--$\CM(R)$, where  non--$\CM(M)=\Spec R\setminus\CM(M)$ (see
Corollary \ref{419}). Finally, we present examples to show that
these two classes of rings are significant in \ref{3.15} and
\ref{3.16}.

Inspired by the above results, we study rings whose formal fibres
are  Cohen-Macaulay in Section 4.2. One of our main results in this
section is Theorem \ref{B} which gives a characterization of a
finitely generated $R$--module which admits a uniform local
cohomological annihilator in terms of certain set of formal fibres
of $R$. In particular we show that, for a prime ideal $\fp$ of $R$,
$R/\fp$ is universally catenary and  the formal fibre of $R$ over
$\fp$ is Cohen-Macaulay if and only if $R/\fp$ has a uniform local
cohomological annihilator (see Theorem \ref{B} and Lemma \ref
{31.7Ma}).

\index{formal fibre}\index{a@$\a(M)$}
 Corollary \ref{C} is a good
summary of connections between uniform annihilators of local
cohomologies and Cousin complexes which shows that a local ring $R$ is
universally catenary and all of it's formal fibres are
Cohen-Macaulay if and only if $\C_R(R/\fp)$ has finitely generated
cohomologies for all $\fp\in\Spec R$, if and only if $R/\fp$ has a
uniform local cohomological annihilator for all $\fp\in\Spec R$.

Note that  for an  $R$--module $M$, non--$\CM(M)= \V(\mathrm{a}(M))$
whenever $\mathcal{C}_R(M)$ is finite (see Corollary \ref{F}). We
close this section with Theorem  which characterizes those
modules $M$ satisfying  non--$\CM(M)= \V(\mathrm{a}(M))$  without
assuming that the Cousin complex of $M$ to be finite, which implies
 also that when $\CM_R(M)$ is finite, then the formal fibres of $R$
over some certain prime ideals are Cohen-Macaulay (see Corollary
\ref{clast}).

Observe  that if $\C_R(M)$ has finitely generated cohomologies, then
$M$ is locally equidimensional and $R/0:_RM$ is universally catenary
by Corollary \ref{cor}. On the other hand if all formal fibres of a
universally catenary local ring $R$ are Cohen-Macaulay, then
$\C_R(M)$ has finitely generated cohomologies  for all
equidimensional $R$--module $M$ (Theorem 4).

These results strengthen our guess that over a local ring $R$, if
$\C_R(R)$ has finitely generated cohomologies, then all formal
fibres of $R$ are Cohen-Macaulay~(Section~4.4).

\vspace{0.3cm}
 Throughout  this thesis, all  known  definitions and
statements are quoted with a reference afterward and all others with
no references are supposed to be new, most of them have been
appeared in \cite{DJ1}, \cite{DJ2} and \cite{DJ3}.


\renewcommand{\chaptermark}[1]{\markboth{#1}{}}
\renewcommand{\sectionmark}[1]{\markright{#1}{}}


\pagestyle{fancy}
\fancyhf{}
\fancyhead[LE,RO]{\thepage}
\fancyhead[RE]{\textit{\nouppercase{\leftmark}}}
\fancyhead[LO]{\textit{\nouppercase{\rightmark}}}

\chapter{Preliminaries}
\def\baselinestretch{1.3}

\quad Throughout this thesis $R$ is a commutative, noetherian ring
with non--zero identity and $M$ is an $R$--module. In the case
$(R,\fm)$ is local, we use $\widehat{M}$ as the $\fm$-adic completion of $M$. The set of all prime ideals of $R$  is
denoted by $\Spec R$ and for an ideal $I$ of $R$, $\V(I)$ denotes
the set of all prime ideals of $R$ contain  $I$. The support of $M$,
denoted by $\Supp M$, is the set $\{\fp\in \Spec R: M_{\fp}\neq 0\}$
and the set of associated prime ideals of $M$, denoted by $\Ass M$,
is the set
$$\{\fp\in \Spec R: \fp= (0:_{R}x) \ \ \textrm{for some non-zero element}\ \  x\in M\}.$$
The set of minimal primes of $M$, is the set of minimal elements of
$\Supp M$ with respect to inclusion and is denoted by $\Min M$. The
Krull dimension of $M$, denoted by $\dim M$, is defined to be the
supremum of lengths of chains of prime ideals of $\Supp M$ if this
supremum exists, and $\infty$ otherwise. By convention, the zero
module has Krull dimension $-1$. For a prime ideal $\fp$ of $R$, the
$M$-–height of $\fp$, denoted by $\h_M\fp$ is defined as $\dim
M_\fp$. Note that $\dim M=\max\{\dim R/\fp : \fp\in\Supp M\}$. The
set of those prime ideals $\fp\in\Supp M$ with $\dim R/\fp=\dim M$,
is denoted by $\Assh M$. \index{height!m@$M$--height, $\h_M$|textbf}

 When $\dim M<\infty$,  $M$ is
\index{equidimensional|textbf} \emph{equidimensional} if $\Min
M=\Assh M$. We say that $M$ is \index{equidimensional!locally-|textbf} \index{locally equidimensional|textbf}
\emph{locally equidimensional} if $M_\fm$ is equidimensional for
every maximal ideal $\fm$ of $\Supp M$. An element $x\in R$ is said
to be \emph{$M$--regular} \index{m@$M$--regular|textbf} if $xm\neq
0$ for each non--zero element $m\in M$. A sequence $x_1,\ldots, x_n$
of elements of $R$ is called an $M$--regular sequence or simply  an
\emph{$M$--sequence} \index{m@$M$--sequence|textbf} if  $x_1$ is
$M$--regular, $x_i$ is $M/(x_1,\ldots,x_{i-1})M$--regular for all
$2\leq i\leq n$ and $M\neq (x_1,\ldots,x_n)M$. If $IM\neq M$ for an
ideal $I$ of $R$, then the length of a maximal $M$--sequence in $I$
is a well determined integer $\grade(I,M)$ and is called the
\index{grade|textbf} \emph{grade of $I$ on $M$}. We set
$\grade(I,M)=\infty$ if $IM=M$. If $(R, \fm)$ is local ring, then
the grade of $\fm$ on $M$ is called the  \index{depth|textbf}
\emph{depth of $M$ }and is denoted by $\depth M$.

\section{Secondary representation theory}


\quad The theory of secondary representation of a module is in a
certain sense dual to the more familiar theory of primary
decomposition and associated primes, and provides a very
satisfactory tool for studying artinian modules. In this section  we
recall the concepts and some main points of this theory  which we
need later on. One may find the developed theory by Macdonald in
\cite{Mac}, it has been also mentioned by  Matsumura in his book
\cite[Section 6, Appendix]{Ma}.

A non--zero $R$--module $S$ is said to be \index{secondary|textbf}
\textit{secondary} if its multiplication map by any element $r$ of
$R$ is either surjective or nilpotent. If $S$ is secondary, then
$\fp=~\rad(0:_R~S~)$ is a prime ideal, and $S$ is said to be
\index{secondary!$\fp$--secondary|textbf} \textit{$\fp$--secondary}.
A \textit{secondary representation} \index{secondary!
representation|textbf} of an $R$--module $S$ is an expression of $S$
as a finite sum of secondary submodules
$$S= S_1+ S_2+ \dots + S_n.$$
The representation is minimal if
               \begin{itemize}
                   \item[(a)] {the prime ideals $\fp_i= \rad(0:_RS_i)$ are all distinct, and}
                   \item[(b)] {none of the $S_i$ is redundant (i.e. $S_j\nsubseteq \underset{i\neq j}{\sum\limits^n_{i=1}}S_i$).}
               \end{itemize}

An $R$--module $S$ is called
\emph{representable}\index{representable|textbf} if it has a
secondary representation. Note that the sum of the empty family of
submodules is zero and  so  a zero $R$--module is representable. It
is easy to see that if $S$ has a secondary representation, then it
has a minimal one.

Consider a minimal secondary representation $S=S_1+\cdots+S_n$ of
$S$, where $S_i$ is $\fp_i$--secondary for all $1\leq i\leq n$. Then
the set $\{\fp_1,\ldots,\fp_n\}$ is independent of the choice of
minimal secondary representation of $S$, is called the set of
\emph{attached prime ideals of $S$} and is denoted by $\Att S$ or
$\Att_R S$. By the \emph{attached prime ideals} \index{attached
prime, $\Att$|textbf} or the \emph{attached primes} of $S$, we mean
exactly the members of $\Att S$.

Here we mention some essential properties and results about
secondary representation and attached prime ideals which are
interpreted by Macdonald in \cite{Mac}.

\begin{itemize}
\item \cite[5.2]{Mac} Any  artinian $R$--module has a secondary representation.
\item \cite[2.2]{Mac} Let $S$ be a representable $R$--module and let $\fp$ be a prime ideal of $R$.
Then $\fp\in \Att S$ if and only if there is a homomorphic image of
$S$ which has annihilator equal to $\fp$.
\item \cite[4.1]{Mac} {Let $0 \longrightarrow S' \longrightarrow S\lo S''\lo 0$ be an exact sequence of representable $R$--modules and $R$--homomorphisms. Then
       $$\Att S'' \subseteq \Att S\subseteq \Att S' \cup \Att S''.$$}
 \item \cite[2.6]{Mac} Let $S$ be a representable $R$--module and let $r$ be an element of $R$. Then
                                  \begin{itemize}
                                        \item[($\alpha$)] {$rS\neq S$ if and only if $r\in \displaystyle\bigcup_{\fp\in \Att S}\fp$.}
                                        \item[($\beta$)] $\rad(0:_RS)= \displaystyle\bigcap_{\fp\in \Att S}\fp.$
                                  \end{itemize}
    \end{itemize}

\vspace{0.5cm} The following result follows easily from the above
$(\beta)$.

\begin{cor}\label{111}
Assume that $S$ is a representable $R$--module.
\begin{itemize}
\item[\emph{(i)}] If $(R,\fm)$ is local and $S$ is artinian, then $S$
is finitely generated, and so of finite length, if and only if
$\emph{\Att} S\subseteq \{\fm\}$.
\item [\emph{(ii)}]  If  $\fp\in\emph{\Min} (R/0:_RS)$, then $\fp\in\emph{\Att} S$.
\item[\emph{(iii)}] All elements of $\emph{\Att} S$ contains $0:_RS$.
\end{itemize}
\end{cor}

Finally, recall the following result which enables us to change the
base ring in studying attached prime ideals.
\begin{thm}\emph{\cite[8.2.5]{BS}}\label{atc}
Let  $f : R\longrightarrow R'$ be a  homomorphism of noetherian
rings. Assume that the $R'$--module $S$ has a secondary
representation. Then $S$ has  a secondary representation as an
$R$--module by means of $f$, and   $$\emph{\Att}_R S=\{f^{-l}(\fp) :
\fp\in\emph{\Att}_{R'} S\}.$$
\end{thm}

\section{Local cohomology modules}

\quad In this section, we review some basic definitions and   known
results about the local cohomology modules. Our main references are
the book of  Brodmann and  Sharp \cite {BS} and the lecture of
Schenzel \cite{Sc}.

Given an ideal $\fa$ of $R$,  we consider the \emph{$\fa$--torsion
functor} \index{a@$\fa$--torsion!-functor|textbf}
over the category of $R$--modules, defined by
$$\Gamma_\fa(M) = \{x\in M: \fa^nx= 0 \ \textrm{for some} \ n\in
\mathbb{N}\},$$ for any $R$--module $M$.  It is easy to see that
$\Gamma_{\fa}(-)$ is an additive, covariant, $R$--linear, and left
exact functor on the category of $R$--modules and
$R$--homomorphisms. So it makes sense to consider the right derived
functors of $\Gamma_\fa(-)$. For each $i\geq 0$, the $i$th right
derived functor of $\Gamma_{\fa}(-)$ is denoted by
and is called
 the \textit{$i$th local cohomology functor}\index{local cohomology!-functor|textbf} with respect to the ideal $\fa$, so that $\Gamma_{\fa}(-)$ and $\H_{\fa}^0(-)$ are naturally
 equivalent. The module  is called the \index{local cohomology!-module|textbf}\emph{$i$th local cohomology  module of $M$ with respect to $\fa$}.
\\
An $R$--module $M$ is called
\textit{$\fa$--torsion-free}\index{a@$\fa$--torsion-free|textbf} if
$\Gamma_\fa(M)= 0$, while it is called  \textit{$\fa$--torsion}
\index{a@$\fa$--torsion!-module|textbf} if $\Gamma_\fa(M)= M$.

 The $\fa$--torsion functor can be expressed as
 $$\Gamma_\fa(M) = \bigcup_{n\geq 0}(0:_M\fa^n)\cong \underset{n}{\varinjlim}\Hom_R(R/{\fa^n},
 M);$$ so for all $i$, we have functorially  $$\H^i_\fa(M)\cong\underset{n}{\varinjlim}\Ext_R^i(R/{\fa^n},
 M).$$


The following basic properties of local cohomology modules will be
often used without further mention.


\begin{itemize}
\item \cite[Remark 1.2.3]{BS} $\H^r_\fa(M)=\H^r_{\rad(\fa)}(M)$.
\item \cite[Corollary 2.1.7(i)]{BS} If $M$ is $\fa$--torsion, then $\H_{\fa}^i(X)=0$ for all $i>0$.
\item \cite[Corollary 2.1.7(iii)]{BS} $\H_{\fa}^i(M)\cong \H_{\fa}^i(M/\Gamma_{\fa}(M))$, for all $i> 0$.
\item \cite[Exercise 2.1.9]{BS} If $M$ is $\fb$--torsion, then $\H_{\fa+\fb}^i(M)\cong \H_{\fa}^i(M)$ for all $i$.
\item \textbf{Independence theorem} \cite[4.2.1]{BS}.  Let
$R\longrightarrow S$ be a homomorphism of rings, $\fa$ an ideal of
$R$ and let $M$ be an $S$--module. Then $\H^i_{\fa
S}(M)\cong\H^i_\fa(M)$ for all $i$.
\item \textbf{Flat base change theorem} \cite[4.3.2]{BS}. Let $R\longrightarrow S$ be
a flat homomorphism of rings, $\fa$ an ideal of $R$ and let $M$ be
an $R$--module. Then $\H^i_{\fa
S}(M\otimes_RS)\cong\H^i_\fa(M)\otimes_RS$ for all $i$.
\item \textbf{Long exact sequence of local cohomology modules}.
Let $0\longrightarrow M_1\longrightarrow M_2\longrightarrow
M_3\longrightarrow 0$ be an exact sequence of $R$--modules. Then we
have the following long exact sequence.
$$\cdots
\longrightarrow\H^i_\fa(M_1)\longrightarrow\H^i_\fa(M_2)\longrightarrow\H^i_\fa(M_3)\longrightarrow\H^{i+1}_\fa(M_1)\longrightarrow\cdots$$
\end{itemize}

The  vanishing of local cohomology modules is an important problem
which there are many results concerning it. The following theorems
are of most famous results.

\begin{thm} \label {GrV}
\emph{(Grothendieck's vanishing theorem)\cite[6.1.2]{BS}}
\index{Grothendieck's vanishing theorem|textbf} Let $M$ be an
$R$--module and $\fa$ be an ideal of $R$. Then $\emph{\H}_\fa^i(M)=
0$ for all $i> \emph{\dim}M$.
\end{thm}

The following result identifies the grade \index{grade} of an ideal
in terms of local cohomology modules.
\begin{thm}\emph{\cite[Theorem 6.2.7]{BS}}\label{grade}
Let $M$ be a finitely generated $R$--module such that $\fa M\neq M$
for an ideal $\fa$ of $R$.  Then   the least integer $i$ for which
$\emph{\H}_\fa^i(M)\neq 0$ is precisely \emph{$\grade(\fa, M)$}.
\end{thm}

\begin{thm} \emph{(The Lichtenbaum-Hartshorne vanishing
theorem)\cite[Theorem 3.1]{H2}}\index{Lichtenbaum-Hartshorne
vanishing theorem|textbf} \label{LH} Assume that $(R,\fm)$ is a
local ring of dimension $n$ and $\fa$ is a proper ideal of $R$. Then
the following statements are equivalent:
\begin{enumerate}
\item[\emph{(i)}] $\emph{\H}^n_\fa(R)=0$.
\item[\emph{(ii)}] $\fa\widehat{R}+\mathfrak{P}$ is not $\widehat{\fm}$--primary for each  prime ideal
  $\mathfrak{P}\in\emph{\Assh} \widehat{R}$.
\end{enumerate}
\end{thm}

The following results illustrate the utility of  the secondary
representation theory as an effective tool for studying local
cohomology modules.


\begin{thm}  \emph{(Grothendieck's non--vanishing theorem|textbf)\cite[Theorem
2.2]{MS}}\index{Grothendieck's non--vanishing
theorem|textbf}\label{G} Assume that $(R, \fm)$ is a local ring and
let $M$ be a non-zero finitely generated  $R$--module. Then
\index{attached prime, $\Att$}
$\emph{\Att}\emph{\H}_\fm^{\emph{{\dim}}M}(M)=\emph{\Assh} M$. In
particular $\emph{\H}_\fm^{\emph{\dim} M}(M)\neq 0$.
\end{thm}

\begin{thm}\emph{\cite[Corollary 4]{DY1}}\index{attached prime, $\Att$}
Let $(R,\fm)$ be local, $\fa$  an ideal of $R$ and let $M$ be a
finitely generated $R$--module. Then
$$\emph{\Att \H}_\fa^{\emph{\dim} M}(M)=\{\fq\cap R : \fq\in\emph{\Assh}
\widehat{R}, \ \emph{\dim} \widehat{R}/(\fa\widehat{R}+\fq)=0\}.$$
\end{thm}

Localization in an important tool in commutative algebra. The
following result provides a useful property of attached primes of
local cohomology modules under localization with a condition on the
base ring.

\begin{thm}\emph{(Shifted localization principle)\cite[Theorem 3.7]{S6}}\label{slp}\index{shifted localization principle|textbf}\index{attached prime, $\Att$} Let
$(R,\fm)$ be a local ring which is a homomorphic image of a
Gorenstein local ring, $M$ a finitely generated $R$--module,
$\fp\in\emph{\Spec} R$ and $t=\emph{\dim} R/\fp$. Assume that
$\fq\in\emph{\Spec} R$  such that $\fq\subseteq \fp$. Then for all
$i\in\mathbb{Z}$, $\fq R_\fp\in\emph{\Att} \emph{\emph{\H}}^i_{\fp
R_\fp}(M_\fp)$ if and only if $\fq\in\emph{\Att}
\emph{\H}^{i+t}_\fm(M)$.
\end{thm}


If we remove the condition that $R$ is a homomorphic image of a
Gorenstein local ring, one have a weaker result.

\begin{thm}\emph{(Weak general shifted localization
principle)\cite[Theorem 4.8]{S6}}\label{wgs}\index{weak general
shifted localization principle|textbf} Let $(R,\fm)$ be a local
ring, $M$  a finitely generated $R$--module,  $\fp,
\fq\in\emph{\Spec} R$ such that $\fq\subseteq \fp$. If for $i\in\Z$,
$\fq R_\fp\in\emph{\Att} \emph{\H}^i_{\fp R_\fp}(M_\fp)$,  then
$\fq\in\emph{\Att \H}^{i+t}_\fm(M)$, where $t=\emph{\dim} R/\fp$.
\end{thm}

The local cohomology modules of a finitely generated $R$--module $M$
are  rarely  finitely generated. For instance, when $(R,\fm)$ is
local and $M$ is finitely generated $R$--module, the fact that
$\H^i_\fm(M)$ is artinian for all $i\geq 0$ together with Corollary
\ref{111}(i), implies  that $\H^i_\fm(M)$ is finitely generated if
and only if $\Att \H^i_\fm(M)\subseteq\{\fm\}$. Now, by
Grothendieck's non-vanishing theorem (\ref{G}), we obtain that
$\H^{\dim M}_\fm(M)$ is not finitely generated if $\dim M>0$. On the
other hand for all ideals $\fa$ of $R$,  $\H^0_\fa(M)$ is a
submodule of $M$  and so it is finitely generated. It is now
interesting to find the least integer $i$ for which $\H^i_\fa(M)$ is
not finitely generated.

\begin{defn}\emph{\cite[Definition 9.1.3]{BS}}
\emph{Let $M$ be a finitely generated $R$--module. We define the
\emph{finiteness dimension of $M$ relative to} an ideal $\fa$ of $R$
by
$$f_\fa(M)=\inf\{i\in\N : \emph{\H}^i_\fa(M) \mbox{ is not finitely
generated }\}.$$}
\end{defn}
 \index{finiteness dimension, $f_\fa(M)$|textbf}

The following result provides further motivation for the concept of
finiteness dimension.

\begin{prop}\emph{\cite[Proposition 9.1.2]{BS}} The following
statements are equivalent for a finitely generated $R$--module $M$,
an ideal $\fa$ of $R$  and an integer $t\in\Bbb{N}$.
\begin{enumerate}
\item[\emph{(i)}] $\emph{\H}^i_\fa(M)$ is finitely generated for all $i < t$.
\item[\emph{(ii)}] $\fa\subseteq\emph{\rad}(0:_R\emph{\H}^i_\fa(M))$  for all $i < t$.
\end{enumerate}
\end{prop}
In  light of the above result we have the following useful
statement.

\begin{thm} Let $M$ be a finitely generated $R$--module and $\fa$ an
ideal of $R$. Then
$$\begin{array}{ll} f_\fa(M)&=\inf\{i\in\N
: \emph{\H}^i_\fa(M) \mbox{ is not finitely generated }\}
\\ & = \inf\{i\in\N : \fa\nsubseteq
\emph{\rad}(0:_R\H^i_\fa(M))\}.\end{array}$$
\end{thm}

\vspace{1cm} \noindent Some times it is more useful and fascinated
to weaken the condition $\fa\nsubseteq \rad(0:_R\H^i_\fa(M))$, using
another ideal $\fb\subseteq \fa$.

\begin{defn}\emph{\cite[Definition 9.1.5]{BS}}
\emph{Let $M$ be a finitely generated $R$--module and let $\fa$ and
$\fb$ be ideals of $R$ such that $\fb\subseteq \fa$. We define
\emph{$\fb$--finiteness dimension}, $f^\fb_\fa(M)$, of $M$  relative
to $\fa$ by
$$f^\fb_\fa(M):=\inf\{i\in\N : \fb\nsubseteq \rad(0:_R\H^i_\fa(M))\}.$$}
\end{defn}
\index{finiteness dimension, $f_\fa(M)$!$\fb$--finiteness-, $f^\fb_\fa(M)$|textbf}


\begin{defn}\emph{\cite[Definition 9.2.2]{BS}}
\emph{Let $M$ be a finitely generated $R$--module. For a prime ideal
$\fp\in\Spec R\setminus \V(\fa)$, the \emph{$\fa$--adjusted depth of
$M$ at $\fp$}, denoted by $\mbox{adj}_\fa\depth M_\fp$, is defined
by
\index{depth!$\fa$--adjusted-|textbf}
$$\mbox{adj}_\fa\depth M_\fp :=\depth M_\fp+\h(\fa+\fp)/\fp.$$
Note that this is $\infty$ unless $\fp\in\Supp M$ and
$\fa+\fp\subset R$, and then it is positive integer.}

\emph{Let $\fb$ be another ideal of $R$ such that $\fb\subseteq
\fa$. We define\emph{ $\fb$--minimum $\fa$--adjusted depth} of $M$,
denoted by $\lambda^\fb_\fa(M)$, by
$$\begin{array}{ll}\lambda^\fb_\fa(M)&=\inf\{\mbox{adj}_\fa\depth M_\fp :
\fp\in\Spec R\setminus \V(\fb)\}\\
&=\inf\{\depth M_\fp+\h(\fa+\fp)/\fp : \fp\in\Spec R\setminus
\V(\fb)\}.\end{array}$$}
\end{defn}

\index{a@$\fa$--adjusted depth|textbf} \index{a@$\fa$--adjusted
depth!$\fb$--minimum-, $\lambda^\fb_\fa(M)$|textbf}

\index{depth!$\fb$--minimum $\fa$--adjusted-|textbf}

\index{b@$\fb$--minimum $\fa$--adjusted
depth, $\lambda^\fb_\fa(M)$|textbf}

\begin{thm}\emph{\cite[Theorem 9.3.5]{BS}}\label{935BS}
Let $\fa$ and $\fb$ be ideals of $R$ such that $\fb\subseteq \fa$,
and let $M$ be a finitely generated $R$--module. Then
$$f^\fb_\fa(M)\leq \lambda^\fb_\fa(M).$$
\end{thm}

Recall that a finitely generated  $R$--module $M$, when $(R,\fm)$ is
local,  is called \emph{Cohen-Macaulay}
\index{Cohen-Macaulay|textbf} if $\dim M=\depth M$. In general, when
$R$ is not local, $M$ is called Cohen-Macaulay if $M_\fp$ is
Cohen-Macaulay  for all $\fp\in\Supp M$. This class of modules are
characterized in terms of local cohomology as well.

\begin{rem}
\emph{Assume that $(R,\fm)$ is a local ring and $M$ is a finitely
generated $R$--module. Theorem \ref{grade} implies that $M$ is
Cohen-Macaulay if and only if for all $i<\dim M$.}
\end{rem}

The Cohen-Macaulay modules have very nice properties and are
interesting to study. As a generalization of this important class of
modules, one may consider those modules, their  local cohomologies
are finitely generated at indices less than dimension.
\begin{defn} \emph{A finitely generated  module $M$ over a local ring $(R, \fm)$  is called
a} {generalized Cohen-Macaulay} \emph{(g.CM) \index{generalized
Cohen-Macaulay, g.CM|textbf} \index{Cohen-Macaulay!generalized-, g.CM|textbf} module whenever $\fm^n\H_\fm^i(M)= 0$ for some
$n\in\mathbb{N}$ and all $i<\dim M$}. \emph{The module $M$ is
called}  quasi--Buchsbaum \index{quasi--Buchsbaum|textbf}
\emph{whenever $\fm\H_\fm^i(M)= 0$ for all $i<\dim M$.}
\end{defn}

\begin{thm}\label{gcm}
Assume that $(R,\fm)$ is a local ring and $M$ is a finitely
generated  $R$--module. Then the following statements hold true.
\begin{itemize}
\item [\emph{(i)}] \emph{\cite[9.5.7(i)]{BS}} If $M$ is g.CM, then $\emph{\dim} R/\fp=\emph{\dim} M$ for all $\fp\in \emph{\Ass}
M\setminus\{\fm\}$ and $M_\fq$ is a Cohen-Macaulay $R_\fq$--module
for all $\fq\in\emph{\Supp} M\setminus\{\fm\}$.
\item[\emph{(ii)}] \emph{\cite[9.5.7(ii)]{BS}} As a converse of \emph{(i)}, if \begin{itemize}\item $R$ is a homomorphic image of a
regular ring, \item $\emph{\dim} R/\fp=\emph{\dim} M$ for all
$\fp\in \emph{\Min} M$, and
\item $M_\fq$ is a Cohen-Macaulay $R_\fq$--module for all
$\fq\in\emph{\Supp} M\setminus\{\fm\}$, \end{itemize} Then $M$ is
g.CM.
\item[\emph{(iii)}] \emph{\cite[9.5.8(i)]{BS}} Assume that $M$ is g.CM  and $r\in\fm$ is a parameter element of $M$ (i.e. $\emph{\dim}
M/rM=\emph{\dim} M-1$). Then $r$ is a non--zero--divisor on
$M/\Gamma_\fm(M)$ and  $M/rM$ is g.CM.
\end{itemize}
\end{thm}

As another generalization of Cohen-Macaulay modules,  an $R$--module
$M$ \emph{satisfies the Serre's condition $(S_n)$} \index{Serre's
condition $(S_n)$|textbf} for some integer $n\geq 0$, if $\depth
M_\fp\geq\min\{n,\dim M_\fp\}$ for all $\fp\in\Supp M$. A module
satisfies  $(S_n)$ for all $n\geq 0$ just when it is a
Cohen-Macaulay module.

We close this section by recalling the canonical module over a
homomorphic image  of a Gorenstein local ring and some properties
which we will use later on.

Assume that $(R,\fm)$ is a homomorphic image of a Gorenstein local
ring $(S,\fn)$. We set $$K_M=\Ext^{\dim S-\dim M}_S(M,S),$$ where
$M$ is a finitely generated $R$--module, and call it the
\emph{canonical module of $M$} \index{canonical module|textbf} (see
\cite{Sc}).

There is a natural map $\tau_M:M\longrightarrow K_{K_M}$ by
\cite[Theorem 1.11]{Sc}, which has an important role in our later
discussing.

\begin{lem}\emph{\cite[Lemma 1.9]{Sc}}\label{1.9Sc} In the above situation, we have the following
statements.
\begin{itemize}
\item[\emph{(i)}] If, for all $\fp\in\Supp M$, $\emph{\dim} M_\fp+\emph{\dim} R/\fp=\emph{\dim} M$,
then $(K_M)_\fp\cong K_{M_\fp}$.
\item[\emph{(ii)}] $\emph{\Ass} K_M=\{\fp\in\emph{\Ass} M : \emph{\dim} R/\fp=\emph{\dim M}\}$,
so $\emph{\dim} M=\emph{\dim} K_M$.
\item[\emph{(iii)}] $K_M$ satisfies the condition $(S_2)$.
\end{itemize}
\end{lem}

\index{Serre's condition $(S_n)$}
\begin{thm}\emph{\cite[Theorem 1.14]{Sc}}\label{1.14Sc}
Let $M$ denote a finitely generated, equidimensional $R$--module
with $d=\emph{\dim} M$, where $R$ is a homomorphic image of a
Gorenstein ring. Then  the following statements are equivalent for
an integer $k\geq 1$.
\begin{itemize}
\item[\emph{(i)}] $M$ satisfies the condition $S_k$.
\item[\emph{(ii)}] The natural map $\tau_M:M\longrightarrow K_{K_M}$ is bijective (resp. injective for $k=1$) and $\emph{\H}^n_\fm(K_M) = 0$
 for all $d - k + 2\leq n < d$.
\end{itemize}
\end{thm}


\section{Cousin complexes}

\begin{defn}\emph{\cite[Definition 1.1]{S9}}
 \emph{A \emph{filtration} \index{filtration|textbf} of $\Spec R$ is a descending sequence
 $\mathcal{F}=(F_i)_{i\geq0}$ of subsets of $\Spec(R)$,
 $F_0\supseteq F_1\supseteq
 F_2\supseteq\cdots\supseteq F_i\supseteq\cdots,$
 with the property that, for each $i\in\Bbb{N}_0$, every  member of
 $\partial F_i=F_i\setminus F_{i+1}$ is a minimal member of
$F_i$ with respect to inclusion. We say the filtration $\mathcal{F}$
\emph{admits} $M$ if $\Supp M\subseteq
F_0$.}
\end{defn}

%

\begin{exno}\emph{\cite[Example 1.2]{S9}}
\emph{Let $M$ be an $R$--module. For each $i\geq 0$, set
$$H_i=\{\fp\in\Supp M \mid \h_M\fp\geq i\}.$$
 The sequence $(H_i)_{i\geq0}$ is a filtration of $\Spec R$ which admits $M$
 and
is called the \emph{height filtration} of $M$ and is denoted by
$\mathcal{H}(M)$.}\index{filtration!height-|textbf}
\index{height!-filtration|textbf}
\end{exno}

%

\begin{nore}\label{cousin}\index{Cousin complex, $\C_R(M)$|textbf}
\emph{Let $\mathcal{F}=(F_i)_{i\geq0}$ be a filtration of $\Spec R$
which admits an $R$--module $M$. An obvious modification of the
construction given in $\S2$ of \cite{S1} produces a complex
$$0\longrightarrow
M\overset{d_M^{-1}}{\longrightarrow}M^0\overset{d_M^0}{\longrightarrow}M^1\longrightarrow\cdots\longrightarrow
M^{i-1}
\overset{d_M^{i-1}}{\longrightarrow}M^{i}\longrightarrow\cdots,$$
denoted by $\C(\mathcal{F},M)$ and called \emph{Cousin complex for
\emph{$M$} with respect to} $\mathcal{F}$, such that
$M^0=\oplus_{\fp\in\partial F_0}M_\fp$ and
$M^i=\oplus_{\fp\in\partial F_i}(\coker d_M^{i-2})_\fp$ for all
$i>0$. The component of $d_M^{-1}(m)$, for $m\in M$ and
$\fp\in\partial F_0$, in $M_\fp$ is $m/1$. Note that for a prime
ideal $\fp$ of $R$, if $d_\fp:M\longrightarrow M_\fp$ denotes the
natural homomorphism given by $d_\fp(x)=x/1$ for $x\in M$, then for
an element $x\in M$ and $i\geq0$, $d_\fp(x)=0$ for all but a finite
number of prime ideals $\fp\in\partial F_i$ by \cite[2.2]{S1}.
Consequently, for all $i\geq 0$, there is an  $R$--homomorphism
$d_M^{i-1}:M^{i-1}\longrightarrow M^{i}$ for which, if $x\in
M^{i-1}$ and $\fq\in\partial F_i$, the component of $d_M^{i-1}(x)$
in $(\coker d_M^{i-2})_\fq$ is $\pi(x)/1$, where $\pi:
M^{i-1}\longrightarrow\coker d_M^{i-2}$ is the canonical
epimorphism.}

\emph{We will denote the Cousin complex for $M$ with respect to
$M$--height filtration, $\mathcal{H}(M)$, by $\C_R(M)$.
 We also use the notation
$$\mathcal{C}_R(M)': 0\longrightarrow
 M^0 \overset{d_M^0}{\longrightarrow}
M^1\overset{d_M^{1}}{\longrightarrow}
M^2\overset{d_M^{2}}{\longrightarrow}\cdots\overset{d_M^{i-1}}{\longrightarrow}
M^i\overset{d_M^{i}}{\longrightarrow}\cdots$$
 and for each
$i\geq -1$,
$$K^i:=\Ker d_M^i, \, D^i:=\Im d_M^{i-1} , \, \mH^i_M:=K^i/D^i.$$
 Then we have the following natural exact
sequences,}
\begin{equation}\label{e1}
0\longrightarrow M^{l-1}/K^{l-1}\longrightarrow M^l\longrightarrow
M^l/D^l\longrightarrow 0,
\end{equation}
\begin{equation}\label{e2}
0\longrightarrow \mH_M^{l-1}\longrightarrow
M^{l-1}/D^{l-1}\longrightarrow M^{l-1}/K^{l-1}\longrightarrow 0,
\end{equation}
\emph{for all $l\geq -1$. \\We call the Cousin complex $\C_R(M)$
\emph{finite}
 whenever each $\mH^i_M$ is finitely generated as $R$--module. }
\end{nore}\index{Cousin complex, $\C_R(M)$!finite-|textbf}


The following Lemma has an important role in our approach for
working with cohomologies of Cousin complexes. For the proof of the
first part we adopt the argument in \cite[Theorem]{S2}.

\begin{lem}\label{tl}
Let $M$ be an $R$-module. For any integer $k$ with $0\leq k <
\h_M\fa$, the following statements are true.
\begin{itemize}
\item[\emph{(a)}] $\emph{\H}_\fa^s(M^k)= 0$ for all integers $s\geq 0$.
\item[\emph{(b)}] $\emph{\Ext}_R^s(R/\fa, M^k)= 0$ for all integers $s\geq 0$.
\end{itemize}
\end{lem}
\begin{proof}
(a). Set $C_{k-1}:= \coker d^{k-2}= M^{k-1}/D^{k-1}$ so that $M^k=
\underset{\underset{\h_M\fp= k}{\fp\in\Supp M}}{\oplus}
(C_{k-1})_\fp$. For each $k< \h_M\fa$ and each $\fp\in\Supp M$ with
$\h_M\fp= k$, there exists an element $x\in \fa\setminus\fp$. Thus
the multiplication map $(C_{k-1})_\fp\overset{x}{\longrightarrow}
(C_{k-1})_\fp$ is an automorphism and so the multiplication map
$\H_\fa^s((C_{k-1})_\fp)\overset{x}{\longrightarrow}
\H_\fa^s((C_{k-1})_\fp)$ is also an automorphism for all integers
$s$. One may then conclude that $\H_\fa^s((C_{k-1})_\fp)= 0$. Now,
from additivity of local cohomology functors, it follows that
$\H_\fa^s(M^k)= 0$.\\

(b). Assume in general that $N$ is an $R$--module such that
$\H_\fa^s(N)= 0$ for all $s\geq 0$. We show, by induction on $i$,
$i\geq 0$, that $\Ext_R^i(R/\fa, N)= 0$. For $i= 0$, one has
$\Hom_R(R/\fa, N)= \Hom_R(R/\fa, \H_\fa^0(N))$ which is zero. Assume
that $i> 0$ and the claim is true for any such module $N$ and all
$j\leq i-1$. Choose $E$ to be an injective hull of $N$ and consider
the exact sequence $0\longrightarrow N\longrightarrow
E\longrightarrow N'\longrightarrow 0$, where $N'= E/N$. As
$\H_\fa^0(E)= 0$, it follows that  $\H_\fa^s(N')= 0$ for all $s\geq
0$. Thus $\Ext_R^{i-1}(R/\fa, N')= 0$, by our induction hypothesis.
As, by the above exact sequence $\Ext_R^{i-1}(R/\fa,
N')\cong\Ext_R^{i}(R/\fa, N)$, the result follows.
\end{proof}


We quote the following results as basic facts on Cousin complexes
from \cite{S1}, \cite{S4},  \cite{SSc} and \cite{D}.

\begin{lem}\emph{\cite[Lemma 1.2]{D}}\label{1.2D}
Let $\overline{R}:=R/0:_RM$, then there exists an isomorphism of
complexes $\C_R(M)\cong\C_{\overline{R}}(M)$.
\end{lem}

\begin{thm}\emph{\cite[Theorem 3.5]{S1}}\label{3.5S1}
Suppose that $S$ is a multiplicatively closed subset of $R$ and $M$
is an $R$--module. Then there is an isomorphism of complexes of
$S^{-1}R$--modules and $S^{-1}R$--chain map,
$$\Psi=\{\psi^n\}_{n\geq-1}:S^{-1}(\C_R(M))\longrightarrow
\C_{S^{-1}R}(S^{-1}M)$$ which is such that
$\psi^{-1}:S^{-1}M\longrightarrow S^{-1}M$ is the identity.
\end{thm}

%
The following result provide a characterization of Cohen-Macaulay
modules in terms of Cousin complexes\index{Cohen-Macaulay}
\begin{thm}
\emph{\cite[Theorem 2.4]{S4}}\label{2.4S4} Assume that $M$ is a
non--zero finitely generated $R$--module. Then $M$ is Cohen-Macaulay
if and only if the Cousin complex of $M$, $\C_R(M)$, is exact.
\end{thm}


The structure of Cousin complexes characterize the condition $(S_n)$
as well. \index{Serre's condition $(S_n)$}

\begin{thm}\emph{\cite[Example 4.4]{SSc}}\label{4.4SSc}
Assume that $M$ is a finitely generated $R$--module. Then $M$
satisfies the  condition $(S_n)$ if and only if $\C_R(M)$ is exact
at $i$th term for $i\leq n-2$.
\end{thm}


\begin{thm}\emph{\cite[2.7]{S1}}\label{2.7S1} Assume that $M$ is an
$R$--module. Then

\emph{(i)} $\emph{\Supp} \emph{\coker} d^{n-1}_M\subseteq H_n$, and

\emph{(ii)} $\emph{\Supp} \mH^n_M\subseteq H_{n+2}$.
\end{thm}


The following two lemmas may be considered as easy applications of
the above theorem, which provide useful  properties of Cousin
cohomologies.

\begin{lem}\label{139}
Assume that $M$ is a finitely generated $R$--module of finite
$\emph{\dim} M=d$ and that $\mathcal{C}_R(M)$ is finite, then
$$\bigcap_{i\geq -1}(0 :_R \mH_M^i)\not\subseteq
\bigcup_{\fp\in\emph{\Min} M}\fp.$$
\end{lem}
\begin{proof}
By Theorem \ref{2.7S1}, for all $i\geq -1$, we have   $$\V(0:_R
\mH_M^i)= \Supp \mH_M^i\subseteq \{\fp\in\Supp M : \dim M_\fp\geq
i+2\}.$$ So $0 :_R \mH^i\not\subseteq \cup_{\fp\in\Min M}\fp$. As
$\dim M=d$, we have $\mH_M^i=0$ for $i\geq d-1$, hence Prime Avoidance
Theorem implies that $\cap_{i\geq -1}(0 :_R \mH_M^i)\not\subseteq
\cup_{\fp\in\Min M}\fp$.
\end{proof}

\begin{lem}\label{1310}
Assume that $M$ is a finitely generated $R$--module with
$d=\emph{\dim} M$.

\emph{(i)}   $\emph{\Ass} M=\emph{\Min} M$, if and only if
$\mH^{-1}_M=0$.

\emph{(ii)} $\mH^{d-1}_M=\mH^d_M=0$.
\end{lem}

\begin{proof}
(i) It is obvious by definition of $(S_1)$ and Theorem \ref{4.4SSc}.

(ii)  It is clear by Theorem \ref{2.7S1}(ii).
\end{proof}

\chapter{Finite Cousin complexes }

\quad The aim of this chapter is  to study the class of modules
 whose Cousin complexes have finitely generated cohomology modules,  as a
subclass of modules which have uniform local cohomological
annihilators.  We  describe some essential properties of the
structure of Cousin complexes which are useful in the rest of the
thesis, in the first section. In the second section, we study the
theory of uniform annihilators of local cohomology and show that the
class of finite Cousin complex modules is a subset of this class,
and finally we use our approach to get some characterizations in the
last section.

\section{Cohomology modules of Cousin complexes}

\quad In this section we study the structure of cohomology modules
of Cousin complexes and develop some techniques which will be used
throughout the rest of the thesis.

Firstly, we discuss the relationship between Cousin complexes of
modules satisfying a short exact sequence. \index{Cousin complex, $\C_R(M)$}
\begin{lem}\label{211}
Assume that $(R, \fm)$ is a local ring.
\begin{itemize}
\item [\emph{(a)}] If $0\longrightarrow L\overset{f}{\longrightarrow}
M\overset{g}{\longrightarrow}N$ is an exact sequence of $R$--modules
with the property  that $\emph\h_M\fp\geq 2$ for all
$\fp\in\emph\Supp N$, then $\mathcal{C}_R(L)'\cong\mathcal{C}_R(M)'
$; in particular, if $L$ and $M$ are finitely generated $R$--modules, then $\mathcal{C}_R(L)$ is finite if and only if
$\mathcal{C}_R(M)$ is finite.
\item[\emph{(b)}] If  $ L\overset{f}{\longrightarrow}
M\overset{g}{\longrightarrow}N\longrightarrow 0$ is an exact
sequence of $R$--modules with the property that $\emph\h_M\fp\geq 1$
for all $\fp\in\emph\Supp L$, then
$\mathcal{C}_R(M)'\cong\mathcal{C}_R(N)'$; in particular,
if $M$ and $N$ are finitely generated $R$--modules,
$\mathcal{C}_R(M)$ is finite if and only if $\mathcal{C}_R(N)$ is
finite. \end{itemize}\end{lem}
\begin{proof}
 To prove (a), we argue on the following diagram.
 \begin{equation}\label{e211}
\begin{CD}
&&&&&0\\
&&&&& @VVV\\
&0&&0&&\Ker \psi \\
& @VVV  @VVV @VVV\\
&L @>{d_L^{-1}}>>  L^0 @>\theta>> L^0/\im d_L^{-1} @>>> 0\\
& @VVfV @VV{f^0}V @VV\psi V\\
&M @>{d_M^{-1}}>> M^0 @>\lambda >> M^0/\im d_M^{-1} @>>> 0\\
& @VVgV @VVV @VVV\\
& N @>>>0&&0\\
& @VVV\\
&0,
\end{CD}
\end{equation}
 As $\Supp \Im g\subseteq\Supp N$ we may replace $N$ by $\Im g$ and assume that
 $$0\longrightarrow L\overset{f}{\longrightarrow}
M\overset{g}{\longrightarrow}N\longrightarrow 0$$ is exact. Since
$\Min M\cap \Supp N=\emptyset$, $\Min M=\Min L$ and it follows  an isomorphism $f^0:
L^0\longrightarrow M^0$ over $f$, where $L^0= \underset{\fp\in\Min
L}{\oplus}L_\fp$ and $M^0= \underset{\fp\in\Min M}{\oplus}M_\fp$ and
$f^0 d_L^{-1}= d_M^{-1} f$.

We next consider the natural epimorphisms $$\theta
:L^0\longrightarrow L^0/\Im d_l^{-1}, \ \lambda :  M^0
\longrightarrow d_M^{-1}, \ \psi :L^0/\Im d_L^{-1} \longrightarrow
M^0/\Im d_M^{-1},$$ which  $\psi\theta= \lambda f^0$ and consider
the map $\varphi:=\theta (f^0)^{-1}d_M^{-1} :M\longrightarrow \Ker
\psi$, where $(f^0)^{-1}$ denotes the inverse map of $f^0$, and show
that $\varphi$ is an epimorphism. In order to prove this, choose an
element $x\in\Ker \psi$, there is an element $m_0\in M^0$ such that
$x= \theta (f^0)^{-1}(m_0)$. It follows that $m_0\in\Ker
\lambda\subseteq\im d_M^{-1}$. Thus $m_0=d_M^{-1}(m)$ for some $m\in
M$. Hence $x= \theta (f^0)^{-1}d_M^{-1}(m)= \varphi (m)$. As $\im
f\subseteq \Ker\varphi$, there is an epimorphism \vspace{-.2cm}
$$(N\cong)M/\im f\longrightarrow M/\Ker \varphi (\cong
\Ker\psi)$$ which implies that $\Supp \Ker \psi\subseteq\Supp N$. As
a result, for each $\fp\in\Supp M$ with $\h_M\fp= 1$ we have
$\fp\not\in\Supp \Ker \psi$.

Now, summing up the localizations of the exact sequence
$$o\longrightarrow \Ker\psi\longrightarrow L^0/\im
d_L^{-1}\longrightarrow M^0/\im d_M^{-1}\longrightarrow 0$$ at all
prime ideals $\fp\in\Supp M$ with $\h_M\fp= 1$, and taking the
natural maps into consideration, we find the following commutative
diagram
$$\begin{CD}
&  L @>{d_L^{-1}}>>  L^0 @>{d_L^{0}}>> L^1 @>\theta^1>>L^1/\im d_L^0 @>>> 0\\
& @VVfV @VV{f^0}V @VVf^1 V @VV\psi^1V  \\
&M @>{d_M^{-1}}>> M^0 @>{d_M^0}>> M^1 @>\lambda^1>>M^1/\im d_M^0 @>>> 0,\\
\end{CD}$$
 with $f^1$ is an isomorphism, $\theta^1$ and $\lambda^1$ are the
 natural epimorphisms and $\psi^1$ is the natural homomorphism. It is
 clear that $\psi^1$ is an isomorphism too. Now, by induction on $i$,
 we get a family of isomorphisms $(f^i)_{i\geq 0}$, $f^i: L^i\longrightarrow M^i, i>
 0$, such that the diagram
$$\begin{CD}
0 && 0 &&0 & & && 0&& 0 \\
@VVV  @VVV  @VVV & & @VVV  @VVV \\
L @>>> L^0 @>{d_L^{0}}>>  L^1 @>{d_L^1}>>
\cdots  @>>> L^{i-1} @>{d_L^{i-1}}>> L^i @>{d_L^i}>>  \cdots \\
@VVfV @VVf^0V @VV{f^1}V & & @VVf^{i-1} V  @VVf^iV \\
M @>>>M^0 @>{d_M^{0}}>>  M^1 @>{d_M^{1}}>> \cdots
@>>> M^{i-1} @>{d_M^{i-1}}>> M^i @>{d_M^{i}}>>  \cdots\\
@VVV  @VVV  @VVV & & @VVV  @VVV \\
N @>>> 0 &&0 & & && 0&& 0 \\
@VVV\\
0,\\
\end{CD}$$
 is commutative with exact rows. Now, it follows that
there are an exact sequence $N\longrightarrow \mH^0_L\longrightarrow
\mH^0_M\longrightarrow 0$ and isomorphisms $\mH_L^i\cong \mH_M^i$, $i>
0$. Thus the claim follows.

 (b) As $\Supp \Ker f\subseteq \Supp L$ we may replace $L$ by $\Ker f$ and assume
 that
  $$0 \longrightarrow L\overset{f}{\longrightarrow}
M\overset{g}{\longrightarrow}N\longrightarrow 0$$ is exact. Since
$\Min M\cap \Supp L= \emptyset$, we get
 $\Min M= \Min N$ and there is an isomorphism
 $g^0: M^0\longrightarrow N^0$ over $g$, where $M^0=
 \underset{\fp\in\Min M}{\oplus}M_\fp$ and $N^0=
 \underset{\fp\in\Min N}{\oplus}N_\fp$. So that we have the commutative
 diagram
 \begin{equation}
\begin{CD}
&&0\\
&&@VVV\\
&& L @>>>0\\
&& @VVV @VVV\\
&&M @>d_M^{-1}>> M^0 @>\gamma_M^0>> M^0/\im d_M^{-1} @>>> 0\\
&& @VVgV @VVg^0V @VV \lambda^0 V\\
&& N @>d_N^{-1}>> N^0 @>\gamma_N^0>> N^0/\im d_N^{-1} @>>> 0\\
&&@VVV @VVV\\
&&0&&0,
\end{CD}
\end{equation}
where $\gamma_M^0$, $\gamma_N^0$ and  $\lambda^0$ are natural
homomorphisms. It is easy to see that $\lambda^0$ is an isomorphism.
Therefore, there is an isomorphism $g^1: M^1\longrightarrow N^1$
over $g^0$, where $M^1=\underset{\underset{\h_M\fp=1}{\fp\in\Supp M
}}{\oplus}(M^0/\im d_M^{-1})_\fp$ and
$N^1=\underset{\underset{\h_N\fp=1}{\fp\in\Supp N }}{\oplus}(N^0/\im
d_N^{-1})_\fp$. Hence we have the following commutative diagram
where $g^0$ and $g^1$ are isomorphisms.

\begin{equation}\label{e213}
\begin{CD}
0 @>>> M^0 @>d_M^0>> M^1 \\
&& @V\cong Vg^0V @V\cong Vg^1V \\
0 @>>> N^0 @>d_N^0>> N^1 \\
\end{CD}\end{equation}
\\
Now, it follows by induction that
$\mathcal{C}_R(M)'\cong\mathcal{C}_R(N)'$.

To prove the final claim, let us assume that $\mathcal{C}_R(M)$ is
finite. It follows from the exact sequence $$0\longrightarrow \im
d_M^{-1}\longrightarrow \Ker d_M^0\longrightarrow
\mH^0_M\longrightarrow 0$$ that $\Ker d_M^0$ is finitely generated.
By (\ref{e213}), $\Ker d_N^0$ is finitely generated and so $\mH^0_N$
is finitely generated. The finiteness of $\mH_N^i$, $i\geq 1$,
follows by the isomorphism
$\mathcal{C}_R(M)'\cong\mathcal{C}_R(N)'$.
\end{proof}


As an application  of the above lemma, we prove the following result
which will be useful in our later methods.

\begin{prop}\label{s1}
Assume that $(R,\fm)$ is local and  $M$ is a finitely generated
$R$--module. Then there is a finitely generated  $R$--module $N$
which satisfies the condition $(S_1)$ with $\emph{\Supp}
N=\emph{\Supp} M$ and $\C_R(M)$ is finite if and only if $\C_R(N)$ is finite.
\end{prop}
\begin{proof}
There exists a submodule $L$ of $M$ such that $\Ass L =\Ass
M\setminus\Min M$, $\Ass M/L=\Min M$ (c.f. \cite[Page 263,
Proposition 4]{B}). Set $N:=M/L$. Now $N$ satisfies $(S_1)$ and
considering the exact sequence $0\longrightarrow L\longrightarrow M
\longrightarrow N\longrightarrow 0$ and the fact that $\Supp L\cap
\Min M=\emptyset$,  the result follows by Lemma \ref{211}(b).
\end{proof}

\begin{cor}\label{cs1}
Assume that $(R,\fm)$ is local and $M$ is a finitely generated
$R$--module such that $\C_R(M)$ is finite. Then there exists a
finitely generated $R$--module $N$ which satisfies $(S_1)$,
$\emph{\Supp} N=\emph{\Supp} M$ and $\C_R(N)$ is finite.
\end{cor}


Consider the assumption and notation of Proposition \ref{s1}. In the
following result we find  $N$ satisfies  the condition $(S_2)$
whenever  $R$ is a homomorphic image of a Gorenestein ring. This
result will be useful  to find some sufficient conditions for
finiteness of Cousin complexes.


\begin{prop}\label{p214}
Assume that $(R,\fm)$ is a homomorphic image of a Gorenestein ring and $M$ is a finitely generated
equidimensional $R$--module. Then there exists a finitely generated $R$--module $N$ which satisfies the condition
$(S_2)$, $\Supp N=\Supp M$ and $\C_R(M)$ is finite if and only if $\C_R(N)$ is finite.
\end{prop}
\begin{proof}
By Proposition \ref{cs1}, we may assume that $M$ satisfies the condition $(S_1)$, i.e. $\Ass M=\Min M$, so that
$\Ass M=\Assh M$. Since $R$ is a homomorphic image of a Gorenestein ring, we may define the canonical module $K_M$.
Set $N:=K_{K_M}$. Then $K_M$ and $N$ satisfy the condition $(S_2)$ by Lemma \ref{1.9Sc}(iii). Now, by Theorem
\ref{1.14Sc}, we may consider the exact sequence
$$0\longrightarrow M\longrightarrow N\longrightarrow L\longrightarrow0.$$
Note that if $\fp\in\Supp M$ such that $\h_M\fp\leq 1$, then $M_\fp$ is a Cohen-Macaulay module. Thus Theorem \ref{1.14Sc}
implies that $L_\fp=0$.
By lemma \ref{1.9Sc}, $\Ass N=\Assh K_M=\Assh M=\Ass M$ so that $\Supp N=\Supp M$ and we have $\h_N\fp\geq2$ for all $\fp\in\Supp L$. Now Lemma \ref{211}(a) implies the result.
\end{proof}

Recall that, for a local ring $(R,\fm)$ and the natural ring
homomorphism $R\longrightarrow \widehat{R}$ and  any prime ideal
$\fp\in\Spec R$, the ring $\widehat{R}\otimes_R k(\fp)$ is called
the  \emph{formal fibre}\index{formal fibre|textbf} of $R$ over
$\fp$, where $k(\fp)=R_\fp/\fp R_\fp$. Note that all formal fibres
of a homomorphic image of a Gorenestein ring are Cohen-Macaulay.

The following result of H.~Petzl, helps us to use Proposition
\ref{p214} more efficiently.

\begin{lem}\emph{\cite[Theorem 3.5]{P}}\label{Pet}\index{formal
fibre}
 Assume  that $(R,\fm)$ is a local ring and all formal fibres of $R$  satisfy the condition $(S_1)$.
 Let $M$ be a finitely generated $R$--module. Then
there is a monomorphism of complexes
$u^\bullet:~\C_R(M)\otimes_R\widehat{R}\longrightarrow\C_{\widehat{R}}(\widehat{M}).$
More over if all formal fibres of $R$ are Cohen-Macaulay,  the
quotient complex $Q^\bullet$, in the exact sequence
$$0\longrightarrow\C_R(M)\otimes_R\widehat{R}\overset{u^\bullet}{\longrightarrow}\C_{\widehat{R}}(\widehat{M})\longrightarrow
Q^\bullet\longrightarrow0,$$ is an exact complex. In particular for
each $i\geq 0$, there exists an $\widehat{R}$--isomorphism
$$\mH^i_M\otimes_R\widehat{R}\cong\mH^i_{\widehat{M}} .$$
\end{lem}

In \cite[Theorem 2.1]{D}, Dibaei proves  that, over a local ring
with Cohen-Macaulay formal fibres, the Cousin complex of a finitely
generated  module $M$ is finite provided $M$ satisfies $(S_2)$ and
$\widehat{M}$ is equidimensional. We are now  able to show that the
condition $(S_2)$ is superfluous. Kawasaki also obtains this result
in the proof of \cite[Theorem 5.5]{K} by a different method.

\begin{cor}\label{216}\index{Cousin complex, $\C_R(M)$!finite-}
Assume that $(R, \fm)$ is a local ring such that all of its formal
fibers are Cohen-Macaulay. Assume that $M$ is a finitely generated
$R$--module such that $\widehat{M}$ is equidimensional
$\widehat{R}$--module. Then the Cousin complex of $M$,
$\mathcal{C}_R(M)$, is finite.\end{cor}

\begin{proof}
We may assume that $\dim M\geq 2$, by Lemma \ref{1310}(ii). Since
for each  $i$, by Lemma \ref{Pet}, there is an isomorphism
$\mH^{i}_M
   \otimes_{R}\widehat{R} \cong
   \mH^{i}_{\widehat{M}}$,  finiteness of
 $\mathcal{C}_R(M)$ is equivalent to
 finiteness  of $\mathcal{C}_{\widehat{R}}(\widehat{M})$.
Hence we may assume that $R$ is complete, and so that $R$ is a
homomorphic image  of a Gorenstein ring. Now, by Proposition
\ref{p214}, there exists a finitely generated $R$--module which
satisfies the condition $(S_2)$ and
$\C_R(M)$ is finite if and only if $\C_R(N)$ is finite.

Since $N$ satisfies the condition $(S_2)$ and $\Supp N=\Supp M$, so
that $N$ is equidimensional, \cite[Theorem 2.1]{D} implies that
$\C_R(N)$ is finite and so the result follows.
\end{proof}


 The following technical result will play a key role in the rest of the chapter.


\begin{prop}\label{p222}
Let $M$ be an $R$--module and let $\fa$ be an ideal of $R$ such that
$\fa M\not = M$. Then, for each non--negative integer $r$ with $r<
\emph{\h}_M\fa$, \emph{$$\prod_{i= 0}^r(0 :_R \Ext_R^{r-i}(R/\fa,
\mH_M^{i-1}))\subseteq 0 :_R \Ext_R^r(R/\fa, M).$$} Here $\prod$ is
used for product of ideals.
\end{prop}
\begin{proof}
 For each $j\geq -1$, recall the natural exact
sequences \ref{e1} and \ref{e2}.
\begin{equation} \label{e221}0\longrightarrow M^{j-1}/K^{j-1}\longrightarrow
M^j\longrightarrow M^j/D^j\longrightarrow 0,
\end{equation}
\begin{equation} \label{e222}0\longrightarrow \mH_M^{j-1}\longrightarrow
M^{j-1}/D^{j-1}\longrightarrow M^{j-1}/K^{j-1}\longrightarrow 0.
\end{equation}
Let $0\leq r< \h_M\fa$. We prove by induction on $j$, $0\leq j\leq
r$, that
\begin{equation}\label{e223} \prod_{i= 0}^j (0 :_R
\Ext_R^{r-i}(R/\fa, \mH_M^{i-1}))\cdot(0 :_R\Ext_R^{r-j}(R/\fa,
M^{j-1}/K^{j-1}))\subseteq 0 :_R \Ext_R^r(R/\fa, M). 
\end{equation} In case $j= 0$, the exact sequence (\ref{e222}) implies the
exact sequence
$$\Ext_R^r(R/\fa, \mH_M^{-1})\longrightarrow \Ext_R^r(R/\fa, M)\longrightarrow
\Ext_R^r(R/\fa, M^{-1}/K^{-1}),$$ so that $$(0 :_R \Ext_R^r(R/\fa,
\mH_M^{-1}))\cdot (0 :_R \Ext_R^r(R/\fa, M^{-1}/K^{-1}))\subseteq 0 :_R
\Ext_R^r(R/\fa, M),$$ and thus the case $j= 0$ is justified.

Assume that $0\leq j< r$ and formula (\ref{e223}) is settled for
$j$. Therefore, by Lemma \ref{tl}(b), the exact sequence \ref{e221}
implies that
\begin{equation}\label{e224}
\Ext_R^{r-j}(R/\fa, M^{j-1}/K^{j-1})\cong \Ext_R^{r-j-1}(R/\fa,
M^j/D^j). 
\end{equation}
On the other hand the exact sequence (\ref{e222}) implies the exact
sequence
$$
\Ext_R^{r-j-1}(R/\fa, \mH_M^j)\longrightarrow \Ext_R^{r-j-1}(R/\fa,
M^j/D^j)\longrightarrow \Ext_R^{r-j-1}(R/\fa, M^j/K^j),
$$
from which it follows that {\small\begin{equation}\label{e225} (0:_R
\Ext_R^{r-j-1}(R/\fa, \mH_M^j))\cdot(0 :_R \Ext_R^{r-j-1}(R/\fa,
{M^j}/{K^j}))\subseteq  \hfill 0 :_R \Ext_R^{r-j-1}(R/\fa,
{M^j}/{D^j}). 
\end{equation}}
Now (\ref{e224}) and (\ref{e225}) imply that {\small
\begin{equation}\label{e226} (0:_R \Ext_R^{r-j-1}(R/\fa,
\mH_M^j))\cdot(0 :_R \Ext_R^{r-j-1}(R/\fa, {M^j}/{K^j}))\subseteq 0 :_R
\Ext_R^{r-j}(R/\fa,
{M^{j-1}}/{K^{j-1}}). 
\end{equation}}
From (\ref{e226}), it follows that {\small $$ \prod_{i= 0}^{j+1}(0
:_R \Ext_R^{r-i}({R}/{\fa}, \mH_M^{i-1}))\cdot
(0:_R\Ext_R^{r-j-1}({R}/{\fa}, {M^j}/{K^j})) =$$
$${\prod}_{i= 0}^{j}(0 :_R \Ext_R^{r-i}({R}/{\fa}, \mH_M^{i-1}))\cdot (0
:_R \Ext_R^{r-j-1}({R}/{\fa}, \mH_M^{j}))\cdot(0 :_R
\Ext_R^{r-j-1}({R}/{\fa}, {M^j}/{K^j})) \subseteq
$$ $$\prod_{i= 0}^{j}(0 :_R \Ext_R^{r-i}({R}/{\fa}, \mH_M^{i-1}))\cdot (0
:_R \Ext_R^{r-j}({R}/{\fa}, M^{j-1}/K^{j-1})),$$}

 \noindent and, by the induction hypothesis (\ref{e223}), it follows that  $$\prod_{i= 0}^{j+1}
 (0 :_R \Ext_R^{r-i}(R/\fa, \mH_M^{i-1}))\cdot
(0:_R\Ext_R^{r-j-1}(R/\fa, M^j/K^j))\subseteq 0 :_R \Ext_R^r(R/\fa,
M).$$ \noindent  This is the end of the induction argument.

Note that  there is an embedding $\Ext_R^0(R/\fa,
M^{r-1}/K^{r-1})\hookrightarrow \Ext_R^0(R/\fa, M^r)$, by the exact
sequence (\ref{e221}). On the other hand $\Ext_R^0(R/\fa, M^r)= 0$
by Lemma \ref{tl}(b). So that $\Ext_R^0(R/\fa, M^{r-1}/K^{r-1})=0$
and now  putting  $j= r$ in (\ref{e223}) gives the result.
\end{proof}


An immediate corollary to the above result is the following.

\begin{cor}\label{c224}
 Assume that  $\fa$ is an
ideal of $R$ such that $\fa M\not= M$. Then, for each integer $r$
with $0\leq r< \emph{\h}_M\fa$, $$\prod_{i= -1}^{r-1}(0 :_R
\mH_M^i)\subseteq\bigcap_{i= 0}^r(0 :_R {\emph{\Ext}}_R^i(R/\fa,
M))\subseteq \bigcap_{i= 0}^r(0:_R\emph{\H}^i_\fa(M)).$$
\end{cor}
\begin{proof}
It follows by Proposition \ref{p222} and the fact that the extension
functors are linear and $\H_\fa^i(M)\cong
\underset{\underset{j}{\longrightarrow}}{\lim}(\Ext_R^i(R/\fa^j,
M))$.
\end{proof}


\section{Uniform annihilators of local cohomology}

\quad Recall that an element $x\in R\setminus\cup_{\fp\in\Min M}\fp$ is a \emph{uniform local
cohomological annihilator}  \index{uniform local cohomological
annihilator|textbf} of an $R$--module $M$ if, for every maximal
ideal $\fm$, $x\H^i_\fm(M)=0$ for all $i<\h_M\fm$. Moreover,  $x$ is
called a \emph{strong uniform local cohomological
annihilator}\index{uniform local cohomological annihilator!
strong-|textbf} \index{strong uniform local cohomological
annihilator|textbf} of $M$ if $x$ is a uniform local cohomological
annihilator of $M_\fp$ for every prime ideal $\fp$ in $\Supp M$.  $R$
is called \emph{universally catenary}\index{universally
catenary|textbf} if every finitely generated $R$--algebra is
catenary.

As a basic property of a ring $R$ containing a uniform local
cohomological annihilator, it is proved that $R$ must be locally
equidimensional \index{locally equidimensional}
\index{equidimensional!locally-} and universally catenary (c.f.
\cite[Theorem 2.1]{Z}).
 One of essential results about uniform annihilators of local cohomology, is the following result due to Zhou which reduces the
property that a ring $R$ has a uniform local cohomological
annihilator to the same property for $R/\fp$ for each minimal prime
$\fp$ of $R$.


\begin{thm}\emph{\cite[Theorem 3.2]{Z}}\label{z3.2}
Let $R$ be of finite dimension $d$. Then the following conditions
are equivalent.
\begin{enumerate}
\item[\emph{(i)}] $R$ has a uniform local cohomological annihilator.
\item[\emph{(ii)}] $R$ is locally equidimensional, and $R/\fp$ has a uniform local
cohomological annihilator for each minimal prime ideal $\fp$ of $R$.
\end{enumerate}
\end{thm}


The module version of the above theorem is also true. One may use an
almost similar method to prove it. We state the proof more
precisely. First  we show that  a module which has a uniform local
cohomological annihilator is locally equidimensional.
\begin{prop}\label{equ}\index{locally equidimensional}\index{equidimensional!locally-}
Let $M$ be a finitely generated $R$--module such that it has a
uniform local cohomological annihilator. Then $M$ is locally
equidimensional.
\end{prop}
\begin{proof}
Let $\fm\in\Max\Supp M$. We will show that $\dim M_\fm= \dim
R_\fm/\fp R_\fm$ for all $\fp\in\Spec R$ with $\fp\in\Min M$ and
$\fp\subseteq \fm$. By assumption, there exists an element $x\in
R\setminus\cup_{\fp\in\Min M} \fp$ such that $x\H_\fm^i(M)= 0$ for
all $i< \dim M_\fm$. As $x\in R_\fm\setminus\underset{\fp
R_\fm\in\Min M_\fm}{\cup} \fp R\fm$, and $\H_\fm^i(M)\cong \H_{\fm
R_\fm}^i(M_\fm)$ by using the definition of local cohomology, we may
assume that $(R, \fm)$ is local with the maximal ideal $\fm$ and
write $d:= \dim M$.

Assume, to the contrary, that there exists $\fp\in\Min M$ with $c:=
\dim R/\fp < d$. Set $S= \{\fq\in\Min M : \dim R/\fq\leq c\}$ and
$T= \Ass M\setminus S$. There exists a submodule $N$ of $M$ such
that $\Ass N= T$ and $\Ass M/N= S$. Note that $\dim M/N= c$ and that
$\dim N= d$. As $\rad(0:_R N)= \cap_{\fq\in T}\fq$, it follows that
there exists an element $y\in (0:_R N)\setminus\cup_{\fq\in S}\fq$.
Thus, trivially, $y\H_\fm^i(N)= 0$ for all $i\geq 0$. The exact
sequence $0\longrightarrow N\longrightarrow M\longrightarrow
M/N\longrightarrow 0$ implies the exact sequence
$$\H_\fm^i(M)\longrightarrow\H_\fm^i(M/N)\longrightarrow\H_\fm^{i+1}(N).$$
As $x\H_\fm^i(M)= 0$ for all $i< d$, it follows that
$xy\H_\fm^i(M/N)= 0$ for all $i< d$. In particular,
$xy\H_\fm^c(M/N)= 0$. Thus $xy\in\cap_{\fq\in\Assh M/N}\fq$ (c.f.
\cite[Proposition 7.2.11 and Theorem 7.3.2]{BS}). Therefore
$xy\in\fp$ by the choice of $\fp$. As $\fp\in S\cap\Min M$, this is
a contradiction.
\end{proof}

  We need the following technical lemma  to extend Theorem \ref{z3.2}
to module version. This result has been proved in \cite[Lemma
3.1]{Z} for $M=R$, the same technique works also for an arbitrary
$R$--module $M$.

\begin{lem}\label{l213}
Let $(R,\fm)$ be a local ring, $M$ a finitely generated $R$--module
of dimension $d$,  $\fp$ be a minimal prime ideal of $M$ and
$$\begin{array}{llllll}
&0\longrightarrow R/\fp\longrightarrow M\longrightarrow
N_1\longrightarrow 0,\\
&0\longrightarrow R/\fp\longrightarrow N_1\longrightarrow
N_2\longrightarrow 0,\\
&\vdots \\
&0\longrightarrow R/\fp\longrightarrow N_{t-2}\longrightarrow
N_{t-1}\longrightarrow 0,\\
&0\longrightarrow R/\fp\longrightarrow N_{t-1}\longrightarrow
N_t\longrightarrow 0.\end{array}$$
be a  series of short exact sequences of finitely generated
$R$--modules.
Let $y$ be an element of $R$ such
that $yN_t=0$.
\begin{enumerate}
\item[\emph{(i)}] If there is an element $x$ of $R$ such that $x\emph{\H}^i_\fm(M)=0$ for $i<d$, then $(xy)^{t^{d-1}}\emph{\H}^i_\fm(R/\fp)=0$ for $i<d$.
\item[\emph{(ii)}] If there is an element $x$ of $R$ such that $x\emph{\H}^i_\fm(R/\fp)=0$
for $i<d$, then $x^ty\emph{\H}^i_\fm(M)=0$ for $i<d$.
\end{enumerate}
\end{lem}


Now, we are able to present the module version of Theorem
\ref{z3.2}.

\begin{prop}\label{ulc}
Let $M$ be a finitely generated  $R$--module. Then the following
conditions are equivalent. \index{locally equidimensional}
\index{equidimensional!locally-}
\begin{enumerate}
\item[\emph{(i)}] $M$ has a uniform local cohomological annihilator.
\item[\emph{(ii)}] $M$ is locally equidimensional and $R/\fp$ has a
uniform local cohomological annihilator for all $\fp\in \emph\Min
M$.
\end{enumerate}
\end{prop}
\begin{proof}
(i) $\Rightarrow $ (ii).  By Proposition \ref{equ}, $M$ is locally
equidimensional.  Assume that $\fp\in\Min M$ and that $\fm$ is a
maximal ideal containing $\fp$. As $M_\fp$ is an $R_\fp$--module of
finite length, we set $t:= l_{R_\fp}(M_\fp)$. Then there exists a
chain of submodules $$N_0\subset N_1\subset N_2\subset\cdots\subset
N_t \subseteq M,$$  such that $N_0\cong N_{i}/N_{i-1}\cong R/\fp$, for $1\leq i\leq t$, and
$$\begin{array}{llllll}
&0\longrightarrow R/\fp\longrightarrow M\longrightarrow
M/N_0\longrightarrow 0,\\
&0\longrightarrow R/\fp\longrightarrow M/N_0\longrightarrow
M/N_1\longrightarrow 0,\\
&\vdots \\
&0\longrightarrow R/\fp\longrightarrow M/N_{t-2}\longrightarrow
M/N_{t-1}\longrightarrow 0,\\
&0\longrightarrow R/\fp\longrightarrow M/N_{t-1}\longrightarrow
M/N_t\longrightarrow 0,\end{array}$$
are exact sequences.
 Since $M_\fm$ is
equidimensional, $\h_M\fm/\fp= \h_M\fm$. As, by definition of $t$,
$l_{R_\fp}((M/N_t)_\fp)=0$,  it follows that $0:_R(M/N_t)\not\subseteq
\fp$. Choose an element
$y\in 0:_R(M/N_t)\setminus \fp$. Localizing the above exact sequences at $\fm$ implies the
following exact sequences.
 $$\begin{array}{llllll}
&0\longrightarrow (R/\fp)_\fm\longrightarrow M_\fm\longrightarrow
(M/N_0)_\fm\longrightarrow 0,\\
&0\longrightarrow (R/\fp)_\fm\longrightarrow
(M/N_0)_\fm\longrightarrow
(M/N_1)_\fm\longrightarrow 0,\\
&\vdots \\
&0\longrightarrow (R/\fp)_\fm\longrightarrow
(M/N_{t-2})_\fm\longrightarrow
(M/N_{t-1})_\fm\longrightarrow 0,\\
&0\longrightarrow (R/\fp)_\fm\longrightarrow
(M/N_{t-1})_\fm\longrightarrow (M/N_t)_\fm\longrightarrow 0.\end{array}$$  By assumption, there is an element
$x\in R\setminus\underset{\fq\in\Min M}{\cup}\fq$ such that
$x\H_{\fm R_\fm}^i(M_\fm)= 0$ for all $i< \h_M\fm$. Now, by Lemma
\ref{l213}, we have $(xy)^l\H_\fm^i(A/\fp)_\fm= 0$ for all $i<
\h_M\fm$ and for some integer $l> 0$.

 (ii) $\Rightarrow$ (i). One may use a similar method as above to get the result.
\end{proof}

The following result is an easy application of the above proposition
which shows that the property that a module $M$ has a uniform local
cohomological annihilator is independent of the module structure and
depends only on the support of $M$.

\begin{cor}\label{c215}
Let  $M$ and  $N$ be finitely generated $R$--modules of finite
dimensions such that $\emph\Supp M=\emph\Supp N$. Then $M$ has a
uniform local cohomological annihilator if and only if  $N$ has a
uniform local cohomological annihilator.
\end{cor}

\begin{cor}\label{cat}Let $M$ be a finitely generated  $R$--module  that has a uniform
local cohomological annihilator. Then $R/0:_RM$ is universally
catenary. \index{universally catenary}
\end{cor}
\begin{proof}
It follows by Corollary \ref{c215} and \cite[Theorem 2.1]{Z}.
\end{proof}

Recall  that for a finitely generated
$R$--module $M$ over a local ring $(R,\fm)$, $$\a(M)=\bigcap_{i<\dim
M}(0:_R\H^i_\fm(M)).$$ \index{a@$\a(M)$|textbf}Note that by
definition, an $R$--module $M$ has a uniform local cohomological
annihilator if and only if $\a(M)\nsubseteq \cup_{\fp\in\Min M}\fp$.
On the other hand if $\fp\in\Min M$, then $\a(M)\subseteq \fp$ if
and only if $\fp\in\Att\H^i_\fm(M)$ for some $i<\dim M$. More
precisely we have the following result.

\begin{lem}\label{3.6}
Assume that $(R, \fm)$ is local and that $M$ is a finitely generated
$R$--module of dimension $d$. Then $M$ has a uniform local
cohomological annihilator if and only if $\emph\Att
\emph\H^i_\fm(M)\cap\emph\Min M=\emptyset$ for all $i=0,\ldots,d
-1$.
\end{lem}\index{attached prime, $\Att$}
\begin{proof}
Assume that $M$ has a uniform local cohomological annihilator.
 Therefore there is an element $x\in R\setminus
\underset{\fp\in\Min M}{\cup}\fp$ satisfying $x\H^i_\fm(M)=0$ for
all $i=0,\ldots,d-1$. Thus, by Corollary \ref{111}(iii),
  $x\in\underset{\fq\in\Att \H^i_\fm(M)}{\cap}\fq$ for all $0\leq i\leq
  d-1$. Now the claim is clear.

Conversely,  note that  if $\a(M)\subseteq\underset{\tiny{\fp\in\Min
M}}{\cup} \fp$, then $\a(M)\subseteq \fp$ for some prime ideal
$\fp\in\Min M$, by prime avoidance theorem. Therefore,
$0:_R\H_\fm^i(M)\subseteq \fp$ for some $0\leq i\leq d-1$. On the
other hand one has  $0:_R
  M\subseteq 0:_R \H_\fm^i(M)\subseteq \fp$ which gives
  $\fp\in\Supp M$. Thus $\fp\in\Min(R/0:_R\H^i_\fm(M))$ and so
  $\fp\in\Att\H^i_\fm(M)$ by Corollary \ref{111}(ii), which
  contradicts our assumption. Hence $\a(M)\nsubseteq\underset{\tiny{\fp\in\Min M}}{\cup}
  \fp$ and the result follows.
\end{proof}

The following lemma provides a relation between prime ideals
containing $\a(M)$ and those $\fp$ which $M_\fp$ is not
Cohen-Macaulay.

\begin{lem}\label{228}
Assume that $(R,\fm)$ is a local ring,  $M$ is a finitely generated $R$--module and
$\fp\in\emph{\Supp} M$ such that $M_\fp$ is not Cohen-Macaulay. Then
$\a(M)\subseteq\fp$.
\end{lem}

 \index{finiteness dimension, $f_\fa(M)$!$\fb$--finiteness-, $f^\fb_\fa(M)$}

\index{a@$\fa$--adjusted depth!$\fb$--minimum-, $\lambda^\fb_\fa(M)$}

\index{depth!$\fb$--minimum $\fa$--adjusted-}

\index{b@$\fb$--minimum $\fa$--adjusted
depth, $\lambda^\fb_\fa(M)$}

\begin{proof} Let $d:=\dim M$  and assume contrarily
that $\mathrm{a}(M)\nsubseteq \fp$. Then by Theorem \ref{935BS}, we
have
$$d=f^{\mathrm{a}(M)}_\fm(M)\leq
\lambda^{\mathrm{a}(M)}_\fm(M)\leq \depth M_\fp+\dim R/\fp \leq \dim
M_\fp+\dim R/\fp\leq d.$$
 Hence $\depth
M_\fp=\dim M_\fp$ which  is a contradiction.
\end{proof}

As an immediate corollary of the above lemma, we have the following
result which  Zhou has also proved it, in \cite[Corollary 2.3]{Z},
for $M=R$ and by a different method.
\begin{cor}\label{Zx}
Assume that $x$ is a uniform local cohomological annihilator of a
finitely generated $R$--module $M$. Then $M_x$ is a Cohen-Macaulay
$R_x$--module.
\end{cor}
\begin{proof}
Let $\fm$ be a maximal ideal of $R$ with $\dim M_\fm>0$. Since $x\in \a(M_\fm)$, for any prime ideal $\fp\subseteq\fm$ with $x\notin\fp$, we
have $M_\fp$ is Cohen-Macaulay by Lemma \ref{228}.
\end{proof}

Another property of rings which contain  uniform local cohomological
annihilators is a result of Zhou \cite[Theorem 2.2]{Z} which proves
that  if $x$ is a uniform local cohomological annihilator of $R$,
then a power of $x$ is a strong uniform local cohomological
annihilator of $R$.  Using the above result and our approach to
uniform annihilators of local cohomology, we are able to recover
this result  in special case when $R$ is local, by a different
method. Before that, we mention the following well known fact.


\begin{lem}\emph{\cite[Thorem 31.7]{Ma}}\label{31.7Ma}
Assume that $(R,\fm)$ is universally catenary local ring and $M$ is
finitely generated and  equidimensional. Then $\widehat{M}$ is
equidimensional $\widehat{R}$-module.
\end{lem}


\begin{prop}\label{Zs}\index{uniform local cohomological
annihilator! strong-}\index{strong uniform local cohomological
annihilator} Assume that $(R,\fm)$ is a local ring, $M$ is a
finitely generated $R$--module  and $x$ is a uniform local
cohomological annihilator of $M$, then a power of $x$ is a strong
uniform local cohomological annihilator of $M$.
\end{prop}

\begin{proof}
Let $d=\dim M_\fm$. Note that
$M$ is equidimensional by Proposition \ref{equ} and  $R/0:_{\widehat{R}}\widehat{M}$ is universally
catenary by Corollary \ref{cat}. Thus Lemma  \ref{31.7Ma},  implies that  $\widehat{M}$ is
equidimensional.

Let $\fp\in\Supp M$ with $r=\h_M\fp$. We may choose elements $x_1,\ldots, x_r$ in $\fp$ such that
$\h_M(x_1,\ldots,x_r)=r$ and $\dim R/(x_1,\ldots,x_r)=d-r$. Set
$I=(x_1,\ldots,x_r)$, then $\dim \widehat{R}/I\widehat{R}=d-r$ and so  $\h_{\widehat{M}}(I\widehat{R})=r$.

Note that $\C_{\widehat{R}}(\widehat{M})$ is finite, by Corollary
\ref{216} and  $\widehat{M}_x$ is Cohen-Macaulay by Corollary
\ref{Zx}, which means that
$\mH^i_{\widehat{M}_x}=(\mH^i_{\widehat{M}})_x=0$ for all $i\geq
-1$. Since $\mH^i_{\widehat{M}}$ is finitely generated
$\widehat{R}$--module, there exists a positive integer $n$ such that
$x^n\in\cap_{i\geq -1}(0:_{\widehat{R}}\mH^i_{\widehat{M}})$ (note
that $\mH^i_{\widehat{M}}=0$ for $i\geq d-1$).

Now, Corollary \ref{c224} implies that $x^{nd}\in
0:_{\widehat{R}}\mH^i_{I\widehat{R}}(\widehat{M})$ for
$i<\h_{\widehat{M}}(I\widehat{R})=r$.  Hence $x^{nd}\mH^i_I(M)=0$
for $i<r=\h_MI$ and the result follows by the fact that
$\H^i_{IR_\fp}(M_\fp)\cong \H^i_{\fp R_\fp}(M_\fp)$.
\end{proof}


The following theorem gives a characterization for a finitely
generated module $M$ over a local ring to have a uniform local
cohomological annihilator in terms of the existence of a specific
parameter element of $M$. In proving (ii)$\Longrightarrow$ (i) of
this theorem,  A. Talemi--Eshmanani had a fruitful cooperation.


\begin{thm}\label{2210}\index{uniform local cohomological
annihilator}
Let $(R,\fm)$ be  local and $M$ be a finitely generated $R$--module
with dimension $d= \emph{\dim} M>1$. Then the following statements
are equivalent.
\begin{enumerate}
\item[\emph{(i)}] $M$ has a uniform local cohomological annihilator.
\item[\emph{(ii)}] $R/0:_RM$ is catenary and equidimensional, there exists a parameter
element $x$ of $M$ such that $\emph\Min (M/xM)\cap \emph\Ass M=
\emptyset$ and all modules $M/x^t M$, $t\in\mathbb{N}$, have a
common uniform local cohomological annihilator.
\end{enumerate}
\end{thm}
\begin{proof} In the following argument we fix $N$ to be a submodule of $M$
such that $\Ass N= \Min M$ and
$\Ass M/N=\Ass M\setminus\Min M$  (see \cite[Page 263, Proposition 4]{B} for existence of $N$).\\

 (i)$\Rightarrow$(ii) By Proposition \ref{equ}, $M$ is equidimensional and  $R/0:_RM$ is catenary by Corollary \ref{cat}. Set
$X=\{\fp\in\Ass M : \h_M\fp= 1\}$.  It can be easily checked that
$X= \{\fp\in\Supp M/N: \h_M\fp= 1\}$ and it is a subset of $\Min
M/N$ so that $X$ is a finite set. Assume that $r$ is an element of
$R\setminus \cup_{\fp\in\Min M}\fp$ which is a uniform local
cohomological annihilator of $M$. If $r$ is unit element then $M$ is
Cohen--Macaulay and the
claim follows by choosing $x$ to be a non--zero--divisor on $M$.\\

Therefore we assume that $r\in\fm\setminus \cup_{\fp\in\Min M}\fp$,
so that $\dim M/rM= d-1$. Note that, as $M$ is equidimensional and
$\dim M>1$, $\fm\not\in\Min M$, $\fm\not\in X$ and $\fm\not\in\Min
M/rM$. Hence there exist
\begin{equation}\label{2.1} x\in\fm\setminus(\underset{\fp\in\Min M}{\cup}\fp)
\cup(\underset{\fp\in X}{\cup}\fp)\cup(\underset{\fp\in\Min
M/rM}{\cup}\fp).\end{equation}
 It follows that $\Min M/xM\cap\Ass M= \emptyset$.

We claim that  $r\not\in \underset{\fp\in\Min M/xM}{\cup}\fp$.
Otherwise $r\in\fp$ for some $\fp\in\Min M/xM$ and so $\h_M\fp=1$
which
 implies that $\fp\in\Min M/rM$. This contradicts with the chosen $x$ in (\ref{2.1}).

 Next we claim that $0:_RM/N\nsubseteq\cup_{\fp\in\Min M/xM}\fp$. Otherwise
$0:_RM/N\subseteq\fp$ for some $\fp\in\Min M/xM$. Thus $\fp\in\Supp
M/N$ and $\h_M\fp= 1$ which shows that $\fp\in X$. This also
contradicts with (\ref{2.1}). As a result there exists an element
$s\in 0:_RM/N\setminus\cup_{\fp\in\Min M/xM}\fp$. Now consider the
induced exact sequences

$$\H_\fm^{i-1}(M/N)\longrightarrow \H_\fm^i(N)\longrightarrow
\H_\fm^i(M).$$ As $s\H_\fm^{i-1}(M/N)=0$ for all $i$, and
$r\H_\fm^i(M)=0$ for all $i<d$, we get $rs\H_\fm^i(N)=0$ for all
$i<d$.

Now, for the module $N$, we have $\Ass N= \Min M$ and so, by
(\ref{2.1}),
 each of the elements $x^t$ is a
non--zero--divisor on $N$. Choose an arbitrary positive integer $t$
and consider the exact sequence

 $$0\longrightarrow N\overset{x^t}{\longrightarrow} N\longrightarrow
 N/x^tN\longrightarrow 0$$
 which induces the exact sequence

$$\H_\fm^i(N){\longrightarrow} \H_\fm^i(N/x^{t}N)\longrightarrow \H_\fm^{i+1}(N).$$
As $rs\H_\fm^i(N)=0$ for $i<d$, we get $(rs)^2\H_\fm^i(N/x^tN)=0$
for all $i<d-1$.

On the other hand the exact sequence
$$0\longrightarrow N/x^tN\longrightarrow M/x^tN\longrightarrow M/N\longrightarrow 0$$
implies the exact sequences
$$\H_\fm^i(N/x^tN)\longrightarrow
\H_\fm^i(M/x^tN)\longrightarrow \H_\fm^i(M/N)$$
 from which it follows that $(rs)^2s\H_\fm^i(M/x^tN)=0$ for all $i<d-1$.

 Finally, the exact sequence
$$0\longrightarrow x^tM/x^tN\longrightarrow M/x^tN\longrightarrow
 M/x^tM\longrightarrow0$$
implies the exact sequences
$$\H_\fm^i(M/x^tN)\longrightarrow \H_\fm^i(M/x^tM)\longrightarrow
 \H_\fm^{i+1}(x^tM/x^tN).$$

Note that $s(M/N)=0$ implies that $s\H_\fm^i(x^tM/x^tN)=0$ for all
$i$. Therefore  $$(rs)^2s^2\H_\fm^i(M/x^tM)=~0,$$ for all $i<d-1$.

As $r,
s\in R\setminus\cup_{\fp\in\Min M/xM}\fp$ and $\dim M/x^tM=d-1$, the
element $r^2s^4$ is a uniform local cohomological annihilator for
$M/x^tM$.\\

(ii)$\Rightarrow$(i). Assume that all modules $M/x^t M$,
$t\in\mathbb{N}$,  have a common uniform local cohomological
annihilator $r$, say. We first observe that
$$0:_RM/N\not\subseteq\underset{\fp\in\Min M/xM}{\cup} \fp.$$
Otherwise there is a prime ideal $\fp\in\Min M/xM$ which $0:_R
M/N\subseteq \fp$. As $\dim M/N\leq d-1$ and $\h_M\fp= 1$, we find
that $\fp\in\Min M/N$ and so $\fp\in\Ass M$ which contradicts our
assumption $\Min M/xM\cap \Ass M= \emptyset$. As a result there is
an element $$s\in 0:_RM/N\setminus \underset{\fp\in\Min
M/xM}{\cup}\fp.$$

Consider an arbitrary positive integer $t$. From the exact sequence
$$0\longrightarrow x^tM/x^tN  \longrightarrow M/x^tN\longrightarrow
 M/x^tM\longrightarrow0$$ it follows the induced exact sequences
 $$\H_\fm^i(x^tM/x^tN)\longrightarrow
\H_\fm^i(M/x^tN)\longrightarrow \H_\fm^i(M/x^tM).$$ Since $s(M/N)=
0$ and $r$ is a uniform local cohomological annihilator of
$\H_\fm^i(M/x^tM)$, it follows that $rs\H_\fm^i(M/x^tN)=0$ for all
$i< d-1$.

On the other hand, the exact sequence $0\longrightarrow N/x^tN
\longrightarrow M/x^tN\longrightarrow
 M/N\longrightarrow0$ implies the exact sequences $$\H_\fm^{i-1}(M/N)\longrightarrow
\H_\fm^i(N/x^tN)\longrightarrow \H_\fm^i(M/x^tN)$$ from which it
follows that $rs^2\H_\fm^i(N/x^tN)= 0$ for all $i< d-1$.

From the fact that $M$ is catenary and equidimensional  and that $x$
is a parameter element of $M$ it follows that $\underset{\fp\in\Min
N}{\cup}\fp\subseteq\underset{\fq\in\Min M/xM}{\cup}\fq$ and so
$rs^2\not\in\underset{\fp\in\Min N}{\cup}\fp$.

Our next step is to show that $rs^2\H_\fm^i(N)= 0$ for all $i<d$.
Let $i<d$ and choose an arbitrary element $\alpha\in \H_\fm^i(N)$.
By torsionness of local cohomology modules, there is a positive
integer $t$ such that $\alpha\in(0\underset{\H_\fm^i(N)}{:} x^t)$.
As $\Ass N= \Min M$ and $x^t\not\in\underset{\fp\in\Min
M}{\cup}\fp$, $x^t$ is a non--zero--divisor on $N$. Thus the exact
sequence $0\longrightarrow N
\overset{x^t}{\longrightarrow}N\longrightarrow
 N/x^tN\longrightarrow0$ implies the exact sequence $\H_\fm^{i-1}(N/x^tN)
\longrightarrow \H_\fm^i(N)\overset{x^t}{\longrightarrow}
 \H_\fm^i(N)$ and also the exact sequence
 $$\H_\fm^{i-1}(N/x^tN)\longrightarrow(0\underset{\H_\fm^i(N)}{:}
 x^t)\longrightarrow0.$$ As $rs^2\H_\fm^{i-1}(N/x^tN)= 0$, we obtain
 that $rs^2(0\underset{\H_\fm^i(N)}{:}
 x^t)= 0$. In particular, $rs^2\alpha= 0$. Therefore,
  as $r,s\not\in\underset{\fp\in\Min N}{\cup}\fp$,  $rs^2$ is
  a uniform local cohomological annihilator of $N$.
Since $\Supp M= \Supp N$, $M$ has a uniform local cohomological
annihilator, by Corollary \ref{c215}.
\end{proof}

Now we introduce an important class of modules which have uniform
local cohomological annihilators. This is the class of modules with
finite Cousin complexes.  In next chapter, we will show that these
two classes coincide, if all formal fibres of $R$ are Cohen-Macaulay.

\begin{thm}\label{229}
Assume that $M$ is a finitely generated  $R$--module of finite
$\emph{\dim} M=d$ and that $\mathcal{C}_R(M)$ is finite, then $M$
has a uniform local cohomological annihilator. \index{uniform local
cohomological annihilator} \index{Cousin complex, $\C_R(M)$!finite-}
\end{thm}
\begin{proof}
By Lemma \ref{139}, there exists an element $x\in\cap_{i\geq
-1}\mH_M^i\setminus\cup_{\fp\in \Min M}\fp$. Now $x^d$ is a uniform
local cohomological annihilator of $M$, by  Corollary \ref{c224}.
\end{proof}

In Theorem \ref{2210}, we have shown that if $M$ has a uniform local
cohomological annihilator, then there exists  a parameter element
$x$ of $M$ which $M/xM$ has also a uniform local cohomological
annihilator. Now, when $\C_R(M)$ is finite, we may show this
property for all parameter elements of $M$. Firstly we see the
result for each non--zero--divisor $x$ of $M$.
\begin{prop}\label{2214}\index{Cousin complex, $\C_R(M)$!finite-}\index{uniform local cohomological
annihilator}
Let $(R, \fm)$ be local ring and let $M$ be a finitely generated
$R$--module with $\emph\dim M= d> 1$ such that $\mathcal{C}_R(M)$ is
finite. Then, for each non--zero--divisor $x$ of $M$, the quotient
module $M/xM$ has a uniform local cohomological annihilator.
Moreover, $\emph\Min(M/xM)\cap \emph\Ass M= \emptyset$ and all
modules $M/x^t M$, $t\in\mathbb{N}$, have a common uniform local
cohomological annihilator.
\end{prop}
\begin{proof}
Let $\fp\in\Min M/xM$. Then $\h_M\fp= 1$ and so $\fp\not\in\Supp
\mH_M^i$ for all $i\geq 0$ by Theorem \ref{2.7S1}(ii). If
$\fp\in\Supp \mH_M^{-1}$ then $\fp\in\Min \mH_M^{-1}$ and so
$\fp\in\Ass M$ which contradicts with the fact that $x$ is a
non--zero--divisor on $M$. Hence $$\sqcap_{i\geq -1}(0:_R
\mH_M^i)\not\subseteq\underset{\fp\in\Min M/xM}{\cup}\fp.$$
Therefore, there is an element $r\in \sqcap_{i\geq -1}(0:_R
\mH^i_M)\setminus\underset{\fp\in\Min M/xM}{\cup}\fp$. Now, from the
exact sequence $0\longrightarrow M\longrightarrow M\longrightarrow
M/xM\longrightarrow 0$ we get the exact sequence
$$\cdots\longrightarrow \H_\fm^i(M)\longrightarrow
\H_\fm^i(M/xM)\longrightarrow \H_\fm^{i+1}(M)\longrightarrow\cdots
.$$ By Corollary \ref{c224}, $r\H_\fm^i(M)= 0$ for all $i< d$. Then
the above exact sequence implies that $r^2\H_\fm^i(M/xM)= 0$ for all
$i< d-1$. As $r\in\underset{\fp\in\Min M/xM}{\cup}\fp$, $r^2$ is a
uniform local cohomological annihilator of $M/x^n M$ for all
positive integers $n$ and $\Min M/xM\cap \Ass M= \emptyset$.
\end{proof}
\begin{thm}\label{cfx}
Let $(R, \fm)$ be a local ring and let $M$ be a finitely generated
$R$--module with finite Cousin complex. Then $M/xM$ has a uniform
local cohomological annihilator for any parameter element $x$ of
$M$.
\end{thm}
\begin{proof}
There is a submodule $N$ of $M$ such that $\Ass M/N= \Min M$ and
$\Ass N= \Ass M\setminus \Min M$ (c.f. \cite[Page 263, Proposition
4]{B}).
 As in the
exact sequence $0\longrightarrow N\longrightarrow M\longrightarrow
M/N\longrightarrow 0$ we have $\h_M\fp\geq 1$ for all $\fp\in\Supp
N$, Lemma \ref{211}(b) implies that $\mathcal{C}_R(M/N)$ is finite.
Assume that $x$ is parameter element of $M$. As $x$ is a
non--zero--divisor on $M/N$, Proposition \ref{2214} implies that
$(M/N)/x(M/N)$ has a uniform local cohomological annihilator. Note
that $\Supp M/xM= \Supp (M/N)/x(M/N)$ so, by Corollary \ref{c215},
$M/xM$ has a uniform local cohomological annihilator.
\end{proof}

\section{Applications}

\quad Our study of properties of Cousin cohomologies in Section 2.1,
provides a new approach to the property of uniform annihilators of
local cohomologies. This point of view enabled us to study both
classes of modules, more deeply in section 2.2. In this section we
present some more applications to characterize  modules with finite
cousin complexes (over rings with  some extra conditions) and also
to investigate a new formula for height of an ideal in terms of
cohomologies of Cousin complexes.

\subsection {Partial characterizations}

We start with the following result which has an essential role in
our approach. \index{locally equidimensional}
\index{equidimensional!locally-} \index{universally
catenary}\index{Cousin complex, $\C_R(M)$!finite-}
\begin{cor}\label{cor}
Assume that $M$ is a finitely generated $R$--module with  finite
dimension and that $ \mathcal{C}_R(M)$ is finite. Then $M$  is
locally equidimensional and $R/0:_RM$ is universally catenary.
\end{cor}
\begin{proof}
It is clear from Theorem \ref{229}, Proposition \ref{equ} and
Corollary \ref{cat}.
\end{proof}


Now it is easy to provide an example of a module whose Cousin
complex has at least one non--finitely generated cohomology.\\

{\bf Example.} Consider a noetherian local ring $R$ of dimension $d>
2$. Choose any pair of prime ideals $\fp$ and $\fq$ of $R$ with
conditions $\dim R/\fp= 2$, $\dim R/\fq= 1$, and $\fp\not\subseteq
\fq$. Then $\Min R/\fp\fq=\{\fp, \fq\}$ and so $R/\fp\fq$ is not an
equidimensional $R$--module and thus its
Cousin complex is not finite.\\


In \cite[Corollary 3.3]{Z},  Zhou proves that any locally
equidimensional noetherian ring has a uniform local cohomological
annihilator provided it is a homomorphic image of a Cohen-Macaulay
ring of finite dimension. Note that any homomorphic image of a
Cohen-Macaulay ring is universally catenary and all of it's formal
fibres are Cohen-Macaulay.  Here we extend this result to any
universally catenary local ring with Cohen-Macaulay formal fibres,
by showing that over these rings, every equidimensional module has
finite Cousin complex,  which also recovers the result of  Kawasaki,
 \cite[Theorem 5.5]{K}, by a different method.

\begin{prop}\label{233}
Assume that $R$ is universally catenary and all formal fibres of $R$
are Cohen-Macaulay. If $M$ is a finitely generated and
equidimensional $R$--module, then $\C_R(M)$ is finite; in particular
$M$  has a uniform local cohomological annihilator. \index{Cousin
complex, $\C_R(M)$!finite-} \index{equidimensional}
\end{prop}

\begin{proof}
By Lemma \ref{31.7Ma}, $\widehat{M}$ is equidimensional, so
Corollary \ref{216} implies the result.
\end{proof}

We are now ready to present the following result which, for a
finitely generated $R$--module $M$, shows connections of finiteness
of its Cousin complex, existence of a uniform local cohomological
annihilator for $M$, and equidimensionality of $\widehat{M}$.

\begin{thm}\label{234}\index{formal fibre}
Assume that $(R,\fm)$ is local and all formal fibers of $R$ are
Cohen-Macaulay. Then the following statements are equivalent for a
finitely generated $R$--module $M$.
\begin{itemize}
\item[\emph{(i)}] $\widehat{M}$ ia an equidimensional $\widehat{R}$--module.
\item[\emph{(ii)}] The Cousin complex of $M$ is
finite.
\item[\emph{(iii)}] $M$ has a uniform local cohomological annihilator.
\end{itemize}
\end{thm}
\begin{proof}
(i) $\Rightarrow$ (ii). This is Corollary \ref{216}.

(ii) $\Rightarrow$ (iii). This is Theorem \ref{229}.

(iii) $\Rightarrow$ (i). There exists an element $x\in R\setminus
\cup_{\fp\in\Min M}\fp$ such that $x\H_\fm^i(M)= 0$ for all $i< \dim
M$, so that  $x\H_{\widehat{\fm}}^i(\widehat{M})= 0$ for all $i<
\dim \widehat{M}$.  Now $\widehat{M}$ is equidimensional by
Proposition \ref{equ}.
\end{proof}

\subsection{Height of an ideal}
\index{height!m@$M$--height, $\h_M$} We use some results about the
annihilators of cohomologies of Cousin complexes, to present the
height of an ideal in terms of Cousin complexes.

 As mentioned in Corollary \ref{c224}, we may write the
following result.

\begin{cor}
For any finitely generated $R$--module $M$ and any ideal $\fa$ of
$R$ with $\fa M\not=M$,
 $$\underset{-1\leq i}{\prod}(0 :_R
\mH_M^i)\subseteq 0 :_R \H_\fa^{\h_M\fa- 1}(M).$$
\end{cor}
 We now raise the question that whether it is
possible to improve the upper bound restriction. \\
 \noindent {\bf
Question.} Does the inequality
$$\underset{-1\leq i}{\prod}(0 :_R
\mH_M^i)\subseteq 0 :_R \H_\fa^{\h_M\fa}(M)$$ hold?

It will be proved that the answer is negative for the class of
finitely generated $R$--modules  with finite Cousin cohomologies.
More precisely,

\begin{thm}\label{height}
Assume that $M$ is a finitely generated  $R$--module of finite
dimension and that its Cousin complex $\mathcal{C}_R(M)$ is finite.
Then
$$\emph{\h}_M\fa= \inf\{r: \underset{-1\leq i}{\prod}(0 :_R
\mH_M^i)\not\subseteq 0 :_R \emph{\H}_\fa^r(M)\},$$ for all ideals
$\fa$ with $\fa M\not= M$.
\end{thm}
\begin{proof}
 $\underset{i\geq -1}{\prod}(0:_R\mH_M^i)
\subseteq 0:_R\H_\fa^r(M)$ for all $r< \h_M\fa$. Hence we have
$$\h_M\fa\leq \inf\{r: \underset{-1\leq i}{\prod}(0 :_R
\mH_M^i)\not\subseteq 0 :_R \H_\fa^r(M)\}.$$ Thus it is sufficient
to show that $\underset{-1\leq i}{\prod}(0 :_R \mH_M^i)\not\subseteq
0 :_R \H_\fa^{\h_M\fa}(M)$.  By Independence Theorem of local
cohomology, $\H_\fa^{\h_M\fa}(M)\cong \H_\fb^{\h_M\fb}(M)$ as
$\overline{R}=R/0:_R M$--module, where $\fb= (\fa+ 0:_R M)/0:_RM$.
Note that $\h_M\fa= \h_M\fb$ and that
$\mathcal{C}_R(M)\cong\mathcal{C}_{ \overline{R}}(M)$ by Lemma
\ref{1.2D}.

Hence we may assume that $0:_RM= 0$. Set $h:= \h_M\fa$. Let $x\in
0:_R \H_\fa^h(M)$. As $\fa M\not= M$, there exists a minimal prime
$\fq$ over $\fa$ in $\Supp M$ such that $\dim R_\fq= \h_M\fa$. Hence
$x/1\in 0:_{R_\fq} \H_{\fq R_\fq}^h(M_\fq)$. Thus, by any choice of
$\fp R_\fq\in\Assh M_\fq$ we have $x/1\in\fp R_\fq$ (see Theorem
\ref{G} and Corollary \ref{111}(iii)) and so $x\in\fp$ . Therefore,
one has $$0:_R \H_\fa^h(M)\subseteq \underset{\fp\in\Min M}{\cup}
\fp.$$ On the other hand, by Lemma \ref{139}, $\prod_{i\geq -1}(0 :_R
\mH_M^i)\not\subseteq \cup_{\fp\in\Min M}\fp$, from which it follows
that
$$\prod_{i\geq -1}(0 :_R \mH_M^i)\not\subseteq 0:_R\H_\fa^h(M).$$
\end{proof}

\def\baselinestretch{1}

\chapter{Attached primes of local cohomology modules}
\def\baselinestretch{1.3}

\quad Throughout this chapter ($R, \fm$) is  a local ring  and $M$
is a finitely generated $R$--module of dimension $d$. Recall that
$M$ has a uniform local cohomological  annihilators if and only if
$\a(M)\nsubseteq \cup_{\fp\in\Min M}\fp$. On the other hand we have
seen in  Lemma \ref{3.6}, that  if $\fp\in\Min M$, then
$\a(M)\subseteq \fp$ if and only if $\fp\in\Att\H^i_\fm(M)$ for some
$i<\dim M$. Inspired by these facts, we study $\Att \H^t_\fm(M)$ for
certain $t$, in particular $\Att \H^{d-1}_\fm(M)$, in terms of
cohomologies of $\C_R(M)$ and obtain a non--vanishing criterion of
$\H^{d-1}_\fm(M)$ when $\C_R(M)$ is finite, in section 3.1. We
continue by study the attached prime ideals of the top local
cohomology module $\H^d_\fm(M)$ in the second section and present a
positive answer to a question of \cite{DY3}, in the case when $R$ is
complete. The last section of this chapter is devoted to some
applications of our results to find a new characterization of
generalized Cohen-Macaulay \index{generalized Cohen-Macaulay, g.CM}
\index{Cohen-Macaulay!generalized-, g.CM} modules. \index{a@$\a(M)$}
\section{Attached primes related to cohomologies of Cousin complexes  }

\quad In this section we study some relations between the set of
attached primes of local cohomology modules $\H_\fm^i(M)$ and those
of cohomologies of the Cousin complex of $M$.

The  following result describes  the situation when all cohomology
modules  $\mH^i_M$ of the Cousin complex of $M$ are local cohomology
modules of $M$.
\index{local cohomology!-module}
\begin{lem}\label{1.3}
Assume that $i$ is an integer  with $0\leq i< d$. The following
statements are equivalent.
\begin{enumerate}
\item[\emph{(i)}] $\emph\dim \mH^{j}_M\leq 0$ for all $j$ with $-1\leq j< i$.
\item[\emph{(ii)}]  $\emph\H^{j+1}_\fm(M)\cong\mH^{j}_M$ for all $j$ with $-1\leq j< i$.
\end{enumerate}
\end{lem}
\begin{proof}
Assume that $s$ is an integer such that $0\leq s<d$ and $\dim
\mH_M^{s-1}\leq 0$. Considering the exact sequence (\ref{e1}) with
$i= s$,  gives the exact sequence
$$\H_\fm^{t-1}(M^{s})\longrightarrow\H_\fm^{t-1}
(M^{s}/D^{s})\longrightarrow
\H_\fm^{t}(M^{s-1}/K^{s-1})\longrightarrow\H_\fm^{t}(M^{s})$$ for
all integers $t$. As $s<d$, by Lemma \ref{tl} (a), we get
\begin{equation}\label{e.211}
\H_\fm^{t-1} (M^{s}/D^{s})\cong
\H_\fm^{t}(M^{s-1}/K^{s-1}).\end{equation}

 Next consider the exact sequence (\ref{e2})
 with $i= s$ which gives the exact sequence
$$\H_\fm^{t}(\mH_M^{s-1})\longrightarrow\H_\fm^{t}
(M^{s-1}/D^{s-1})\longrightarrow
\H_\fm^{t}(M^{s-1}/K^{s-1})\longrightarrow\H_\fm^{t+1}(\mH_M^{s-1}).$$
Choosing $t>0$ in the above exact sequence we obtain
\begin{equation}\label{e.212}
\H_\fm^{t} (M^{s-1}/D^{s-1})\cong
\H_\fm^{t}(M^{s-1}/K^{s-1}).\end{equation} As a consequence, from
(\ref{e.211}) and (\ref{e.212}), we get
\begin{equation}\label{e.213}
\H_\fm^t(M^{s-1}/D^{s-1})\cong\H_\fm^{t-1}(M^{s}/D^{s})\end{equation}
for all $t> 0$.

(i)$\Rightarrow$ (ii). Let $-1\leq j< i$. By repeated use of
(\ref{e.213}), we get $$\H_\fm^{j+1}(M^{-1}/D^{-1})\cong
\H_\fm^0(M^{j}/D^{j}).$$ From the exact sequence (\ref{e1}) with $i=
j+1$ we have $\H_\fm^0(M^{j}/K^{j})= 0$ (because $j+1\leq i< d$ and
Lemma  \ref{tl}). Hence the exact sequence (\ref{e2}) with $i= j+1$
implies that $ \H_\fm^0(M^j/D^j) \cong \H_\fm^0(\mH_M^j)\cong
\mH_M^j$. Therefore $\H_\fm^{j+1}(M)\cong \mH_M^j$.

(ii)$\Rightarrow$ (i) is clear.
\end{proof}


 Theorem \ref{2.7S1}, implies that $\dim \mH^i_M\leq d-i-2$ for
all $i\geq -1$. The following lemma states some properties for $t$th
local cohomology modules of $M$,  whenever  $\dim \mH^i_M\leq t-i-1$
for all $i\geq -1$, in particular for $\H^{d-1}_\fm(M)$.

\begin{lem}\label{312} \index{attached
prime, $\Att$}
Assume that $0\leq t< d$ is an integer such that $\emph\dim
\mathcal{H}_M^i\leq t-i-1$, for all $i\geq -1$.  Then the following
statements hold true.
\begin{itemize}
\item[\emph{(i)}]
$\emph\Att \emph\H^t_\fm(M)\subseteq
\underset{i=-1,\ldots,t-1}{\bigcup}\emph\Att
\emph\H^{t-i-1}_\fm(\mH_M^i).$

\item[\emph{(ii)}] There is an epimorphism $\emph\H_\fm^t(M)\twoheadrightarrow\emph\H_\fm^0(\mH_M^{t-1})$.
\item[\emph{(iii)}] Assume that $\mathcal{C}_A(M)$ is finite. Then $\mH_M^{t-1}$ is non--zero if and only if
$\fm\in\emph\Att \emph\H^t_\fm(M)$.
\end{itemize}
\end{lem}
\begin{proof}
(i).  We prove by induction on $j$, $-1\leq j\leq t-1$, that
\begin{equation}\label{e.214}
\Att \H_\fm^t(M)\subseteq\bigcup_{i\geq-1}^j
\Att\H_\fm^{t-i-1}(\mH_M^i)\bigcup\Att\H_\fm^{t-j-1}(M^{j}/K^{j}).\end{equation}
Due to  $\dim\mH_M^i\leq t-i-1$, the \index{Grothendieck's vanishing
theorem} Grothendieck's vanishing theorem implies that
$\H_\fm^{t-i}(\mH_M^i)= 0$. The exact sequence (\ref{e2}) with
$l=0$, implies the exact sequence
$\H_\fm^t(\mH_M^{-1})\longrightarrow\H_\fm^t(M)\longrightarrow\H_\fm^t(M^{-1}/K^{-1})\longrightarrow
0$. Thus we get
$$\Att \H_\fm^t(M)\subseteq\Att \H_\fm^t(\mH_M^{-1})\bigcup\Att \H_\fm^t(M^{-1}/K^{-1}).$$
Assume that $-1\leq j< t-1$ and (\ref{e.214}) holds. First note that
(\ref{e1}) with $l=j+1$  implies the exact sequence
$$\H_\fm^{t-j-2}(M^{j+1})\longrightarrow\H_\fm^{t-j-2}(M^{j+1}/D^{j+1})\longrightarrow\H_\fm^{t-j-1}(M^j/K^j)
\longrightarrow\H_\fm^{t-j-1}(M^{j+1}).$$ As $-1\leq j< t-1$,
$\H_\fm^{t-j-2}(M^{j+1})= 0$ and $\H_\fm^{t-j-1}(M^{j+1})= 0$, by
Lemma \ref{tl}, therefore
\begin{equation}\label{e.215}
\H_\fm^{t-j-2}(M^{j+1}/D^{j+1})\cong\H_\fm^{t-j-1}(M^j/K^j).
\end{equation}
 On the other hand from the exact sequence (\ref{e2}) with $l= j+2$ we have the exact
sequence
\begin{equation}\label{e.216}
\H_\fm^{t-j-2}(\mH_M^{j+1})\longrightarrow\H_\fm^{t-j-2}(M^{j+1}/D^{j+1})
\longrightarrow\H_\fm^{t-j-2}(M^{j+1}/K^{j+1})\longrightarrow\H_\fm^{t-j-1}(\mH_M^{j+1}).\end{equation}
As $\H_\fm^{t-j-1}(\mH_M^{j+1})= 0$, (\ref{e.215}) with the exact
sequence (\ref{e.216}) imply that
\begin{equation}\label{e.217}
\begin{array}{ll}\Att \H_\fm^{t-j-1}(M^{j}/K^{j})&=
\Att \H_\fm^{t-j-2}(M^{j+1}/D^{j+1})\\&\subseteq \Att
\H_\fm^{t-j-2}(\mH_M^{j+1})\bigcup\Att
\H_\fm^{t-j-2}(M^{j+1}/K^{j+1}).\end{array}
\end{equation}
Now, (\ref{e.217}) and (\ref{e.214}) complete the induction
argument. Thus we have
$$\Att \H^t_\fm(M)\subseteq\underset{i= -1, 0, \cdots,
t-1}{\bigcup}\Att \H_\fm^{t-i-1}(\mH_M^i)\bigcup\Att
\H_\fm^0(M^{t-1}/K^{t-1}).$$ On the other hand, considering the fact
that $\H_\fm^0(M^t)= 0$, it follows from the exact sequence
(\ref{e1}) with $l= t$, that $\H_\fm^0(M^{t-1}/K^{t-1})= 0$.

(ii). Consider the exact sequence (\ref{e1}) with $l= t-i$
 which imply the exact sequence
$$\H_\fm^{i-1}(M^{t-i})\longrightarrow\H_\fm^{i-1}(M^{t-i}/D^{t-i})\longrightarrow\H_\fm^{i}(M^{t-i-1}/K^{t-i-1})
\longrightarrow\H_\fm^{i}(M^{t-i}).$$ Taking $0\leq i\leq t$ we get
$0\leq t-i\leq t< d$ and so $\H_\fm^{i}(M^{t-i})= 0 =
\H_\fm^{i-1}(M^{t-i})$, by Lemma \ref{tl}. Therefore we have
isomorphisms
\begin{equation}\label{e.218}
\H_\fm^{i-1}(M^{t-i}/D^{t-i})\cong
\H_\fm^{i}(M^{t-i-1}/K^{t-i-1})\end{equation} for all $i$, $0\leq i
\leq t$.

Consider the exact sequence (\ref{e2}) with $l= t-i$ which induces
the exact sequence
$$\H_\fm^i(\mH_M^{t-i-1})\longrightarrow\H_\fm^i(M^{t-i-1}/D^{t-i-1})\longrightarrow
\H_\fm^i(M^{t-i-1}/K^{t-i-1})\longrightarrow\H_\fm^{i+1}(\mH_M^{t-i-1}).$$
As, by assumption $\dim \mH_M^{t-i-1}\leq i$, we have
$\H_\fm^{i+1}(\mH_M^{t-i-1})= 0$ and so one obtains an epimorphism
\begin{equation}\label{e.219}
\H_\fm^i(M^{t-i-1}/D^{t-i-1})\twoheadrightarrow\H_\fm^i(M^{t-i-1}/K^{t-i-1}).
\end{equation}
By successive use of (\ref{e.219}) and (\ref{e.218}) one obtains an
epimorphism

$$\H_\fm^t(M^{-1}/D^{-1})\twoheadrightarrow\H_\fm^0(M^{t-1}/D^{t-1}).
$$
 On the other hand, we have seen at the end of part (i)
that, we have
 $\H_\fm^0(M^{t-1}/K^{t-1})= 0$. Therefore, from  the exact sequence (\ref{e2}) with $l= t$,
 we get $\H_\fm^0(\mH_M^{t-1})\cong\H_\fm^0(M^{t-1}/D^{t-1})$
which
 results an epimorphism
 \begin{equation}\label{e.2110}
 \H_\fm^t(M)\twoheadrightarrow\H_\fm^0(\mH_M^{t-1}).\end{equation}

 (iii). Assume that $\mH_M^{t-1}\not =0$. As, by assumption
 $\dim \mH_M^{t-1}\leq 0$,  we have
 $\H_\fm^0(\mH_M^{t-1})= \mH_M^{t-1}$ and so
 $\Att \H_\fm^0(\mH_M^{t-1})=\{\fm\}$. Now (\ref{e.2110}) implies
 that $\fm\in\Att \H_\fm^t(M)$.

 Conversely, assume that $\fm\in\Att \H_\fm^t(M)$. By part (i),
 $\fm\in\Att \H_\fm^{t-i-1}(\mH_M^i)$ for some $i$, $-1\leq i\leq
 t-1$, and thus $\dim \mH_M^i\geq t-i-1$. As $\dim \mH_M^i\leq
 t-i-1$ we have equality $\dim \mH_M^i=
 t-i-1$. Note that $\mH_M^i$ is finitely generated and so $\fm\in\Assh \mH_M^i$ (see Theorem \ref{G})
 from which it follows
 that $t-i-1= 0$, i.e. $\mH_M^{t-1}\not=0.$
\end{proof}

It is  known that $\Att \H_\fm^d(M)= \Assh M$ (Theorem \ref{G}). The
following result provides some information about $\Att \H_\fm^t(M)$
for certain $t$, in particular for $t= d-1$.

 \begin{prop}\label{3231}\index{Cousin complex, $\C_R(M)$!finite-}

  \index{attached prime, $\Att$}
 Assume that  $\mathcal{C}_R(M)$ is finite. Let $t$, $0\leq t<
d$, be an integer such that $\emph\dim \mH_M^i\leq t-i-1$, for all
$i\geq -1$. Then

 $$\emph\Att \emph\H^{t}_\fm(M)=\bigcup^{t-1}_{i=-1}\{\fp\in\emph\Ass \mH_M^i: \emph\dim R/\fp=t-i-1\}.$$
\end{prop}
\begin{proof}
Assume that $\fp\in\Att \H_\fm^t(M)$. Then $\fp\in\Att
\H_\fm^{t-i-1}(\mH_M^i)$ for some $i$, $-1\leq i\leq t-1$ by Lemma
\ref{312}(i). As $\dim \mH_M^i \leq t-i-1$, we have the equality
$\dim \mH_M^i= t-i-1$ and so $\fp\in\Ass \mH_M^i$ and $\dim R/\fp=
t-i-1$.

 Conversely, assume that $-1\leq
i_0\leq t-1$  and that $\fp\in \Ass \mH_M^{i_0}$ such that $\dim
R/\fp=t-i_0-1$. Set $d':=\dim M_{\fp}$ and $t':= t-\dim R/\fp$.

 As
$\C_R(M)$ is finite, $M$ is equidimensional and $\Supp M$ is
catenary by Proposition \ref{equ} and Corollary \ref{cat}, we have
$0\leq t'<d'$. Note that $\mathcal{C}_{A_\fp}(M_\fp)\cong
(\mathcal{C}_A(M))_\fp$ so that $(\mH_M^j)_\fp\cong\mH_{M_\fp}^j$
for all $j$. As $\dim \mH_{M_\fp}^j\leq \dim \mH_M^j -\dim R/\fp$ ,
we find that $\dim \mH^j_{M_\fp}\leq t'-j-1$ for all $j\geq -1$. On
the other hand $\mH_{M_\fp}^{t'-1}= (\mH_M^{i_0})_\fp\not =0$. Lemma
\ref{312}(iii), replacing $M$ by $M_\fp$ implies that $\fp
R_\fp\in\Att \H_{\fp R_\fp}^{t'}(M_\fp)$. Finally the Weak General
Shifted Localization Principle (\ref{slp}), implies that $\fp\in\Att
\H_\fm^{t'+\dim A/\fp}(M)$ that is $\fp\in\Att \H_\fm^t(M)$.
\end{proof}

 \quad

As a consequence of the above result, we may find a connection
between vanishing of certain local cohomology modules and the
dimensions of cohomologies of Cousin complex.

\begin{cor}
 Assume that  $\mathcal{C}_R(M)$  is finite. Let $l< d$ be
an integer. The following statements are equivalent.
\begin{itemize}
\item[\emph{(i)}] $\emph\H_\fm^j(M) =0$ for all $j$, $l<j< d$.
\item[\emph{(ii)}] $\emph\dim \mH_M^i\leq l-i-1$ for all $i\geq -1.$
\end{itemize}
\end{cor}
\begin{proof} (i)$\Longrightarrow$(ii). We prove it by descending induction on $l$. For $i=
d-1$ we have nothing to prove by Lemma \ref{tl}. Assume that $l<
d-1$. We have, by induction hypothesis, that $\dim R/\fp\leq
(l+1)-i-1$ for all $\fp\in\Supp \mH_M^i$ and for all $i\geq -1$. If,
for an ideal $\fp\in\Supp \mH_M^i$ and an integer $i$, $\dim
R/\fp=(l+1)-i-1$, then we get $\fp\in\Ass \mH_M^i$ and so
$\H_\fm^{l+1}(M)\not = 0$ by Proposition \ref{3231}, which
contradicts the assumption. Therefore $\dim R/\fp\not=(l+1)-i-1$ for
any $\fp\in\Supp \mH_M^i$ and all $i\geq -1$. That is $\dim \mH_M^i<
(l+1)-i-1$ for all $i\geq -1$. In other words, $\dim \mH_M^i\leq
l-i-1$ for all $i\geq -1$.

(ii)$\Longrightarrow$(i). By descending induction on $l$. For $i=
d-1$ we have nothing to prove. Assume that $l<d-1$. As $\dim
\mH_M^i\leq l-i-1<(l+1)-i-1$ for all $i\geq -1$, we have, by
induction hypothesis, that $\H_\fm^j(M)=0$ for all $j$, $l+1<j< d$.
Moreover, Proposition \ref{3231} implies that $\Att \H_\fm^{l+1}(M)$
is empty so that $\H_\fm^{l+1}(M)= 0$.
\end{proof}

The following result is now a clear conclusion of the above
corollary.

\begin{cor}
 Assume that  $M$ is not Cohen-Macaulay and that $\mathcal{C}_R(M)$
is finite. Set $s= 1+ \sup\{\emph\dim \mH_M^i+i: i\geq -1\}$. Then
$\emph\H_\fm^s(M)\not= 0$ and  $\emph\H_\fm^i(M)= 0$ \emph{for all}
$i, s< i< d$.
\end{cor}
The following corollary gives us a non--vanishing criterion of
$\H_\fm^{d-1}(M)$ when $\mathcal{C}_R(M)$ is finite.
\begin{cor}\label{2.6}
Assume that $\mathcal{C}_R(M)$ is finite. Then
\begin{itemize}
\item[\emph{(i)}] $\emph\Att \emph\H^{d-1}_\fm(M)=\bigcup^{d-2}_{i=-1}\{\fp\in\emph\Ass \mH_M^i:
\emph\dim R/\fp=d-i-2\}.$
\item[\emph{(ii)}] $\emph\H_\fm^{d-1}(M)\not =0$ if and only if
$\emph\dim \mH_M^i= d-i-2$ for some $i$, $-1\leq i\leq d-2$.
\end{itemize}
\end{cor}
\begin{proof}
It is clear by Proposition \ref{3231}.
\end{proof}

  \index{attached prime, $\Att$}

\section{Top local cohomology modules}

\quad In this section we assume that $(R,\fm)$  is a local ring. We
say that $R$ is complete precisely when it is   complete with
respect to the $\fm$--adic topology.

 As mentioned before, by Theorem \ref{G},
$\Att \lc^d_\fm(M)= \Assh M$.   The following result is due to
Dibaei and Yassemi.

\begin{thm}\emph{\cite[Theorem A]{DY1}} \label{A} \index{cohomological dimension, $\cd(\fa,K)$|textbf} For any ideal $\fa$ of $R$,
$$\emph{\Att}\emph{\H}^{d}_\fa(M)=\{\fp\in \emph{\Supp} M \ : \ \emph{\cd}(\fa, R/\fp)=d\},$$
where $\emph{\cd}(\fa, K)$ is the cohomological dimension of an
$R$--module $K$ with respect to $\fa$, that is $\emph{\cd}(\fa,
K)=\sup\{i\in\Z \ : \ \emph{\H}^i_\fa(K)\neq 0\}$.
\end{thm}

Note that if $\fp\in\Supp M$ such that $\cd(\fa,R/\fp)=d$, then
$\dim R/\fp=d$. Therefore  $\Att \H^d_\fa(M)\subseteq \Assh M$. The
number of subsets  $T$ of $\Assh M$ with $\Att\lc^d_\fa(M)=T$, for
some  ideal $\fa$ of a complete ring $R$, is exactly the number of
 non--isomorphic top local cohomology modules of $M$ with respect to
all ideals of $R$, by the following theorem.

\begin{thm}\emph{\cite[Theorem 1.6]{DY2}}\label{dy2}
Assume that $R$ is complete. Then for any pair of ideals $\fa$ and
$\fb$ of $R$, if $\emph{\Att}\emph{\H}^d_\fa(M)=\emph{\Att}
\emph{\lc}^d_\fb(M)$, then $\emph{\lc}^d_\fa(M) \cong
\emph{\lc}^d_\fb(M)$.
\end{thm}

Now it is natural interesting to ask,

\begin{ques}\emph{\cite[Question 2.9]{DY3}} \label{question3}For any subset $T$ of
$\emph{\Assh} M$, is there an ideal $\fa$ of $R$ such that
$\emph{\Att} \emph{\lc}^d_\fa(M)=T$?
\end{ques}

Note that if $T=\Assh M$, then the maximal ideal is the answer. Thus
through this section we always assume that $T$ is a non--empty
proper subset of $\Assh M$.

In special case when $d=1$, it is easy to  deal  Question
\ref{question3}.

\begin{prop}\label{323}
If  $\emph{\dim} M=1$,
 then any subset $T$ of $\emph{\Assh} M$ is equal to the set
 $\emph{\Att} \emph{\lc}^1_\fa(M)$ for some ideal $\fa$ of $R$.
 \end{prop}
 \begin{proof}
 Set   $\fa:=\underset{\fp\in \Assh M\setminus T}{\cap}\fp$. Note that
  $\rad(\fa+\fp)=\fm$ for all $\fp\in T$ and $\fa+\fp=\fp$ for all
  $\fp\in\Assh M\setminus T$. Thus $\H^1_\fa(R/\fp)\neq 0$ if and
  only if $\fp\in T$.
 \end{proof}

In the following result we find a characterization for a subset of
$\Assh M$ to be the set of attached primes of the top local
cohomology of $M$ with respect to an ideal~$\fa$.

\index{attached prime, $\Att$} \index{local cohomology!-module}
\begin{prop}\label{321}
Assume that $R$ is  complete, $d\geq 1$ and set $\emph{\Assh}
M\setminus T=\{\fq_1,\ldots,\fq_r\}$. The following statements are
equivalent.
\begin{enumerate}
  \item[\emph{(i)}] There exists an ideal $\fa$ of $R$ such that
  $\emph{\Att} \emph{\lc}^d_\fa(M)=T$.
  \item[\emph{(ii)}] For each $i,\,1\leq i\leq r$, there exists $Q_i\in \emph{\Supp} M$
   with $\emph{\dim} R/Q_i=1$
  such that
  $$\underset{\fp\in T}{\bigcap}\fp\nsubseteq Q_i \quad \mbox{and}
  \quad \fq_i\subseteq Q_i.$$
\end{enumerate}
With $Q_i,\, 1\leq i\leq r$,  as above,  $\emph{\Att}
\emph{\lc}^d_\fa(M)=T$ where $\fa=\bigcap\limits_{i=1}^rQ_i$.
\end{prop}

\begin{proof}
(i) $\Rightarrow$ (ii). By Theorem \ref{A}, $\lc^d_\fa(R/\fp)\neq 0$
for all $\fp\in T$, that is $\fa+\fp$ is $\fm$--primary for all
$\fp\in T$, by the \LH (\ref{LH}). \index{Lichtenbaum-Hartshorne
vanishing theorem} On the other hand, for $1\leq i\leq r,
\fq_i\notin T$ which is equivalent to say that $\fa+\fq_i$ is not an
$\fm$--primary ideal. Hence there exists a prime ideal $Q_i\in \Supp
M$ such that $\dim R/Q_i=1$ and $\fa+\fq_i\subseteq Q_i$. It follows
that $\underset{\fp\in T}{\bigcap}\fp\nsubseteq Q_i$.
\\
\indent (ii) $\Rightarrow$ (i). Set
$\fa:=\bigcap\limits_{i=1}^rQ_i$. For each $i, 1\leq i\leq r$,
$\fa+\fq_i\subseteq Q_i$ implies that $\fa+\fq_i$ is not
$\fm$--primary and so $\lc^d_\fa(R/\fq_i)= 0$. Thus by Theorem
\ref{dy2} $\Att \lc_\fa ^d(M)\subseteq T$. Assume $\fp\in T$ and
$Q\in \Supp M$ such that $\fa+\fp \subseteq Q$. Then $Q_i\subseteq
Q$ for some $i, 1\leq i\leq r$. Since $\fp\nsubseteq Q_i$, we have
$Q_i\neq Q$, so $Q=\fm$. Hence $\fa+ \fp$ is $\fm$--primary ideal.
Now, by the \LH (\ref{LH}), and Theorem \ref{A}, it follows that
$\fp\in\Att \lc^d_\fa (M)$.
\end{proof}
\index{attached prime, $\Att$} \index{local cohomology!-module}
\begin{cor}\label{322}
If $\emph{\lc}^d_\fa (M)\not=0$ and $R$ is complete, then there is
an ideal $\fb$ of $R$ such that $\emph{\dim} R/\fb\leq 1$ and
$\emph{{\lc}}^d_\fa (M)\cong\emph{\lc}^d_\fb (M)$.
\end{cor}
\begin{proof}
 If $\Att \lc^d_\fa(M)=\Assh M$, then $\lc^d_\fa (M)\cong
\lc^d_\fm (M)$ by Theorem \ref{dy2}. Otherwise $d\geq 1$ and $\Att \lc^d_\fa(M)$ is a
proper subset of $\Assh M$. Set $\Assh M\setminus \Att
{\lc}^d_\fa(M):=\{\fq_1, \ldots, \fq_r\}$. By Proposition \ref{321},
there are $Q_i\in \Supp M$ with $\dim R/Q_i= 1, \ i=1, \ldots, r$,\
such that $\Att \emph{\lc}^d_\fa(M)= \Att \emph{\lc}^d_\fb(M)$ with
$\fb= {\bigcap\limits_{i=1}^rQ_i}$. Now, by  Theorem \ref{dy2}, we
have $\lc^d_\fa(M)\cong \lc^d_\fb(M)$. As $\dim R/\fb= 1$, the proof
is complete.
\end{proof}

 \vspace{0.5cm}
 Proposition \ref{321} provides  a useful method to find examples
of top local cohomology modules with specified attached primes.

\begin{exam}\label{324}
 \emph{Set $R=k[[X,Y,Z,W]]$, where $k$ is a field
and $X,Y,Z,W$ are independent indeterminates. Then $R$ is a complete
Noetherian local ring with maximal ideal $\fm=(X,Y,Z,W)$. Consider
prime ideals
$$\fp_1=(X,Y) \quad  ,  \quad \fp_2=(Z,W)\quad , \quad  \fp_3=(Y,Z)
\quad , \quad  \fp_4=(X,W)$$ and  set  $\displaystyle
M=\frac{R}{\fp_1\fp_2\fp_3\fp_4}$ as an $R$--module, so that we have
$\Assh  M=\{\fp_1,\fp_2,\fp_3,\fp_4\}$ and $\dim M=2$. We get
$\{\fp_i\}=\Att \lc^2_{\fa_i}(M)$, where $\fa_1=\fp_2, \fa_2=\fp_1,
\fa_3=\fp_4, \fa_4=\fp_3$, and $\{\fp_i,\fp_j\}=\Att
\lc^2_{\fa_{ij}}(M)$, where
$$\begin{array}{l}
\fa_{12}=(Y^2+YZ,Z^2+YZ,X^2+XW,W^2+WX),\\
\fa_{34}=(Z^2+ZW,X^2+YX,Y^2+YX,W^2+WZ),\\
\fa_{13}=(Z^2+XZ,W^2+WY,X^2+XZ),\\
\fa_{14}=(W^2+WY,Z^2+ZY,Y^2+YW),\\
\fa_{23}=(X^2+XZ,Y^2+WY,W^2+ZW),\\
\fa_{24}=(X^2+XZ,Y^2+WY,Z^2+ZW).\\
\end{array}$$
Finally, we have $\{\fp_i,\fp_j,\fp_k\}=\Att \lc^2_{\fa_{ijk}}(M)$,
where
$\fa_{123}=(X,W,Y+Z)$, $\fa_{234}=(X,Y,W+Z)$, $\fa_{134}=(Z,W,Y+X)$}.\\
\end{exam}

\index{attached prime, $\Att$} \index{local cohomology!-module}
\begin{lem}\label{325}
Assume that $R$ is complete, $d\geq 2$, and
  $\underset{\fp\in T}{\bigcap} \fp \nsubseteq \underset
   {\fq\in \emph{\Assh} R/\sum\limits_{\fp\in T'}\fp}{\bigcap} \fq$, where
   $T'=\emph{\Assh} M\setminus~T$.
Then there exists a prime ideal $Q\in \emph{\Supp}M$ with
$\emph{\dim}R/Q=1$ and $\emph{\Att}\emph{\lc}^d_Q(M)=T.$
\end{lem}

\begin{proof}
Set $s:=\h_M(\sum\limits_{\fp\in T'}\fp)$. We have $s\leq d-1$,
otherwise $\Assh(R/\sum\limits_{\fp\in T'}\fp)= \{\fm\}$ which
contradicts the condition  $\underset{\fp\in T}{\bigcap} \fp
\nsubseteq \underset
   {\fq\in \Assh(R/\sum\limits_{\fp\in T'}\fp)}{\bigcap} \fq$.

As $R$ is catenary, we have $\dim(R/\sum\limits_{\fp\in T'}\fp)=n-s$.
We first prove, by induction on $j$, $0\leq j\leq d-s-1$, that there
exists a chain of prime ideals

\centerline{$Q_0 \subset Q_1 \subset \cdots \subset Q_j \subset
\fm,$} \noindent such that $Q_0\in\Assh(R/\sum\limits_{\fp\in
T'}\fp)$, $\dim R/Q_j=d-s-j$ and $\underset{\fp\in
T}{\bigcap}\fp\nsubseteq Q_j$.

There is $Q_0\in\Assh(R/\sum\limits_{\fp\in T'}\fp)$ such that
$\underset{\fp\in T}{\bigcap}\fp\nsubseteq Q_0$. Note that $\dim
R/Q_0=\dim(R/\sum\limits_{\fp\in T'}\fp)=d-s$. Now, assume that
$0<j\leq d-s-1$ and that we have proved the existence of a chain
$Q_0 \subset Q_1 \subset \cdots \subset Q_{j-1}$ of prime ideals
such that $Q_0\in\Assh(R/\sum\limits_{\fp\in T'}\fp)$, $\dim
R/Q_j=d-s-(j-1)$ and that $\underset{\fp\in T}{\bigcap}\fp\nsubseteq
Q_{j-1}$.
\\
Since $d-s-(j-1)=d-s+1-j\geq 2$, the set $V$ defined as
\\
$$\begin{array}{ll}
V= \{\fq\in \Supp M |&  Q_{j-1}\subset \fq\subset \fq'\subseteq
 \fm, \dim R/\fq=d-s-j,\\ &
\fq'\in\Spec R \, \mbox{and}\,  \dim R/\fq'=d-s-j-1\}
\end{array}$$
\\
is non--empty and so, by Ratliff's weak existence theorem
\cite[Theorem 31.2]{Ma}, is not finite. As $\underset{\fp\in
T}{\bigcap}\fp \nsubseteq Q_{j-1}$, we have $Q_{j-1}\subset
Q_{j-1}+\underset{\fp\in T}{\bigcap}\fp$. If, for $\fq\in V$,
$\underset{\fp\in T}{\bigcap}\fp\subseteq \fq$, then $\fq$ is a
minimal prime of $Q_{j-1}+\underset{\fp\in T}{\bigcap}\fp$. As $V$
is an infinite set, there is $Q_j\in V$ such that $\underset{\fp\in
T}{\bigcap}\fp\nsubseteq Q_j$. Thus the induction is complete. Now
by taking $Q:=Q_{d-s-1}$ and by Proposition \ref{321}, the claim
follows.
\end{proof}

\begin{cor}\label{326}
Assume that  $R$ is complete and  $|T|=|\emph{\Assh}M|-1$. Then
there is an ideal $\fa$ of $R$ such that
$\emph{\Att\lc}^d_\fa(M)=T$.
\end{cor}\index{attached prime, $\Att$} \index{local cohomology!-module}

\begin{proof}
Note that $\Assh M\setminus T$ is a singleton set $\{\fq\}$, say,
and so $\h_M\fq=0$ and $\underset{\fp\in T}{\bigcap}\fp\nsubseteq
\fq$. Therefore the result follows  by Lemma \ref{325}.
\end{proof}

\begin{lem}\label{327}
Let  $\fa_1$ and $\fa_2$ be ideals of a complete ring  $R$. Then
there exists an ideal $\fb$ of $R$ such that
$\emph{\Att}\emph{\lc}^d_\fb(M)=\emph{\Att}\emph{\lc}^d_{\fa_{1}}(M)\cap\emph{\Att}\emph{\lc}^d_{\fa_{2}}(M)$.
\end{lem}

\begin{proof}
Set $T_{1}=\emph{\Att}\lc^d_{\fa_{1}}(M)$ and
$T_{2}=\Att\lc^d_{\fa_{2}}(M)$. We may assume that $T_1\bigcap T_2$
is a non-empty proper subset of $\Assh M$. Assume that $\fq \in
\Assh M\setminus (T_1\bigcap T_2)=(\Assh M\setminus
T_1)\bigcup(\Assh M\setminus T_2) $. By Proposition \ref{321}, there
exists $Q\in \Supp M$ with $\dim R/Q=1$ such that $\fq\subseteq Q$
and $\bigcap_{\fp\in T_1\bigcap T_2}\fp \nsubseteq Q$. Now, by
Proposition \ref{321}, again there exists an ideal $\fb$ of $R$ such
that $\Att \lc^d_\fb(M)=T_1\bigcap T_2$.
\end{proof}

\indent Now we are ready to present our main result.

\begin{thm}\label{328}
Assume that $R$ is complete and  $T\subseteq \emph{\Assh} M$, then
there exists an ideal $\fa$ of $R$ such that
$T=\emph{\Att}\emph{\lc}^d_\fa(M)$.
\end{thm}

\begin{proof}
By Corollary \ref{323}, we may assume that $\dim M\geq 2$ and that
$T$ is a non-empty proper subset of $\Assh M$. Set
$T=\{\fp_1,\ldots,\fp_t\}$ and $\Assh M\setminus
T=\{\fp_{t+1},\ldots,\fp_{t+r}\}$. We use induction on $r$. For
$r=1$, Corollary \ref{326} proves the first step of induction.
Assume that $r>1$ and that the case $r-1$ is proved. Set
$T_1=\{\fp_1,\ldots,\fp_t,\fp_{t+1}\}$ and
$T_2=\{\fp_1,\ldots,\fp_t,\fp_{t+2}\}$. By induction assumption
there exist ideals $\fa_1$ and $\fa_2$ of $R$ such that $T_1=\Att
\lc^d_{\fa_1}(M)$ and $T_2=\Att \lc^d_{\fa_2}(M)$. Now by the Lemma
\ref{327}, there exists an ideal $\fa$ of $R$ such that
$T=T_1\bigcap T_2=\Att \lc^d_{\fa}(M)$.
\end{proof}

\begin{cor}\label{329}
\emph{(See \cite[Corollary 1.7]{C})} Assume that $R$ is complete.
Then the number of non--isomorphic top local cohomology modules of
$M$ with respect to all ideals of $R$ is equal to
$2^{|\emph{\Assh}M|}$.
\end{cor}
\begin{proof}
It follows from Theorem \ref{328} and Theorem \ref{dy2}.
\end{proof}


\section{Applications to  generalized Cohen-Macaulay  modules}

\index{generalized Cohen-Macaulay, g.CM}\index{Cohen-Macaulay!generalized-, g.CM}

\quad In this section we study some properties of a generalized
Cohen-Macaulay modules in terms of certain prime ideals $\fp$ which
$R/\fp$ has a uniform local cohomological annihilator and we give a
new characterization of these modules.


\begin{lem}\label{1.5}
Let $(R,\fm)$ be a g.CM local ring. Then $R/\fp$ has a uniform local
cohomological annihilator for all $\fp\in\emph\Spec R$. In
particular, any equidimensional $R$--module $M$ has a uniform local
cohomological annihilator.
\end{lem}
\begin{proof}
Note that if $\dim R=0$, then there is nothing to prove. So we
assume that $\dim R>0$. Let $\fp\in\Spec R$ with $\h \,\fp=0$. As
$R$ is g.CM and  $\fm\nsubseteq \cup_{\fp\in \Min R}\fp$,  $R$ has a
uniform local cohomological annihilator and thus $R/\fp$ has a
uniform local cohomological annihilator by Theorem \ref{z3.2}.
Assume that $\h_M\fp=t>0$. There is a subset of system of parameters
$x_1,\ldots,x_t$ of $R$ contained in $\fp$. By Theorem
\ref{gcm}(iii), $R/(x_1,\ldots,x_t)$ is g.CM and so it has a uniform
local cohomological annihilator. In particular $R/\fp$ has a uniform
local cohomological annihilator by Theorem \ref{z3.2}. The final
part  follows  immediately  from the first part and Proposition
\ref{ulc}.
\end{proof}

The converse of the above result is not true in general, but we may
get a positive answer in a special case where $M_\fp$ is a
Cohen-macaulay $R_\fp$--module for all $\fp\in\Supp
M\setminus\{\fm\}$. We need the following lemma which is
straightforward and we give a proof for completeness.

\begin{lem}\label{eq}
Assume that $R$ is a noetherian local ring. Then
 \begin{itemize}
 \item[\emph{(a)}]If
  $\fQ\in\emph\Min \widehat{M}$, then
  $\fQ\in\emph\Min \widehat{A}/\fQ^{ce}$.
\item[\emph{(b)}] If $R$ is universally catenary and $M$ is equidimensional, then
$\widehat{M}$ is equidimensional as $\widehat{R}$--module.
\end{itemize}
\end{lem}
\begin{proof}
(a). It follows that  $\fQ^c\in\Min M$ by the Going Down Theorem.
Assume that $\fQ'\in\Min \widehat{R}/\fQ^{ce}$ such that
$\fQ'\subseteq\fQ$. The exact sequence $0\longrightarrow
R/\fQ^c\longrightarrow M$ implies the exact sequence
$0\longrightarrow \widehat{R}/\fQ^{ce}\longrightarrow \widehat{M}$.
Therefore $\fQ'\in\Ass \widehat{M}$ and so $\fQ'=\fQ$.

(b). Assume that $\fQ\in\Min\widehat{M}$. By the Going Down Theorem,
$\fQ^c\in\Min M$ from which we have $\dim \widehat{A}/\fQ^{ce}= \dim
R/\fQ^c= \dim M$. As, by part (a), $\fQ\in\Min \widehat{R}/\fQ^{ce}$
and using the fact that $R/\fQ^c$ is formally equidimensional, we
get $\dim \widehat{R}/\fQ=\dim {\widehat{R}/\fQ^{ce}}$ which implies
that $\dim \widehat{R}/\fQ=\dim M$.
\end{proof}


\begin{lem}\label{3.8}
Assume that $(R, \fm)$ is a local ring such that $R/\fp$ has a
uniform local cohomological annihilator for all $\fp\in\emph\Spec
R$. Then  the following statements are
equivalent.

 \begin{itemize}
 \item[\emph{(i)}] $M$ is  equidimensional $R$--module and for all $\fp\in\Supp
M\setminus\{\fm\}$, $M_\fp$ is a Cohen-macaulay $R_\fp$--module.
 \item[\emph{(ii)}] $M$ is a g.CM module.
 \end{itemize}
\end{lem}
\begin{proof}
 (i)$\Rightarrow$ (ii). Since
$\H^i_\fm(M)\cong\H^i_\fm(M/\Gamma_\fm(M))$ for $i>0$, we may assume
that $\Gamma_\fm(M)=0$ and so $\fm\notin\Ass M$. As, for each
$\fp\in \Ass M$, $M_\fp$ is Cohen-Macaulay, so $\Ass M=\Min M$.
 Thus $\mH^{-1}_M=0$, by Lemma \ref{1310}. Since
$M$ is equidimensional and $R/\fp$ has a uniform local cohomological
annihilator for all $\fp\in\Min M$, $M$ has a uniform local
cohomological annihilator by Proposition \ref{ulc}, and so
$R/(0:_RM)$ is universally catenary by Corollary \ref{cat}. As a
result, considering $M$ as an $R/0:_RM$--module, Lemma \ref{eq}
implies that $\widehat{M}$ is equidimensional. Hence
$\C_{\widehat{R}}(\widehat{M})$ is finite by Corollary \ref{216}.

Now, we prove the statement by using induction on $d=\dim M$. For
$d=2$, we have, by Corollary \ref{2.6}, that
$$\Att \H^1_{\widehat{\fm}}(\widehat{M})=\{\fp\in\Ass \mH^{-1}_{\widehat{M}}
: \h_{\widehat{M}}(\fp)=1\} \cup\{\fp\in\Ass \mH^{0}_{\widehat{M}} :
\h_{\widehat{M}}(\fp)=2\}.$$ If $\fp\in\Ass \mH^{-1}_{\widehat{M}}$
with $\h_{\widehat{M}}(\fp)=1$,
 then $\fp\in\Ass \widehat{M}$ and so  $\fp^c\in\Ass M=\Min M$.
 On the other hand, since
 $\fp\in\Att \H_{\widehat{\fm}}^1(\widehat{M})$,
 $\fp^c\in\Att \H^1_\fm(M)$ by Theorem \ref{atc} which contradicts, with Lemma \ref{3.6}.
 Hence $\Att \H^1_{\widehat{\fm}}(\widehat{M})\subseteq \{\widehat{\fm}\}$. Now, Corollary \ref{111}(i) implies that
  $\H^1_{\fm}(M)\otimes_R\widehat{R}$ is finitely generated  $\widehat{R}$--module and
  it is  the first step of the induction.

 Now assume that $d>2$ and the statement holds up to $d-1$. Let $x$ be a uniform local cohomological
 annihilator of $M$.
 Since $\Min M=\Ass M$, $x$ is a nonzero divisor on $M$ by using its definition. On the other
 hand, as $R/0:_RM$ is catenary,  it is straightforward to
  see that $M/xM$
 satisfies the induction hypothesis for $d-1$. Therefore, $\H_\fm^i(M/xM)$ is finitely generated for all $i< d-1$.
  The exact sequence
 $0\longrightarrow M\overset{x}{\longrightarrow} M\longrightarrow M/xM \longrightarrow 0$ implies the
 long exact sequence
 $$\cdots\longrightarrow\H^i_\fm(M) \overset{x.}{\longrightarrow} \H^i_\fm(M)\longrightarrow
  \H^i_\fm(M/xM)\longrightarrow \H^{i+1}_\fm(M)\overset{x.}{\longrightarrow} \H^{i+1}_\fm(M)
  \longrightarrow\cdots.$$
 Since $x\H^j_\fm(M)=0$ for $j<d$, we get the exact sequence
 $$0\longrightarrow \H^i_\fm(M)\longrightarrow \H^i_\fm(M/xM)\longrightarrow \H^{i+1}_\fm(M)
 \longrightarrow0,$$ for $i=0,\ldots,d-2$. Now the result follows.

(ii)$\Rightarrow$(i) is clear by Theorem \ref{gcm}(i).
 \end{proof}

Now we can state a criterion for an equidimensional local ring to be
a g.CM ring in terms of uniform local cohomological annihilators.

\index{generalized Cohen-Macaulay, g.CM}\index{Cohen-Macaulay!generalized-, g.CM}\index{uniform local
cohomological annihilator}\index{Cohen-Macaulay}
\begin{cor}\label{3.9}
Assume that $R$ is an equidimensional noetherian local ring. The
following statements are equivalent.
\begin{itemize}
\item [\emph{(i)}]$R$ is  g.CM.
\item [\emph{(ii)}] For all $\fp\in\Spec R\setminus\{\fm\}$, $R_\fp$ is a Cohen-Macaulay ring and $R/\fp$ has a uniform local cohomological annihilator.
\end{itemize}
\end{cor}
\begin{proof}
(i)$\Rightarrow$(ii). We know that $R_\fp$ is Cohen-Macaulay for all
$\fp\in\Spec R\setminus\{\fm\}$, by Theorem \ref{gcm}(i). The rest
is the subject of Lemma \ref{1.5}.

(ii)$\Rightarrow$(i) is immediate from Lemma \ref{3.8}.
\end{proof}


The following remark is easy but we bring it here for completeness
and future reference.

\begin{rem}\label{1.4}
Assume that $(R, \fm)$ is local. \index{generalized Cohen-Macaulay, g.CM}\index{Cohen-Macaulay!generalized-, g.CM}
\begin{enumerate}
\item[\emph{(i)}] A finitely generated $R$--module $M$ is g.CM if and only if all
cohomology modules of $\C_R(M)$ are of finite lengths.

\item[\emph{(ii)}]
     A finitely generated $R$--module $M$ is quasi--Buchsbaum module if and only if
$\C_R(M)$ is finite and $\fm\mH^i_M=0$ for all $i$.
   \end{enumerate}
\end{rem}\index{quasi--Buchsbaum}
\begin{proof}
(i).  Assume that $M$ is g.CM. By Theorem \ref{gcm}(i), we have
$M_\fp$ is Cohen-Macaulay for all $\fp\in\Supp M\setminus\{\fm\}$.
So that $\Supp \mH^i_M\subseteq \{\fm\}$  and, by Lemma \ref{3.8},
the result follows. The converse is clear by Lemma \ref{3.8}

(ii). It is similar to (i).
\end{proof}

\chapter{ Cohen-Macaulay loci of modules }

\quad Throughout this chapter  $M$ is a finitely generated
$R$--module. In the case $(R,\fm)$ is local, we use  as
the completion of $M$ with respect to $\fm$. The main objective of
this chapter is to study the Cohen-Macaulay locus of a module. We
show that it is a Zariski--open subset of $\Spec R$ in certain
cases.  Our results are also related to Cohen-Macaulayness of formal
fibres over certain prime ideals.

\section{Openness of Cohen-Macaulay locus }

\quad \index{Cohen-Macaulay!-locus, $\CM(M)$|textbf}
\index{non--Cohen-Macaulay locus, non-$\CM(M)$|textbf} The
Cohen-Macaulay locus of $M$ is denoted by
\begin{center} $\CM(M):=\{\fp\in\Spec R: M_\fp$ is Cohen-Macaulay
as $R_\fp$--module$\}$.\end{center}
 Let non--$\CM(M)=\Spec R\setminus\CM(M)$. Trivially the Cohen-Macaulay
locus of a Cohen-Macaulay module is $\Spec R$ and  of a generalized
Cohen-Macaulay module $M$ over a local ring $(R,\fm)$ contains
$\Spec R\setminus\{\fm\}$ by Theorem \ref{gcm}(i). In these cases
$\CM(M)$ are Zariski--open subsets of $\Spec R$.

The objective of this section is to study the Cohen-Macaulay locus
of $M$ and find out when it is a Zariski--open subset of $\Spec R$.
We first mention a remark for future references.


\begin{rem}\label{3.1}\index{Cousin complex, $\C_R(M)$!finite-}
For an  $R$--module $M$ of finite dimension, if the Cousin complex
of $M$ is finite, then \emph{non}-$\emph{\CM}(M)=
\emph{\emph{\V}}(\underset{i}{\prod}(0:_R \mH_M^i))$
  so that $\emph{\CM}(M)$
 is open.
\end{rem}
\begin{proof}
 It is clear, by Theorem \ref{2.4S4} and Theorem \ref{3.5S1},
 that

 \centerline{$\CM(M)=\Spec(R)\setminus\underset{i\geq -1}{\cup}\Supp_R(\mH_M^i).$}
\end{proof}

As we have seen in Corollary \ref{216}, the Cousin complex of every
equidimensional module over a complete local ring is finite. So it
is natural to ask, over which rings the  openness of Cohen-Macaulay
locus property is heritable from the completion ring. As an example
of such rings one may consider the rings whose formal fibres are
Cohen-Macaulay. To see this, recall first the standard dimension and
depth formula.

\begin{lem}\emph{\cite[Chapitre IV, (6.1.2),
(6.1.3),(6.3.3)]{G}}\label{standard} Assume that $(R,\fm)$ and
$(S,\fn)$ are local rings with $k=R/\fm$ and  $f:R\longrightarrow S$
is a flat local  homomorphism. For a finitely generated $R$--module
$M$, $M\otimes_R S$ is a finitely generated $S$--module and the
following statements hold true.

\emph{(i)} $\emph{\dim} M\otimes_R S=\emph{\dim} M+\emph{\dim}
S\otimes_Rk$;

\emph{(ii)} $\emph{\depth} M\otimes_R S=\emph{\depth}
M+\emph{\depth} S\otimes_Rk$;

\emph{(iii)} $M\otimes_R S$ is Cohen-Macaulay $S$--module if and
only if $M$ is a Cohen-Macaulay $R$--module and $S\otimes_Rk$ is a
Cohen-Macaulay ring.
\end{lem}

As an immediate corollary of the above lemma, we have the following
result.

\begin{cor}\label{ffcm}\index{Cohen-Macaulay}
Assume that $\fP\in\emph{\Supp} \widehat{M}$ and  $\fp=R\cap\fP$. If
the formal fibre $\widehat{R}\otimes_Rk(\fp)$ over  $\fp$ is
Cohen-Macaulay, then $M_\fp$ is Cohen-Macaulay if  and only if
$\widehat{M}_\fP$ is Cohen-Macaulay.
\end{cor}
\begin{proof}
Note that the natural local homomorphism $R_\fp\longrightarrow
\widehat{R}_\fP$ is flat, so that the result follows by the above
lemma.
\end{proof}

Now we may easily see that the openness of Cohen-Macaulay locus
property is heritable from the completion ring, if all formal fibres
are Cohen-Macaulay.
\index{Cohen-Macaulay}\index{Cohen-Macaulay!-locus, $\CM(M)$}\index{formal fibre}
\begin{cor}\label{414}
Assume that all formal fibres of $R$ are Cohen-Macaulay. If the
Cohen-Macaulay locus of $\widehat{M}$ is a Zariski--open subset of
$\emph\Spec \widehat{R}$, then the Cohen-Macaulay locus of ${M}$ is
a Zariski--open subset of $\emph\Spec R$.
\end{cor}
\begin{proof}\index{non--Cohen-Macaulay locus, non-$\CM(M)$}
Equivalently, we prove that  $\Min($non--$\CM(M))$ is a finite set.
Choose $\fp\in\Min($non--$\CM(M))$ and let $\fQ$ be a minimal member
of the non--empty set \begin{center}$\{\fq\in\Supp \widehat{M} :
\fq\cap R=\fp\}.$\end{center} Since the formal fibre of $R$ over
$\fp$ is Cohen--Macaulay, $\widehat{M}_\fQ$ is not Cohen-Macaulay by
Corollary \ref{ffcm}. On the other hand, for each $\fq\in\Supp
\widehat{M}$ with $\fq\subset \fQ$ we have $\fq\cap R\subset \fp$
and so $\widehat{M}_{\fq}$ is Cohen-Macaulay again by Corollary
\ref{ffcm}. Hence $\fQ\in\Min($non--$\CM(\widehat{M}))$ which is a
finite set.
\end{proof}

The following lemma shows that the Cohen-Macaulay locus of $M$ is
open if it is true for certain submodules of $M$.

\begin{lem}\label{3.2}\index{non--Cohen-Macaulay locus, non-$\CM(M)$}
Let \\ {\small $S=\{T\subseteq\emph\Min M$: there exists
$\fq\in\emph\Supp M$ such that $\emph\h (\fq/\fp)$ is constant for
all $\fp\in T\}$.} For each $T\in S$, we assign a submodule $M^T$ of
$M$ with $\emph\Ass M^T= T$ and $\emph\Ass M/M^T= \emph\Ass
M\setminus T$. Then
$$\emph\CM(M)= \underset{T\in
S}{\bigcup}(\emph\CM(M^T)\setminus\underset{\fp\in\emph\Ass
M\setminus T}{\cup}\V(\fp)).$$
\end{lem}
\begin{proof} For each $T\in S$, there exists a submodule $M^T$ of
$M$ with $\Ass M^T= T$ and $\Ass M/M^T=\Ass M\setminus T$ (c.f.
\cite[Page 263, Proposition 4]{B}). it is clear that
$$\Supp M/M^T=
 \underset{\fp\in\Ass M\setminus
T}{\cup}\V(\fp).$$ Let $\fq\in\CM(M)$ and set $T':= \{\fP\cap A:
\fP\in\Ass M_\fq\}$.
 As $M_\fq$ is Cohen-Macaulay, $\h(\fq/\fp)= \dim M_\fq$ for all $\fp\in T'$ and so
  $T'\in S$. We claim that $\fq\not\in\Supp M/M^{T'}$. Assuming
  contrary, there is $\fp\in\Ass M/M^{T'}$ such that $\fp\subseteq
  \fq$. Hence $\fp A_\fq\in\Ass M_\fq$ which implies that
   $\fp\in\ T'$. This contradicts with the fact that $\Ass M/M^{T'}= \Ass M \setminus T'$.
 Therefore from the exact sequence
 \begin{equation}\label{e4.11}
 0\longrightarrow M^{T'}\longrightarrow
 M\longrightarrow M/M^{T'}\longrightarrow 0
 \end{equation}
  we get $(M^{T'})_\fq\cong M_\fq$
  so that $\fq\in\CM(M^{T'})$.

Conversely, assume that $T\in S$ and that
$\fq\in\CM(M^T)\setminus\underset{\fp\in\Ass M\setminus
T}{\cup}\V(\fp)$. That is $(M^T)_\fq$ is Cohen-Macaulay and
$\fq\not\in\Supp M/M^T$. Therefore $M_\fq$ is
   Cohen-Macaulay by (\ref{e4.11}), replacing $T$ by $T'$.
\end{proof}

Note that if $R/0:_R M$ is catenary, then each module $M^T$ in the
above lemma is an equidimensional $R$--module. Therefore one can
state the following remark. \index{equidimensional}
\begin{rem}\label{3.3} \emph{ If $R$ is catenary and $\CM(N)$ is open for all equidimensional
submodules $N$ of $M$, then $\CM(M)$ is open.}
\end{rem}\index{Cohen-Macaulay!-locus, $\CM(M)$}

It is now a routine check to see that, over a local ring $R$ with
Cohen-Macaulay formal fibres, the Cohen-Macaulay locus of any
finitely generated $R$--module is open.

\begin{rem}\label{3.7}\index{formal fibre} \index{Cohen-Macaulay!-locus, $\CM(M)$}
 Assume that all formal fibres of $R$ are Cohen-Macaulay.  Then the Cohen-Macaulay locus of  $M$ is
a Zariski--open subset of  $\emph\Spec R$.
\end{rem}
\begin{proof}
By Corollary \ref{414},  it is enough to show that
$\CM(\widehat{M})$ is a Zariski--open subset of  $\Spec
\widehat{R}$. As $\widehat{R}$ is catenary, we may assume that
$\widehat{M}$ is equidimensional $\widehat{R}$--module by Remark
\ref{3.3}. Finally, Corollary \ref{216} implies that
$\mathcal{C}_{\widehat{A}}(\widehat{M})$ has finite cohomologies and
so $\CM(\widehat{M})$ is open by Remark \ref{3.1}.
\end{proof}
\index{Cousin complex, $\C_R(M)$!finite-}

We are now able to prove that any minimal element of non-$\CM(M)$ is
either an attached prime of $\H_\fm^i(M)$ for some $i$ or $R_\fp$ is
not a Cohen-Macaulay ring.
\begin{thm}\label{3.10}\index{attached
prime, $\Att$}\index{equidimensional}\index{non--Cohen-Macaulay
locus, non-$\CM(M)$}  Assume that $(R,\fm)$ is a catenary local
ring and that $M$ is equidimensional $R$--module. Then

\centerline{$\emph\Min($\emph{non}--$\emph\CM(M))\subseteq
\underset{0\leq i\leq\emph\dim M}{\cup}\emph\Att
\emph\H^i_\fm(M)\cup$\emph{non}--$\emph\CM(R)$.}
\end{thm}
\begin{proof}

Choose $\fp\in\Min($non--$\CM(M))$. As $R$ is catenary and $M$ is
equidimensional, $M_\fp$ is also equidimensional as $R_\fp$--module.
Assume that $R_\fp$ is a Cohen-Macaulay ring.  For each $\fq\in\Spec
R$ with $\fq\subseteq\fp$, $R_\fp/\fq R_\fp$ has a uniform local
cohomological annihilator by Proposition \ref{233}. Therefore, by
Lemma \ref{3.8}, $M_\fp$ is a g.CM $R_\fp$--module. As $M_\fp$ is
not Cohen-Macaulay, $\H^i_{\fp R_\fp}(M_\fp)\neq 0$ for some integer
$i$, $i< \dim M_\fp$. In particular, $\H^i_{\fp R_\fp}(M_\fp)$ is a
non--zero finite length $R_\fp$--module so that $\Att \H^i_{\fp
R_\fp}(M_\fp)=\{\fp A_\fp\}$. By Weak general shifted localization
principle (Theorem \ref{wgs}), $\fp\in\Att \H^{i+t}_\fm(M)$, where
$t=\dim(A/\fp)$. Now the result follows.
\end{proof}
\index{uniform local cohomological annihilator}\index{generalized
Cohen-Macaulay, g.CM}\index{weak general shifted localization
principle}


\begin{cor}\label{419}
Assume that $(R, \fm)$ is a catenary local ring and that the
non--$\emph\CM(R)$ is a finite set.  Then the Cohen-Macaulay locus
of $M$ is open.
\end{cor}
\begin{proof} By Lemma \ref{3.2}, we may assume that $M$ is equidimensional.
Now  Theorem \ref{3.10} implies that $\Min($non--$\CM(M))$ is a
finite set. In other words  non--$\CM(M)$ is a Zariski--closed
subset of $\Spec R$.
\end{proof}

Here is an example of local rings which satisfy  the above
condition.

\begin{exam}
\emph{Consider a local ring  $R$ satisfying Serre's condition
$(\emph{S}_{d-2})$, $d:=\dim R$, such that $\mathcal{C}_R(R)$ is
finite. Then $\mH^i_R=0$ for $i\leq d-4$ and $i\geq d-1$, by Theorem
\ref{4.4SSc} Lemma \ref{1310}(ii). Now, $\dim \mH^{d-3}_R\leq 1$
and $\dim^{d-2}_R\leq 0$, by Theorem \ref{2.7S1}(ii). Thus
non--$\CM(R)=\Supp \mH^{d-2}_R\cup\Supp \mH^{d-1}_R$ is a finite
set.}
\end{exam}

\index{non--Cohen-Macaulay locus, non-$\CM(M)$}
\index{Cohen-Macaulay!-locus, $\CM(M)$} \index{formal fibre}
\index{Cohen-Macaulay} By Remark \ref{3.7}, for a local ring $R$, if
all formal fibres of $R$ are Cohen-Macaulay, then the Cohen-Macaulay
locus of each finitely generated $R$--module is open and in
Corollary \ref{419}, we have seen  that the same statement holds if
$R$ is a catenary local ring and that the non--$\emph\CM(R)$ is a
finite set. The following examples, show that these two conditions
are significant.

Example \ref{3.15} gives a local ring $S$ with Cohen-Macaulay formal
fibres for which  the set non--$\CM(S)$ is infinite. Example
\ref{3.16} presents a local ring $T$ which admits a
non--Cohen-Macaulay formal fibre with finite non--$\CM(T)$.

\begin{exam}\label{3.15} \emph{Set $S= k[[X, Y, Z, U, V]]/(X)\cap(Y, Z)$, where $k$ is a
field. It is clear that $S$ is a local ring with Cohen-Macaulay
formal fibres. By Ratliff's weak existence theorem \cite[Theorem
31.2]{Ma}, there are infinitely many prime ideals $P$ of $k[[X, Y,
Z, U, V]]$, with $(X, Y, Z)\subset P\subset (X, Y, Z, U, V)$. For
any such prime ideal $P$, $S_{\overline{P}}$ is not equidimensional
and so it is not Cohen-Macaulay. In other words, non--$\CM(S)$ is
infinite.}
\end{exam}
\begin{exam}\label{3.16}
\emph{It is shown in \cite[Proposition 3.3]{FR} that there exists a
local integral domain $(R, \fm)$ of dimension 2 such that
$\widehat{R}= \mathbb{C}[[X, Y, Z]]/(Z^2, tZ)$, where $\mathbb{C}$
is the field of complex numbers and $t= X+ Y+ Y^2s$ for some
$s\in\mathbb{C}[[Y]]\setminus\mathbb{C}\{Y\}$. As $\Ass \widehat{R}=
\{(Z), (Z, t)\}$, $\widehat{R}$ does not satisfy $(S_1)$. Thus
$\mH^{-1}_{\widehat{R}} \neq 0$ while $\mH^{-1}_R=0$, by Theorem
\ref{4.4SSc}. Now Lemma \ref{Pet} implies that there exists a formal
fibre of $R$ which is not Cohen-Macaulay. As $R$ is an integral
local domain, we have non--$\CM(R)= \{\fm\}$.}
\end{exam}

\section{Rings whose formal fibres are Cohen-Macaulay}

\quad It is shown in Corollary \ref{414} that if all formal fibres
of $R$ are Cohen-Macaulay, then  the Cohen-Macaulay locus of any
finitely generated $R$--module $M$ is a Zariski--open subset of
$\Spec R$. This result motivates us to determine rings whose formal
fibres are Cohen-Macaulay. More precisely, we study the affect of
certain formal fibres being Cohen-Macaulay on the structure of a
module.

Throughout this section $(R, \fm)$ is a local ring and $M$ is a
finitely generated  $R$--module of dimension $d$.

We begin with the following result which is the heart of the proof of
our main result Theorem \ref{B}
 \begin{prop}\label{4.21}\index{formal fibre}
 \index{uniform local cohomological annihilator}
Assume that $\fp$ is a prime ideal of $\emph{\Spec} R$ such that
$R/\fp$ has a uniform local cohomological annihilator. Then the
formal fibre of $R$ over $\fp$ is Cohen-Macaulay.
\end{prop}
\begin{proof}
 It is well known that $(R_\fp/\fp R_\fp)\otimes_R \widehat{R}\cong
S^{-1}(\widehat{R}/\fp \widehat{R})$, where $S$ is the image of
$R\setminus \fp$ in $\widehat{R}$. Therefore we want to show that
$(S^{-1}(\widehat{R}/\fp \widehat{R}))_{S^{-1}\fq}$ is
Cohen-Macaulay for all $\fq\in \Spec \widehat{R}$ with $S\cap\fq=
\emptyset$. It is enough to show that $(\widehat{R}/\fp
\widehat{R})_\fq$ is Cohen-Macaulay $\widehat{R}_\fq$--module. Since
$R/\fp $ has a uniform local cohomological annihilator,
$\widehat{R}/\fp \widehat{R}$ has a uniform local cohomological
annihilator which, in particular, implies that $\widehat{R}/\fp
\widehat{R}$ is equidimensional by Proposition \ref{equ}. Assume,
contrarily, $(\widehat{R}/\fp \widehat{R})_\fq$ is not
Cohen-Macaulay. We may assume that $\fq\in\Min(\nCM(\widehat{R}/\fp
\widehat{R}))$ and that $(\fq\cap R)\cap(R\setminus\fp)= \emptyset$.
In other words, $\nCM((\widehat{R}/\fp \widehat{R})_\fq)= \{\fq
\widehat{R}_\fq\}$ and $\fq\cap R= \fp$.

Let us replace $R$ and $M$ in Lemma \ref{3.3} by $\widehat{R}_\fq$
and $(\widehat{R}/\fp \widehat{R})_\fq$, respectively. Note that
$\widehat{R}/\fq'$ is equidimensional for all $\fq'\in\Spec
\widehat{R}$, so that $\mathcal{C}_{\widehat{R}}(\widehat{R}/\fq')$
is finite by Proposition \ref{233} and thus
$\mathcal{C}_{\widehat{R}_\fq}(\widehat{R}_\fq/\fq'\widehat{R}_\fq)$
is finite. Therefore $\widehat{R}_\fq/\fq'\widehat{R}_\fq$ has a
uniform local cohomological annihilator as $\widehat{R}_\fq$--module
 for all $\fq'\in\Spec \widehat{R}$ with $\fq'\subseteq\fq$ by Theorem \ref{229}. As $(\widehat{R}/\fp
\widehat{R})_\fq$ is equidimensional, we can apply Lemma \ref{3.8}
to deduce that $(\widehat{R}/\fp \widehat{R})_\fq$ is  g.CM as
$\widehat{R}_\fq$--module. In particular,
$\H_{\fq\widehat{R}_\fq}^i((\widehat{R}/\fp \widehat{R})_\fq))$ is a
non--zero  $\widehat{R}_\fq$--module of finite length for some
$i<\dim(\widehat{R}/\fp \widehat{R})_\fq$ for which we get $\Att
\H_{\fq\widehat{R}_\fq}^i((\widehat{R}/\fp \widehat{R})_\fq)=
\{\fq\widehat{R}_\fq\}$. Now, the weak general shifted localization
principle (Theorem \ref{wgs}) \index{weak general shifted
localization principle} implies that $\fq\in\Att
\H_{\widehat{\fm}}^j(\widehat{R}/\fp\widehat{R})$ for some $j< \dim
R/\fp$ which gives $\fp= \fq\cap R\in\Att \H_\fm^j(R/\fp)$. This
contradicts with Lemma \ref{3.6}.
\end{proof}

The above proposition enables us to give a characterization of a
finitely generated module which admits a uniform local cohomological
annihilator in terms of a certain set of formal fibres of the ground
ring.

\begin{thm}\label{B}\index{formal fibre}\index{uniform local cohomological
annihilator}
 The following statements
are equivalent.
\begin{itemize}
\item[\emph{(i)}] $\widehat{M}$ is equidimensional $\widehat{R}$--module
and all formal fibres of $R$ over minimal members of $\emph\Supp M$
are Cohen-Macaulay.
\item[\emph{(ii)}] $M$ has a uniform local cohomological annihilator.
\end{itemize}
\end{thm}
\begin{proof}\index{Cohen-Macaulay!-locus, $\CM(M)$}
\index{non--Cohen-Macaulay locus, non-$\CM(M)$}
(i)$\Rightarrow$(ii). By Proposition \ref{233},
$\mathcal{C}_{\widehat{R}}(\widehat{M})$ is finite, which implies
that the Cohen-Macaulay locus of $M$ is open or equivalently
$\Min($non-$\CM(\widehat{M}))$ is a finite set (see Remark
\ref{3.1}). Thus Lemma \ref{139} implies that
\begin{equation}
\underset{\fq\in \nCM(\widehat{M})} {\cap}\fq \not\subseteq
\underset{\fq\in\Min \widehat{M}} {\cup}\fq.
\end{equation}
Note that for an element $r\in(\underset{\fq\in \nCM(\widehat{M})}
{\cap}\fq)\cap R$, Corollary \ref{c224} implies that
$r^n\H^i_{\widehat{\fm}}(\widehat{M})= 0$ for some positive integer
$n$ and for all $0\leq i< \dim M$. On the other hand
 $\H_{\widehat{\fm}}^i(\widehat{M})\cong\H_\fm^i(M)$, so  it is
enough to show that
\begin{equation}\label{sub}
(\underset{\fq\in \nCM(\widehat{M})} {\cap}\fq) \cap R\not\subseteq
\underset{\fp\in\Min M} {\cup}\fp.\end{equation}

Assume contrarily that (\ref{sub}) does not hold. Then since
$\Min(\nCM(\widehat{M}))$ is a finite set, there is $\fp\in\Min M$
such that $\fp= \fq\cap R$ for some $\fq\in \nCM(\widehat{M})$. Note
that $R_\fp\longrightarrow (\widehat{R})_\fq$ is a faithfully flat
ring homomorphism and its fibre ring over $\fp R_\fp$ is
$((R_\fp/\fp R_\fp) \otimes_R \widehat{R})_\fq$ which is
Cohen-Macaulay by our assumption. Therefore, by Corollary
\ref{ffcm}, $M_\fp$ is not Cohen-Macaulay. This contradicts with the
fact that $\dim_{R_\fp}(M_{\fp})= 0$.

(ii)$\Rightarrow$(i). As $M$ has a uniform local cohomological
annihilator, $M$ is equidimensional and $R/\fp$ has uniform local
cohomological annihilator for all minimal prime $\fp$ of $M$, by
Proposition \ref{ulc}. Thus the formal fibre over $\fp$ is
Cohen-Macaulay for $\fp\in\Min M$ by Proposition \ref{4.21}.
\end{proof}

The following result shows that if all formal fibres of a ring $R$
are Cohen-Macaulay, then $R$ is universally catenary if and only if
$\mathcal{C}_R(R/\fp)$ is finite for all $\fp\in\Spec R$.
\begin{cor} \label{C}\index{uniform local cohomological
annihilator}

\index{Cousin complex, $\C_R(M)$!finite-} \index{universally
catenary}
The following statements are equivalent. \index{uniform local
cohomological annihilator}\index{formal fibre}
\begin{itemize}
\item[\emph{(i)}] $R$ is universally catenary  ring and all of its formal fibres are Cohen-Macaulay.
\item[\emph{(ii)}] The Cousin complex $\mathcal{C}_{R}(R/\fp)$ is finite for all $\fp\in\emph\Spec R$.
\item[\emph{(iii)}]  $R/\fp$ has
a uniform local cohomological annihilator for all $\fp\in\emph\Spec
R$.
\end{itemize}
\end{cor}
\begin{proof}
(i)$\Rightarrow$(ii) is clear by Proposition \ref{233}.

(ii)$\Rightarrow$(iii) is clear by Theorem \ref{229}.

(iii)$\Rightarrow$(i). By Corollary \ref{cat}, $R/\fp$ is
universally catenary for all  primes $\fp$ and so is $R$.  The rest
is clear by Proposition \ref{4.21}.

\end{proof}

As we have seen in Lemma \ref{228},  non--$\CM(M)\subseteq
\V(\mathrm{a}(M))$. Our final goal is to investigate that when the
equality holds. In 1982,  Schenzel proved that the equality holds
when $M$ is equidimensional and $R$ posses a dualizing
complex\cite[p. 52]{Sc1}. Recall that if a local  $R$ posses a
dualizing complex \index{dualizing complex} then it is a homomorphic
image of a Gorenstein local ring by \cite[Theorem 6.1]{K1}  and so
$R$ is universally catenary and all formal fibres of $R$ are
Cohen-Macaulay. Hence $\C_R(M)$ is finite for all equidimensional
$R$--module $M$ by Corollary \ref{216}  and Lemma \ref{31.7Ma}.

In the following,  we see that non--$\CM(M)= \V(\mathrm{a}(M))$
whenever $\mathcal{C}_R(M)$ is finite.

\begin{cor}\label{F}\index{Cousin complex, $\C_R(M)$!finite-}
\index{non--Cohen-Macaulay locus, non-$\CM(M)$} \index{a@$\a(M)$}
Assume that $M$ is a finitely generated $R$--module of dimension $d$
and that $\mathcal{C}_R(M)$ is finite.  Then

\centerline{$\V(\prod_{i= -1}^{d-1}(0 :_R
\mathcal{H}_M^i))=$\emph{non}--$\emph{\CM}(M)=\V(\mathrm{a}(M))$.}
\end{cor}
\begin{proof}
The first equality is in Remark \ref{3.1}. The second inequality is
clear by Lemma \ref{228} and Corollary \ref{c224}.
\end{proof}

\index{formal fibre}\index{Cohen-Macaulay} As we have seen in
Proposition \ref{4.21}, if $R/\fp$ has a uniform local cohomological
annihilator, then the formal fibre of $R$ over $\fp$ is
Cohen-Macaulay. So to find the Cohen-Macaulay formal fibres, it is
useful to study  those prime ideals $\fp$ such that $R/\fp$ has a
uniform local cohomological annihilator. Characterizing these ideals
enables also us to characterize those modules $M$ with non--$\CM(M)=
\V(\mathrm{a}(M))$ (see also Corollary \ref{C}).
\begin{prop}\label{G4}\index{uniform local cohomological
annihilator}\index{a@$\a(M)$}
Assume that $\fp\in\emph\Spec R$. A necessary and sufficient
condition for  $R/\fp$ to have a uniform local cohomological
annihilator is that there exists an equidimensional $R$--module $M$
such that
 $\fp\in\emph\Supp M\setminus\V(\mathrm{a}(M))$.
\end{prop}
\begin{proof}
The necessary condition is clear by taking $M:=R/\fp$. For the
converse, assume that there is an equidimensional $R$--module $M$
such that $\fp\in\Supp M\setminus\V(\mathrm{a}(M))$. We prove the
claim by induction on $h:=\h_M\fp$. When $h=0$, we have $\fp\in\Min
M$. Choose a submodule $N$ of $M$ with $\Ass N= \{\fp\}$ and $\Ass
M/N=\Ass M\setminus\{\fp\}$. It is clear that $(M/N)_\fp =0$ so that
$r(M/N)= 0$ for some $r\in R\setminus\fp$. On the other hand the
fact that $\mathrm{a}(M)\not\subseteq\fp$ implies that there is
$s\in R\setminus\fp$ such that $s\H_\fm^i(M)= 0$ for all $i<\dim M$.
The exact sequence
$\H_\fm^{i-1}(M/N)\longrightarrow\H_\fm^i(N)\longrightarrow\H_\fm^i(M)$
implies $rs\H_\fm^i(N)= 0$ for all $i<\dim N$. As $\fp\in\Min N$,
$R/\fp$ has a uniform local cohomological annihilator by Proposition
\ref{ulc}.

Now assume that $h>0$. For any $\fq\in\Supp M$ with $\fq\subset\fp$
we have $\fq\not\in\V(\mathrm{a}(M))$ so that $R/\fq$ has a uniform
local annihilator by induction hypothesis. As
$\fp\not\in\V(\mathrm{a}(M))$, $M_\fp$ is Cohen-Macaulay, by
Corollary  \ref{F}. Choose a submodule $K$ of $M$ with $\Ass K= \Min
M$ and $\Ass M/K= \Ass M\setminus\Min M $. If $\fp\in\Supp M/K$ then
there is $\fq\in\Ass M/K$ with $\fq\subseteq\fp$. Therefore
$\fq\in\Ass M$ and $\fq\not\in\Min M$. As $M_\fp$ is Cohen-Macaulay,
so is $M_\fq$ which gives $\fq\in\Min M$, which is a contradiction.
Hence we have $\fp\not\in\Supp M/K$ which yields $r(M/K)=0$ for some
$r\in A\setminus\fp$ and so, by applying local cohomology to the
exact sequence $0\longrightarrow K\longrightarrow M\longrightarrow
M/K\longrightarrow 0$, it follows that
$\mathrm{a}(K)\not\subseteq\fp$. As $M_\fp\cong K_\fp$, $K_\fp$ is
Cohen-Macaulay and $\h_K\fp>0$, there is $x\in\fp$ which is
non--zero--divisor on $K$. The exact sequence $0\longrightarrow
K\overset{x}{\longrightarrow} K\longrightarrow K/xK\longrightarrow
0$ implies that $\mathrm{a}(K)^2\subseteq\mathrm{a}(K/xK)$ which
implies that $\mathrm{a}(K/xK)\not\subseteq\fp$. As
$\h_{K/xK}(\fp)<h$, $R/\fp$ has a uniform local cohomological
annihilator by the induction hypothesis.
\end{proof}

As our first application of the above proposition, we have the following result.

\begin{cor}\label{4266}\index{Serre's condition $(S_n)$}\index{formal
fibre}\index{Cohen-Macaulay} Assume that $M$ is a finitely generated
$R$--module which satisfies the condition $(S_n)$. If $\C_R(M)$ is
finite, then the formal fibres of $R$ over all prime ideals
$\fp\in\emph{\Supp} M$ with $\h_M\fp\leq n$ are Cohen-Macaulay.
\end{cor}

\begin{proof}
Let $\fp\in\Supp M$ with $\h_M\fp\leq n$. Note that
$\V(\a(N))=$non-$CM(N)$ by Corollary \ref{F}. On the other hand
$M_\fp$ is Cohen-Macaulay, so that the result follows by Proposition
\ref{G4}.
\end{proof}

Now we may show that if $\C_R(M)$ is finite, then the formal fibres
of $R$ over some certain prime ideals are Cohen-Macaulay.

\begin{cor}\label{426}
Assume that $\C_R(M)$ is finite. Then the formal fibres of $R$ over
all prime ideals $\fp\in\emph{\Supp} M$ with $\emph{\h}_M\fp\leq 1$
are Cohen-Macaulay.
\end{cor}
\begin{proof}
  By Corollary \ref{cs1}, there
exists a finitely generated $R$--module $N$ which satisfies $(S_1)$,
$\Supp N=\Supp M$ and $\C_R(N)$ is finite. Now the result follows by
Corollary \ref{4266}.
\end{proof}


In Corollary \ref{F}, it is shown that, for a finitely generated
$R$--module $M$, non--$\CM(M)= \V(\mathrm{a}(M))$ whenever
$\mathcal{C}_R(M)$ is finite. In the following we characterize those
modules $M$ satisfying  non--$\CM(M)= \V(\mathrm{a}(M))$ without
assuming that the Cousin complex of $M$ to be finite.

\begin{thm}\label{last}\index{non--Cohen-Macaulay locus, non-$\CM(M)$}\index{a@$\a(M)$}\index{uniform local
cohomological annihilator}\index{Cohen-Macaulay!-locus, $\CM(M)$}
For an equidimensional $R$--module $M$, the following statements are
equivalent.
\begin{itemize}
\item[\emph{(i)}] $R/\fq$ has a uniform local cohomological annihilator for
all $\fq\in\emph\CM(M)$.
\item[\emph{(i$'$)}] $\widehat{R}/\fq\widehat{R}$ is equidimensional
$\widehat{R}$--module and the formal fibre ring $(R_\fq /\fq
R_\fq)\otimes_R \widehat{R}$ is Cohen-Macaulay for all
$\fq\in\emph\CM(M)$.
\item[\emph{(ii)}] \emph{non}--$\emph\CM(M)= \emph{\V}(\mathrm{a}(M))$.
\item[\emph{(iii)}] \emph{non}--$\emph\CM(M)\supseteq \emph{\V}(\mathrm{a}(M))$.
\end{itemize}
\end{thm}
\begin{proof}
The equivalence of (i) and (i$'$) is the subject of Theorem \ref{B}.

(i)$\Longrightarrow$(ii). The inclusion
non-$\CM(M)\subseteq\V(\mathrm{a}(M))$ is clear by Lemma \ref{228}.

Now assume that $\fp\supseteq\mathrm{a}(M)$. Thus there is an
integer $i$, $0\leq i< d$, such that $\fp\supseteq 0:_R
\H_\fm^i(M)$. There is $\fQ\in\Att
\H_{\widehat{\fm}}^i(\widehat{M})$ with
$\fq:=R\cap\fQ\in\Att\H_\fm^i(M)$  and $\fp\supseteq\fq$. To show
$\fp\in$non--$\CM(M)$ it is enough to show that
$\fq\in$non--$\CM(M)$. Assuming contrarily, $\fq\in\CM(M)$, $R/\fq$
has a uniform local cohomological annihilator by our assumption and
so the formal fibre $k(\fq)\otimes_R \widehat{R}$ is Cohen-Macaulay,
by Theorem \ref{B}. As the map $R_\fq\longrightarrow
\widehat{R}_\fQ$ is faithfully flat ring homomorphism, we find that
$k(\fq)\otimes_{R_\fq}\widehat{R}_\fQ$ is also Cohen-Macaulay.
Therefore the standard dimension and depth formulas (Lemma
\ref{standard}), applied to the faithfully flat extension
$R_\fq\longrightarrow \widehat{R}_\fQ$, implies that
$\widehat{M}_\fQ$ is Cohen-Macaulay. On the other hand, $R/\fr$ has
a uniform local cohomological annihilator for all $\fr\in\Min M$
(simply because in this case $M_\fr$ has zero dimension and so
$\fr\in\CM(M)$). Thus, by Proposition \ref{ulc}, $M$ has a uniform
local cohomological annihilator and so does $\widehat{M}$.
Therefore, by Proposition \ref{equ} $\widehat{M}$ is
equidimensional. Thus Corollary \ref{216} implies that the Cousin
complex $\mathcal{C}_{\widehat{R}}(\widehat{M})$ is finite. As
$\fQ\in\Att \H_{\widehat{\fm}}^i(\widehat{M})$, we have, by
Corollary \ref{F}, $\fQ\in$non--$\CM(\widehat{M})$. This is a
contradiction.

(iii)$\Longrightarrow$(i). Assume that $\fq\in\CM(M)$ so that
$\fq\not\supseteq\mathrm{a}(M)$ by our assumption. Now Proposition
\ref{G4} implies that $R/\fq$ has a uniform local cohomological
annihilator.
\end{proof}

Our last result in this section is a  consequence of the above
theorem, which improves Corollary \ref{426} to a generalization  of
the fact that, all formal fibres of a Cohen-Macaulay ring are
Cohen-Macaulay.

\begin{cor}\label{clast}
Assume that $\C_R(M)$ is finite. Then the formal fibres of $R$ over
all prime ideals $\fp\in\emph{\CM}(M)\cup\{\fp\in\emph{\Supp} M :
\emph{\h}_M\fp=1\}$ are Cohen-Macaulay.
\end{cor}
\begin{proof}
By  Corollary \ref{426} all formal fibres of $R$ over prime ideals
$\fp$ with $\h \fp=1$, are Cohen-Macaulay. On the other hand
{non}--$\CM(M)= {\V}(\mathrm{a}(M))$ by Corollary \ref{F}. Now,
Theorem \ref{last} completes the proof.\end{proof}


\section{Some comments}

\index{Cousin complex, $\C_R(M)$!finite-} \quad Throughout this
section $(R,\fm)$ is a local ring. As we have seen in Corollary
\ref{cor}, if $\C_R(M)$  is finite, then  $M$ is equidimensional and
$R/0:_RM$ is universally catenary.  On the other hand, Proposition
\ref{233} shows that, if all formal fibres of a universally catenary
local ring $R$ are Cohen-Macaulay, then $\C_R(M)$ is finite for all
equidimensional $R$--module $M$. These results lead us to the
following  natural question.

\index{formal fibre}\index{universally
catenary}\index{equidimensional}
\begin{ques} \label{question} Assume that  $\C_R(M)$ is finite. Are the
formal fibres of $R$ over all prime ideals $\fp\in\emph{\Supp} M$,
Cohen-Macaulay? \end{ques}

 Corollary \ref{426} shows that if $\C_R(M)$ is finite, then the
formal fibres of $R$ over all prime ideals
 $\fp\in\Supp M$ with ${\h}_M\fp\leq1$, are Cohen-Macaulay. As a consequence we may have
a positive answer for the above question in the following  special
case.

\begin{cor}\label{c1}
Assume that ${\C}_R(M)$ is finite and $\emph{\dim} M\leq 3$. Then
the formal fibres of $R$ over all prime ideals $\fp\in\emph{\Supp}
M$ are Cohen-Macaulay.
\end{cor}

\begin{proof}
Let $\fp\in\Supp M$. If $\h_M\fp\leq 1$, then the formal fibre over
$\fp$ is Cohen-Macaulay by Corollary \ref{426}. Now assume that
$\h_M\fp>1$. Thus $\dim R/\fp\leq 1$ and  so  $R/\fp$ has a uniform
local cohomological annihilator. Hence the formal fibre over $\fp$
is Cohen-Macaulay by Proposition \ref{4.21}.
\end{proof}

In Corollary \ref{clast}, we have seen that if $\C_R(M)$ is finite,
then the formal fibres of $R$ over all prime ideals
$\fp\in{\CM}(M)\cup\{\fp\in{\Supp} M : \emph{\h}_M\fp=1\}$ are
Cohen-Macaulay. Thus we have the following result in which,

\centerline{$\dim($non--$\CM(M))=\sup\{\dim R/\fp : \fp\in$
non--$\CM(M)\}$.}

\begin{cor}
Assume that $\C_R(M)$ is finite and
\emph{$\dim$(non}--$\emph{\CM}(M))\leq 1$. Then the  formal fibres
of $R$ over all prime ideals $\fp\in\emph{\Supp} M$ are
Cohen-Macaulay.
\end{cor}
\begin{proof}
Let $\fp\in\Supp M$. If $M_\fp$ is Cohen-Macaulay, the formal fibre
of $R$ over $\fp$ is Cohen-Macaulay by Corollary \ref{clast}. Now,
assume that $\fp\in$ non--$\CM(M)$. Thus $\dim R/\fp\leq 1$ and  so
that $R/\fp$ has a uniform local cohomological annihilator. Hence
the formal fibre over $\fp$ is Cohen-Macaulay by Proposition
\ref{4.21}.
\end{proof}

 One way of dealing Question \ref{question}, is to find a reduction
 technique e.g. for dimension of the module. In this connection, we
 propose the following.

\index{Cousin complex, $\C_R(M)$!finite-|textbf}
\begin{prob}\label{p1}
Assume that $\C_R(M)$ is  finite and $x$ is a non--zero--divisor on
$M$. Then $\C_R(M/xM)$ is finite.
\end{prob}

Assume that $\C_R(M)$ is finite and $\fp\in\Supp M$. Consider the
submodule $N$ of $M$ with $\Ass N=\Ass M\setminus\Assh M$ and $\Ass
M/N=\Assh M$ (c.f. \cite[Page 263, Proposition 4]{B}).
 As in the
exact sequence $0\longrightarrow N\longrightarrow M\longrightarrow
M/N\longrightarrow 0$ we have $\h_M\fp\geq 1$ for all $\fp\in\Supp
N$, Lemma \ref{211}(b) implies that $\mathcal{C}_R(M/N)$ is finite.
Note that $\Supp M=\Supp M/N$ and so $\fp\in\Supp M/N$. Since $\Ass
M/N=\Min M/N$, $\fp$ contains a non--zero--divisor $x$ on $M/N$. If
we have positive answer for the above problem, the Cousin complex
$\C_R(\frac{M/N}{xM/N})$ is finite. Now by an induction argument on
$\dim M$ and using Corollary \ref{c1} one finds  that the formal
fibre of $R$ over $\fp$ is Cohen-Macaulay.

Note that if $\C_R(M)$ is finite, then $M$ is equidimensional and
$R/0:_RM$ is universally catenary by  Corollary \ref{cor} and so
$M/xM$ is equidimensional. Now, if  the formal fibres over all
$\fp\in\Supp M$ are Cohen-Macaulay, then Proposition \ref{233}
follows that $\C_R(M/xM)$ is finite. Hence solving Problem \ref{p1}
is equivalent to find an answer for Question \ref{question}.


\bibliographystyle{amsplain}

\clearpage

\addcontentsline{toc}{chapter}{Index}
\def\baselinestretch{1}
\small \printindex
\end{document}